%% file: gysin.tex
\documentclass[10pt,twoside]{amsart}

\usepackage[english]{babel}
\usepackage[all]{xy}
\usepackage[mathscr]{euscript}
\usepackage{mathrsfs,amsmath,amssymb,bbm,stmaryrd,url}

\input{command}

\newtheorem{thm}{Theorem}[section]
\newtheorem{prop}[thm]{Proposition}
\newtheorem{cor}[thm]{Corollary}
\newtheorem{lm}[thm]{Lemma}

\theoremstyle{definition}
\newtheorem{df}[thm]{Definition}
\newtheorem{ex}[thm]{Example}
\newtheorem{num}[thm]{}

\theoremstyle{remark}
\newtheorem{rem}[thm]{Remark}

\numberwithin{equation}{section}

\def\YEAR{\year}\newcount\VOL\VOL=\YEAR\advance\VOL by-1995
\def\firstpage{1}\def\lastpage{1000}
\def\received{}\def\revised{}
\def\communicated{}

\makeatletter
\def\magnification{\afterassignment\m@g\count@}
\def\m@g{\mag=\count@\hsize6.5truein\vsize8.9truein\dimen\footins8truein}
\makeatother

\font\eightrm=cmr8
\font\caps=cmcsc10                    
\font\Caps=cmcsc10 scaled \magstep1   

\makeatletter
\@specialpagefalse\headheight=8.5pt
\def\DocMath{}
\renewcommand{\@evenhead}{%
    \ifnum\thepage>\lastpage\rlap{\thepage}\hfill%
    \else\rlap{\thepage}\slshape\leftmark\hfill{\caps\SAuthor}\hfill\fi}%
\renewcommand{\@oddhead}{%
    \ifnum\thepage=\firstpage{\DocMath\hfill\llap{\thepage}}%
    \else{\slshape\rightmark}\hfill{\caps\STitle}\hfill\llap{\thepage}\fi}%
\makeatother

\def\TSkip{\bigskip}
\newbox\TheTitle{\obeylines\gdef\GetTitle #1
\ShortTitle  #2
\SubTitle    #3
\Author      #4
\ShortAuthor #5
\EndTitle
{\setbox\TheTitle=\vbox{\baselineskip=20pt\let\par=\cr\obeylines%
\halign{\centerline{\Caps##}\cr\noalign{\medskip}\cr#1\cr}}%
	\copy\TheTitle\TSkip\TSkip%
\def\next{#2}\ifx\next\empty\gdef\STitle{#1}\else\gdef\STitle{#2}\fi%
\def\next{#3}\ifx\next\empty%
    \else\setbox\TheTitle=\vbox{\baselineskip=20pt\let\par=\cr\obeylines%
    \halign{\centerline{\caps##} #3\cr}}\copy\TheTitle\TSkip\TSkip\fi%
\centerline{\caps #4}\TSkip\TSkip%
\def\next{#5}\ifx\next\empty\gdef\SAuthor{#4}\else\gdef\SAuthor{#5}\fi%
\ifx\received\empty\relax
    \else\centerline{\eightrm Received: \received}\fi%
\ifx\revised\empty\TSkip%
    \else\centerline{\eightrm Revised: \revised}\TSkip\fi%
\ifx\communicated\empty\relax
    \else\centerline{\eightrm Communicated by \communicated}\fi\TSkip\TSkip%
\catcode'015=5}}\def\Title{\obeylines\GetTitle}
\def\Abstract{\begingroup\narrower
    \parskip=\medskipamount\parindent=0pt{\caps Abstract. }}
\def\EndAbstract{\par\endgroup\TSkip}

\long\def\MSC#1\EndMSC{\def\arg{#1}\ifx\arg\empty\relax\else
     {\par\narrower\noindent%
     2000 Mathematics Subject Classification: #1\par}\fi}

\long\def\KEY#1\EndKEY{\def\arg{#1}\ifx\arg\empty\relax\else
	{\par\narrower\noindent Keywords and Phrases: #1\par}\fi\TSkip}

\newbox\TheAdd\def\Addresses{\vfill\copy\TheAdd\vfill
    \ifodd\number\lastpage\vfill\eject\phantom{.}\vfill\eject\fi}
{\obeylines\gdef\GetAddress #1
\Address #2 
\Address #3
\Address #4
\EndAddress
{\def\xs{4.3truecm}\parindent=0pt
\setbox0=\vtop{{\obeylines\hsize=\xs#1\par}}\def\next{#2}
\ifx\next\empty 
     \setbox\TheAdd=\hbox to\hsize{\hfill\copy0\hfill}
\else\setbox1=\vtop{{\obeylines\hsize=\xs#2\par}}\def\next{#3}
\ifx\next\empty 
     \setbox\TheAdd=\hbox to\hsize{\hfill\copy0\hfill\copy1\hfill}
\else\setbox2=\vtop{{\obeylines\hsize=\xs#3\par}}\def\next{#4}
\ifx\next\empty\ 
     \setbox\TheAdd=\vtop{\hbox to\hsize{\hfill\copy0\hfill\copy1\hfill}
                \vskip20pt\hbox to\hsize{\hfill\copy2\hfill}}
\else\setbox3=\vtop{{\obeylines\hsize=\xs#4\par}}
     \setbox\TheAdd=\vtop{\hbox to\hsize{\hfill\copy0\hfill\copy1\hfill}
	        \vskip20pt\hbox to\hsize{\hfill\copy2\hfill\copy3\hfill}}
\fi\fi\fi\catcode'015=5}}\gdef\Address{\obeylines\GetAddress}

\begin{document}
\Title
Around the Gysin triangle {II}.
\ShortTitle 
\SubTitle   
\Author 
Fr\'ed\'eric D\'eglise\footnote{Partially supported by the \emph{Agence Nationale de la Recherche}, project no.\ ANR-07-BLAN-0142 ``M\'ethodes \`a la Voevodsky, motifs mixtes et G\'eom\'etrie d'Arakelov''. }
\ShortAuthor 
F. D\'eglise
\EndTitle
\Abstract 
The notions of orientation and duality are well understood 
in algebraic topology in the framework of the stable homotopy category. 
In this work, we follow these lines in algebraic geometry, 
in the framework of motivic stable homotopy, introduced by F. Morel and V. Voevodsky. 
We use an axiomatic treatment which allows us to consider 
both mixed motives and oriented spectra over an arbitrary base scheme. 
In this context, we introduce the Gysin triangle and prove several formulas 
extending the traditional panoply of results 
on algebraic cycles modulo rational equi\-valence.
We also obtain the Gysin morphism of a projective morphism
 and prove a duality theorem in the (relative) pure case.
These constructions involve certain characteristic classes 
(Chern classes, fundamental classes, cobordism classes) 
together with their usual properties.
They imply statements in motivic cohomology, algebraic K-theory 
(assuming the base is regular) 
 and "abstract" algebraic cobordism as well as the dual statements in 
  the corresponding homology theories.
They apply also to ordinary cohomology theories in algebraic geometry 
through the notion of a mixed Weil cohomology theory,
 introduced by D.-C. Cisinski and the author in \cite{CD2}, 
 notably rigid cohomology.
\EndAbstract
\MSC 
14F42, 14C17.
\EndMSC
\KEY 
Orientation, transfers, duality, characteristic classes.
\EndKEY
\Address 
CNRS, LAGA, Institut Galil\'ee, Universit\'e Paris 13
99 avenue Jean-Baptiste Cl\'ement, 93430 Villetaneuse, FRANCE
\Address
\Address
\Address
\EndAddress

%
%
%
%
%

\setcounter{tocdepth}{1}
\tableofcontents

\input{intro}

\input{axioms}
\input{chern}
\input{purity}
\input{associativity}
\input{duality}

\bibliographystyle{amsalpha}
\bibliography{gysin}

\Addresses
\end{document}

%% file: command.tex
\newcommand{\NN} {\mathbb N}
\newcommand{\ZZ} {\mathbb Z}

\renewcommand{\AA} {\mathbb A}
\newcommand{\GG} {\mathbb G_m }
\newcommand{\PP} {\mathbb P}

\newcommand{\T}{\mathscr T} 

\newcommand{\C}{\mathscr C}
\newcommand{\D}{\mathscr D}

\newcommand{\cM}{\mathcal M}
\newcommand{\E}{\mathbb E} 
\newcommand{\F}{\mathbb F} 

\newcommand{\I}{\mathcal I} 
\newcommand{\J}{\mathcal J} 

\newcommand{\un}{\mathbbm 1} 

\renewcommand{\L}{\lambda}
\newcommand{\Q}{\xi}

\newcommand{\Hs}{\mathscr H^s_\bullet(S)}
\renewcommand{\H}{\mathscr H_\bullet(S)}
\newcommand{\SH}{S\mathscr{H}}
\newcommand{\sm}{\mathscr Sm_S}
\newcommand{\smc} {\mathscr Sm^{\mathrm{cor}}_S}
\newcommand{\DMgme} {DM_{gm}^{eff}\!(S)}
\newcommand{\DMgm} {DM_{gm}\!(S)}
\newcommand{\ftr} {Sh(\smc)}
\newcommand{\DMe} {DM^{eff}\!(S)}

\newcommand{\plim}[1] { \underset{#1}{\varprojlim} \ }
\DeclareMathOperator{\tra}{{}^t \!}

\newcommand{\M}{M}
\newcommand{\Mr}{\tilde M}
\newcommand{\MTh}{M\mathrm{Th}}
\newcommand{\MD}[2]{M\left(\frac{#1}{#2}\right)}
\DeclareMathOperator{\ncup} {\scriptstyle\cup\textstyle}
\DeclareMathOperator{\ncap} {\scriptstyle\cap\textstyle}
\DeclareMathOperator{\gcup} {\scriptstyle\boxtimes\textstyle}

\DeclareMathOperator\pic{Pic}
\newcommand{\spec}[1] {\operatorname{\mathrm{Spec}}(#1)}
\newcommand{\specx}[2] {\operatorname{\mathrm{Spec}_{#1}}\left(#2\right)}
\DeclareMathOperator{\Hom}{Hom}
\DeclareMathOperator{\uHom}{\underline{Hom}}

\newcommand{\lef}[2] {\mathfrak l_{#1}(#2)} 
\newcommand{\pur}[1] { \mathfrak p_{(#1)} } 
\newcommand{\MU}{\mathbf{MU}}
\newcommand{\MGl}{\mathbf{MGL}}
\newcommand{\BGL}{\mathbf{BGL}}

%% file: intro.tex
\section*{Notations}

We fix a noetherian base scheme $S$. 
The schemes considered in this paper are always assumed to be finite type 
$S$-schemes. Similarly, a smooth scheme (resp. morphism of schemes) 
means a smooth $S$-scheme (resp. $S$-morphism of $S$-schemes).
We eliminate the reference to the base $S$ in
all notation (e.g. $\times$, $\PP^n$, ...)

An immersion $i$ of schemes will be a locally closed immersion
and we say $i$ is an open (resp. closed) immersion when $i$ is open
(resp. closed).
We say a morphism $f:Y \rightarrow X$ is projective\footnote{
If $X$ admits an amble line bundle, this definition coincide
with that of \cite{EGA2}.}
 if $Y$ admits a closed $X$-immersion into a \emph{trivial}
  projective bundle over $X$.
  
Given a smooth closed subscheme $Z$ of a scheme $X$,
we denote by $N_ZX$ the normal vector bundle of $Z$ in $X$.
Recall a morphism $f:Y \rightarrow X$ of schemes is said
to be transversal to $Z$ if $T=Y \times_X Z$ is smooth
and the canonical morphism
$N_TY \rightarrow T \times_Z N_ZX$ is an isomorphism.


For any scheme $X$, we denote by $\pic(X)$ the Picard group of $X$.

Suppose $X$ is a smooth scheme.
Given a vector bundle $E$ over $X$, 
we let $P=\PP(E)$ be the projective bundle of lines in $E$.
Let $p:P \rightarrow X$ be the canonical projection.
There is a canonical line bundle $\L_P$ on $P$
 such that $\L_P \subset p^{-1}(E)$. 
We call it the \emph{canonical line bundle} on $P$.
We set $\Q_P=p^{-1}(E)/\L_P$, called the \emph{universal quotient bundle}. 
For any integer $n \geq 0$, we also use the abbreviation 
$\L_n=\L_{\PP^{n}_S}$. We call the projective bundle $\PP(E \oplus 1)$,
with its canonical open immersion $E \rightarrow \PP(E \oplus 1)$,
the \emph{projective completion} of $E$.

\section{Introduction}

In algebraic topology, it is well known that 
oriented multiplicative cohomology theories
correspond to algebras over the complex cobordism spectrum $\MU$. Using the
stable homotopy category allows a systematic treatment of this
kind of generalized cohomology theory, 
which are considered as oriented ring spectra.

In algebraic geometry, the motive associated to a smooth scheme plays the
role of a universal cohomology theory.
In this article, we unify the two approaches~:
on the one hand, we replace ring spectra by spectra with a structure 
of modules over a suitable oriented ring spectra - e.g. the spectrum $\MGl$ of
algebraic cobordism.
On the other hand, we introduce and consider formal group laws in the
motivic theory, generalizing the classical point of view.

More precisely, we use an axiomatic treatment based on homotopy invariance 
and excision property which allows to formulate results in a triangulated category 
which models both stable homotopy category and mixed motives. 
A suitable notion of orientation is introduced which implies 
the existence of Chern classes together with a formal group law. 
This allows to prove a purity theorem which implies the existence of Gysin morphisms
 for closed immersions and their companion residue morphisms. 
We extend the definition of the Gysin morphism to the case of a projective morphism, 
which involves a delicate study of cobordism classes 
in the case of an arbitrary formal group law. 
This theory then implies very neatly the duality statement 
in the projective smooth case.
Moreover, these constructions are obtained over an arbitrary base scheme,
eventually singular and with unequal characteristics.

Examples are given
which include triangulated mixed motives, 
generalizing the constructions and results of V.~Voevodsky, and $\MGl$-modules.
Thus, this work can be applied in motivic cohomology (and motivic homology), 
as well as in algebraic cobordism.
It also applies in homotopy algebraic $K$-theory\footnote{Recall homotopy algebraic
K-theory was introduced by Weibel in \cite{Wei}. This cohomology theory
coincide with algebraic K-theory when $S$ is regular.}
 and some of the formulas obtained here are new in this context.
It can be applied finally to classical cohomology theories through the notion
of a \emph{mixed Weil theory} introduced in \cite{CD2}.
In the case of rigid cohomology, the formulas and constructions given here
generalize some of the results obtained by P.~Berthelot and D.~P\'etrequin.
Moreover, the theorems proved here are used in an essential way in \cite{CD2}.

%

\subsection{The axiomatic framework}


We fix a triangulated symmetric monoidal category $\T$,
 with unit $\un$,
 whose objects are simply called \emph{motives}\footnote{
A correct terminology would be to call these objects
 \emph{generalized triangulated motives}
  or triangulated motives with coefficients as the
  triangulated mixed motives defined by Voevodsky
  are particular examples.}.
To any pair of smooth schemes $(X,U)$ such that $U \subset X$ is associated
 a motive $\M(X/U)$ functorial with respect to $U \subset X$,
  and a canonical distinguished triangle~:
$$
\M(U) \rightarrow \M(X) \rightarrow \M(X/U) \xrightarrow{\partial} \M(U)[1],
$$
where we put $M(U):=M(U/\emptyset)$ and so on.
The first two maps are obtained by functoriality.
As usual, the Tate motive is defined to be
 $\un(1):=\M(\PP^1_S/S_\infty)[-2]$
where $S_\infty$ is the point at infinity.

The axioms we require are, for the most common,
additivity (Add), homotopy invariance (Htp), 
Nisnevich excision (Exc), K\"unneth formula for pairs of schemes (Kun)
and stability (Stab) -- \emph{i.e.} invertibility of $\un(1)$ 
(see paragraph \ref{axioms} for the precise statement). 
All these axioms are satisfied
 by the stable homotopy category of schemes of F.~Morel and V.~Voevodsky. 
However, we require a further axiom which is in fact our principal object
of study, the orientation axiom (Orient)~: to any line bundle $L$ over a smooth
scheme $X$ is associated a morphism $c_1(L):\M(X) \rightarrow \un(1)[2]$
-- the first Chern class of $L$ --
compatible with base change and constant on the isomorphism class of $L/X$.

The best known example of a category satisfying this set of axioms 
is the triangulated category of (geometric) mixed motives over $S$, 
denoted by $\DMgm$.
It is defined according to V. Voevodsky along the lines of the case of 
a perfect base field but replacing Zariski topology by the Nisnevich one
 (cf section \ref{motives}).
Another example can be obtained by considering
the category of oriented spectra in the sense of F. Morel (see \cite{Vez}).
However, in order to define a monoidal structure on that category,
we have to consider modules over the algebraic cobordism spectrum
$\MGl$, in the $E_\infty$-sense. One can see that oriented spectra are
equivalent to $\MGl$-modules, but the tensor product is given
with respect to the $\MGl$-module structure.

Any object $\E$ of the triangulated category $\T$ defines a bigraded cohomology
  (resp. homology) theory on smooth schemes by the formulas
$$
\E^{n,p}(X)=\Hom_\T\big(\M(X),\un(p)[n]\big)
 \ resp. \ \E_{n,p}(X)=\Hom_\T\big(\un(p)[n],\E \otimes \M(X)\big).
$$
As in algebraic topology, there is a rich algebraic structure on these
graded groups (see section \ref{sec:products}). 
The K\"unneth axiom (Kun) implies that,
in the case where $\E$ is the unit object $\un$,
we obtain a multiplicative cohomology theory simply denoted by $H^{**}$.
It also implies that for any smooth scheme $X$,
 $\E^{**}(X)$ has a module structure over $H^{**}(X)$.
More generally, if we put $A=H^{**}(S)$, called the ring of (universal) coefficients,
cohomology and homology groups of the previous kind are graded $A$-modules.

\subsection{Central constructions}

These axioms are sufficient to establish an essential basic fact,
 the projective bundle theorem~: \\
(\textbf{th. \ref{th:projbdl}})\footnote{The proof is essentially
based on a very elegant lemma due to F. Morel.}
Let $X$ be a smooth scheme,
 $P \xrightarrow p X$ be a projective bundle of dimension $n$,
  and $c$ be the first Chern class of the canonical line bundle.
Then the map: \\
 $\sum_{0 \leq i \leq n} p_* \gcup c^i:
M(P) \rightarrow \bigoplus_{0 \leq i \leq n} M(X)(i)[2i]$
is an isomorphism.

Remark that considering any motive $\E$, even without ring structure,
we obtain \\ $\E^{**}(P)=\E^{**}(X) \otimes_{H^{**}(X)} H^{**}(P)$
where tensor product is taken with respect 
to the $H^{**}(X)$-module structure.
In the case $\E=\un$, we thus obtain the projective
bundle formula for $H^{**}$ which allows the definition of 
\textit{(higher) Chern classes} following the classical method 
of Grothendieck~:

\noindent (\textbf{def. \ref{df:Chern}})
For any smooth scheme $X$, 
 any vector bundle $E$ over $X$
  and any integer $i \geq 0$, $c_i(E):\M(X) \rightarrow \un(i)[2i]$.

Moreover, the projective bundle formula leads to the following
constructions~:
\begin{enumerate}
\item[(i)] (\textbf{\ref{fgl} \& \ref{nilpotence&FGL}})
A \emph{formal group law} $F(x,y)$ over $A$ such that 
 for any smooth scheme $X$ which admits an ample line bundle,
  for any line bundles $L,L'$ over $X$, the formula
$$
c_1(L \otimes L')=F(c_1(L),c_1(L'))
$$
is well defined and holds in the $A$-algebra $H^{**}(X)$.
\item[(ii)] (\textbf{def. \ref{df:Gysin}})
For any smooth schemes $X$, $Y$
 and any projective morphism $f:Y \rightarrow X$ of relative dimension $n$, 
  the associated \emph{Gysin morphism} $f^*:\M(X) \rightarrow \M(Y)(-n)[-2n]$.
\item[(iii)] (\textbf{def. \ref{df_purity&Gysin}})
For any closed immersion $i:Z \rightarrow X$ of codimension $n$
between smooth schemes, with complementary open immersion $j$,
the Gysin triangle~:
$$
\M(X-Z) \xrightarrow{j_*} \M(X) \xrightarrow{i^*} \M(Z)(n)[2n]
 \xrightarrow{\partial_{X,Z}} \M(X-Z)[1].
$$
The last morphism in this triangle is called the \emph{residue morphism}.
\end{enumerate}

The Gysin morphism permits the construction of a \emph{duality pairing}
 in the pure case~: \\
(\textbf{th. \ref{thm:duality}})
For any smooth projective scheme $p:X \rightarrow S$ 
of relative dimension  $n$, 
with diagonal embedding $\delta:X \rightarrow X \times X$, 
there is a strong duality\footnote{Note we use essentially 
 the axiom (Kun) here.} (in the sense of Dold-Puppe)~:
\begin{equation} \tag{iv}
\begin{split}
\mu_X: \ & \un \xrightarrow{p^*} \M(X)(-n)[-2n]
 \xrightarrow{\delta_*} \M(X)(-n)[-2n] \otimes \M(X) \\
\epsilon_X: \ & \M(X) \otimes \M(X)(-n)[-2n]
 \xrightarrow{\delta^*} \M(X) \xrightarrow{p_*} \un.
\end{split}
\end{equation}
In particular,
 the Hom object $\uHom(\M(X),\un)$ is defined in the monoidal
category $\T$ and $\mu_X$ induces a canonical duality isomorphism~:
\begin{center}
$\uHom(\M(X),\un) \rightarrow \M(X)(-n)[-2n]$.
\end{center}
This explicit duality allows us to recover the usual form of duality 
between cohomology and homology as in algebraic topology,
 in terms of the fundamental class of $X$ and cap-product on one hand
  and in terms of the fundamental class of $\delta$ and slant product
   on the other hand.
Moreover, considering a motive $\E$ with a monoid structure in $\T$
and such that the cohomology $\E^{**}$ satisfies the K\"unneth formula, 
we obtain the usual Poincar\'e duality theorem in terms of the trace morphism 
(induced by the Gysin morphism $p^*:\un \rightarrow \M(X)(n)[2n]$)
and cup-product
(see \textbf{paragraph \ref{explicit_duality_coh_hom}}). \\
Note also we deduce easily from our construction that the Gysin morphism
associated to a morphism $f$ between smooth projective scheme
is the dual of $f_*$ (\textbf{prop. \ref{Gysin=transpose}}).

Remark finally that, considering any closed subscheme $Z_0$ of $S$, 
and taking tensor product with the motive $M(S/S-Z_0)$
 in the constructions (ii), (iii) and (iv), 
  we obtain a \emph{Gysin morphism and a Gysin triangle with support}. 
For example, given a projective morphism 
$f:Y \rightarrow X$ as in $(ii)$, $Z=X \times_S Z_0$ and
$T=Y \times_S Z_0$, we obtain the morphism
 $\M_Z(X) \rightarrow \M_T(Y)(-n)[-2n]$.
Similarly, if $X$ is projective smooth of relative dimension $n$,
 $\M_Z(X)$ admits a strong dual, $M_Z(X)(-n)[-2n]$.
Of course,
 all the other formulas given below are valid for these motives with support.

\subsection{Set of formulas}

The advantage of the
motivic point of view is to obtain universal formulae which 
imply both cohomological and homological statements, 
 with a minimal amount of algebraic structure involved.

\subsubsection{Gysin morphism}
We prove the basic properties of the Gysin morphisms
such as functoriality ($g^*f^*=(fg)^*$),
compatibility with the monoidal structure $(f \times g)^*=f^* \otimes g^*$), 
the projection formula ($(1_{Y*} \gcup f_*)f^*=f^* \gcup 1_{X*}$)
 and the base change formula in the transversal case 
 ($f^*p_*=q_*g^*$).

For the needs of the following formulae, we introduce a useful notation
 which appear in the article. For any smooth scheme $X$,
  any cohomology class $\alpha \in H^{n,p}(X)$ and
   any morphism $\phi:\M(X) \rightarrow \M$ in $\T$,
we put
$$
\phi \gcup \alpha:=(\phi \otimes \alpha) \circ \delta_*
 :\M(X) \rightarrow \M(p)[n]
$$
where $\alpha$ is considered as a morphism $\M(X) \rightarrow \un(p)[n]$,
 and $\delta_*:\M(X) \rightarrow \M(X) \otimes \M(X)$
  is the morphism induced by the diagonal of $X/S$ and by
  the K\"unneth axiom (Kun).

More striking are the following formulae which express 
 the \emph{defect} in base change formulas.
Fix a commutative square of smooth schemes
\begin{equation} \label{squareDelta}
\xymatrix@=10pt{
T\ar^q[r]\ar_g[d]\ar@{}|\Delta[rd] & Y\ar^f[d] \\
Z\ar_p[r] & X
}
\end{equation}
which is cartesian on the underlying topological spaces,
 and such that $p$ (resp. $q$) is projective of
 relative dimension $n$ (resp. $m$).

\noindent \textit{Excess of intersection}
(\textbf{prop. \ref{excess}}).-- 
Suppose the square $\Delta$ is cartesian.
We then define the \emph{excess intersection bundle} $\xi$ associated to $\Delta$
 as follows.
Choose a projective bundle $P/X$
and a closed immersion $Z \xrightarrow i P$ over $X$ with normal
bundle $N_ZP$. Consider the pullback $Q$ of $P$ over $Y$ and 
the normal bundle $N_YQ$ of $Y$ in $Q$. 
Then $\xi=N_YQ/g^{-1}N_ZP$ is independent up to isomorphism 
 of the choice of $P$ and $i$. 
The rank of $\xi$ is the integer $e=n-m$.

Then, $p^*f_*= \big(g_* \gcup c_e(\xi)\big)q^*$.

\noindent \textit{Ramification formula} (\textbf{th. \ref{ramification}}).--
Consider the square $\Delta$ and assume $n=m$.
Suppose that $T$ admits an ample line bundle
 and (for simplicity) that $S$ is integral\footnote{
We prove in the text a stronger statement assuming only that $S$ is reduced.}. \\
Let $T=\cup_{i \in I} T_i$ be the decomposition of $T$ into connected components. 
Consider an index $i \in I$.
We let $p_i$ and $g_i$ be the restrictions of $p$ and $g$ to $T_i$.
The canonical map $T \rightarrow Z \times_X Y$ is a thickening.
Thus, the connected component $T_i$ corresponds to a unique connected
component $T'_i$ of $Z \times_X Y$. According to the classical definition,
the \emph{ramification index} of $f$ along $T_i$ is the geometric 
multiplicity $r_i \in \NN^*$ of $T'_i$.
We define (cf \textbf{def. \ref{inter_mult}})
a generalized intersection multiplicity 
for $T_i$ which takes into account the formal group law $F$, called for this reason
the \emph{$F$-intersection multiplicity}. 
It is an element $r(T_i;f,g) \in H^{0,0}(T_i)$.
We then prove the formula~:
$$
p^*f_*=\sum_{i \in I} \big(r(T_i;f,g) \gcup_{T_i} g_{i*}\big) q_i^*.
$$
In general, $r(T_i;f,g)=r_i+\epsilon$ where 
the correction term $\epsilon$ is a function of the coefficients of $F$ -- 
it is zero when $F$ is additive.

\subsubsection{Residue morphism}

A specificity of the present work is the study of the Gysin triangle,
 notably its boundary morphism, called the residue morphism.
Consider a square $\Delta$ as in \eqref{squareDelta}.
Put $U=X-Z$, $V=Y-T$ and let $h:V \rightarrow U$ be the morphism
induced by $f$.

We obtain the following
 formulas~:
\begin{enumerate}
\item $(j_* \gcup 1_{U*}) \partial_{X,Z}=\partial_{X,Z} \gcup i_*$.
\item For any smooth scheme $Y$,
 $\partial_{X \times Y,Z \times Y}=\partial_{X,Z} \otimes 1_{Y*}$. 
\item If $f$ is a closed immersion,
$\partial_{X-Z,Y-T} \partial_{Y,T}
 +\partial_{X-Y,Z-T} \partial_{Z,T}=0$.
\item If $f$ is projective,
 $\partial_{Y,T} g^*=h^* \partial_{X,Z}$.
\item When $f$ is transversal to $i$, $h_* \partial_{Y,T}=\partial_{X,Z} g_*$.
\item When $\Delta$ satisfies the hypothesis of \emph{Excess of intersection}, \\
\indent $h_* \partial_{Y,T}=\partial_{X,Z} \big(g_* \gcup c_e(\xi)\big)$.
\item When $\Delta$ satisfies the hypothesis of \emph{Ramification formula}, \\
\indent $\sum_{i \in I} h_* \partial_{Y,T_i}
=\sum_{i \in I} \partial_{X,Z} \big(r(T_i;f,g) \gcup g_{i*}\big)$.
\end{enumerate}

The differential taste of the residue morphism appears clearly in the last formula
(especially the cohomological formulation) where the multiplicity $r(T_i;f,g)$
takes into account the ramification index $r_i$.
Even in algebraic $K$-theory, this formula seems to be new.
 
\subsubsection{Blow-up formulas}
Let $X$ be a smooth scheme
 and $Z \subset X$ be a smooth closed subscheme of codimension $n$.
Let $B$ be the blow-up of $X$ with center $Z$ and consider the cartesian square
$
\xymatrix@=10pt{
P\ar^k[r]\ar_p[d] & B\ar^f[d] \\
Z\ar^i[r] & X
}
$.
Let $e$ be the top Chern class of the canonical quotient bundle 
on the projective space $P/Z$.
\begin{enumerate}
\item (\textbf{prop. \ref{proj_bdl_fml1}})
Let $\M(P)/\M(Z)$ be the kernel of the split monomorphism $p_*$.
The morphism $(k_*,f^*)$ induces an isomorphism~:
\begin{center}
 $\M(P)/\M(Z) \oplus \M(X) \rightarrow \M(B)$.
\end{center}
\item (\textbf{prop. \ref{proj_bdl_fml2}}) 
The short sequence
$$
0 \rightarrow \M(B)
 \xrightarrow{
\text{\scriptsize $\begin{pmatrix} k^* \\ f_* \end{pmatrix}$}}
 \M(P)(1)[2] \oplus \M(X)
  \xrightarrow{(p_* \gcup e,-i^*)} \M(Z)(n)[2n] \rightarrow 0
$$
is split exact. Moreover, 
$(p_* \gcup e,-i^*) \circ ${\tiny $\begin{pmatrix} p^* \\ 0 \end{pmatrix}$}
is an isomorphism\footnote{This isomorphism is the identity at least in the
 case when $F(x,y)=x+y$}.
\end{enumerate}

The first formula was obtained by V. Voevodsky using resolution of singularities
in the case where $S$ is the spectrum of a perfect field.
The second formula is the analog of a result on Chow groups,
formulated by W. Fulton (cf \cite[6.7]{Ful}).

\subsection{Characteristic classes}

Besides Chern classes, we can introduce the following characteristic
classes in our context.

Let $i:Z \rightarrow X$ be a closed immersion of codimension $n$
between smooth schemes, $\pi:Z \rightarrow S$ the canonical projection. 
We define the \emph{fundamental class} of $Z$ in $X$ 
(\textbf{paragraph \ref{fund_class}})
as the cohomology class represented by the morphism
$$
\eta_X(Z):\M(X) \xrightarrow{i^*} \M(Z)(n)[2n] \xrightarrow{\pi_*} \un(n)[2n].
$$
It is a cohomology class in $H^{2n,n}(X)$
 satisfying the more classical expression $\eta_X(Z)=i_*(1)$.

Considering the hypothesis of the ramification formula above,
 when $n=m=1$,
we obtain the enlightening formula (cf \textbf{cor. \ref{pullback_fund_class}})~:
$$
f^*(\eta_X(Z))=\sum_{i \in I} [r_i]_F \cdot \eta_Y(T_i)
$$
where $r_i$ is the ramification index of $f$ along $T_i$ and
$[r_i]_F$ is the $r_i$-th formal sum with respect to $F$ applied
to the cohomological class $\eta_Y(T_i)$. Indeed, 
the fact $T$ admits an ample line bundle implies this class is nilpotent.

The most useful fundamental class in the article is
 the Thom class of a vector bundle $E/X$ of rank $n$. 
Let $P=\PP(E \oplus 1)$ be its projective completion
and consider the canonical section $X \xrightarrow s P$.
The \emph{Thom class} of $E/X$ is $t(E):=\eta_P(X)$. 
By the projection formula,
 $s^*=p_* \gcup t(E)$, where $p:P \rightarrow X$
is the canonical projection.
Let $\L$ (resp. $\Q$) be the canonical line bundle 
(resp. universal quotient bundle) on $P/X$. 
We also obtain the following equality\footnote{This
corrects an affirmation of I. Panin 
in the introduction of \cite[p. 268]{Panin_r}
where equality {\tiny $(*)$} is said not to hold.}~:
$$
t(E)=c_n(\L^\vee \otimes p^{-1}E)\stackrel{(*)}=c_n(\Q)
 =\sum_{i=0}^n c_i(p^{-1}E) \ncup \big(-c_1(\L)\big)^i.
$$
This is straightforward in the case where $F(x,y)=x+y$ but more
difficult in general.

We also obtain a computation which the author has not seen in
the literature (even in complex cobordism).
Write $F(x,y)=\sum_{i,j} a_{ij}.x^iy^j$ with $a_{ij} \in A$.
Consider the diagonal embedding 
$\delta:\PP^n \rightarrow \PP^n \times \PP^n$.
Let $\L_1$ and $\L_2$ be the respective canonical line bundle on
the first and second factor of $\PP^n \times \PP^n$.
Then (\textbf{prop. \ref{fdl_class_diagonal}}) the fundamental class 
of $\delta$ satisfies
$$
\eta_{\PP^n \times \PP^n}(\PP^n)
=\sum_{0 \leq i,j \leq n} a_{1,i+j-n}.c_1(\L_1^\vee)^ic_1(\L_2^\vee)^j.
$$

Another kind of characteristic classes are \emph{cobordism classes}.
Let $p:X \rightarrow S$ be a smooth projective scheme of relative dimension $n$.
The cobordism class of $X/S$ is the cohomology class represented by the morphism
$$
[X]:\un \xrightarrow{p^*} \M(X)(-n)[-2n] \xrightarrow{p_*} \un(-n)[-2n].
$$
It is a class in $A^{-2n,-n}$. As an application of the previous equality,
we obtain the following computation (\textbf{cor. \ref{cor:Myschenko}})~:
$$
[\PP^n]=(-1)^n.\det\left(
\raisebox{0.5\depth}{
\xymatrix@=0.001pt{
0\ar@{.}[rr]\ar@{.}[dd]
 & & 0\ar@{.}[lldd] & 1\ar@{-}[lllddd]
 & a_{1,1}\ar@{-}[lllldddd] \\
&&&& a_{1,2}\ar@{-}[lllddd]\ar@{.}[ddd] \\
0 & & & & \\
1 & & & & \\
a_{1,1} & a_{1,2}\ar@{.}[rrr] & & & a_{1,n}
}
}
\right)
$$
which of course coincides with the expression given by the classical theorem
of Myschenko in complex cobordism. 
In fact, our method gives a new proof of the latter theorem.

\subsection{Outline of the work}

In section 2, we give the list of axioms (cf \ref{axioms}) 
satisfied by the category $\T$ and discuss the first consequences
of these. Remark an originality of our axiomatic is that we 
not only consider pairs of schemes but also quadruples
(used in the proof of \ref{thm:assoc}).
The last subsection 2.3 gives the principle examples
which satisfy the axiomatic \ref{axioms}. Section 3 contains 
the projective bundle theorem and its consequences,
the formal group law and Chern classes. \\
Section 4 contains the study of the Gysin triangle. 
The fundamental result in this section is the purity theorem \ref{prop:purity}. 
Usually,
one constructs the Thom isomorphism using the Thom class (\ref{Thom_class}).
Here however, we directly construct the former isomorphism from the
projective bundle theorem and the deformation to the normal cone. 
This makes the construction more canonical 
-- though there is a delicate choice of signs hidden 
(cf beginning of section \ref{sec:purity}) -- 
and it thus gives a canonical Thom class. 
We then study the two principle subjects around the Gysin
triangle~: the base change formula and its defect (section 4.2
which contains notably \ref{ramification} and \ref{excess} cited above) 
and the interaction (containing notably the functoriality of the Gysin
morphism) of two Gysin triangles attached with smooth subschemes of
a given smooth scheme (th. \ref{thm:assoc}). \\
In section 5, we first recall the notion of
\emph{strong duality} introduced by A. Dold and D. Puppe and give
some complements. 
Then we give the construction of the Gysin morphism
 in the projective case and the duality statement.
The general situation is particularly complicated when
the formal group law $F$ is not the additive one, as the Gysin morphism
associated to the projection $p$ of $\PP^n$ is not easy to handle.
Our method is to exploit the strong duality on $\PP^n$ 
implied by the projective bundle theorem. 
We show that the fundamental class of the diagonal $\delta$ of $\PP^n/S$
determines canonically the Gysin morphism of the projection 
(see def. \ref{gysin4projection}). This is due to the explicit form
of the duality pairing for $\PP^n$ cited above~: 
the motive $\M(\PP^n)$ being strongly dualizable, 
one morphism of the duality pairing $(\mu_X,\epsilon_X)$ 
determines the other; the first one is induced by $\delta^*$
and the other one by $p^*$. Once this fact is determined,
we easily obtain all the properties required to define the Gysin
morphism and then the general duality pairing.
The article ends with the explicit determination of the cobordism
class of $\PP^n$ and the blow-up formulas as illustrations of
the theory developed here.

\subsection{Final commentary}

In another work \cite{Deg6},
we study the Gysin triangle directly in the category of geometric
mixed motives over a perfect field. In the latter, we used the isomorphism
of the relevant part of motivic cohomology groups and Chow groups and prove
our Gysin morphism induces the usual pushout on Chow groups 
via this isomorphism (cf \cite[1.21]{Deg6}).
This gives a shortcut for the definitions and propositions 
proved here
in the particular case of motives over a perfect field.
In \emph{loc. cit.} moreover, we also use the isomorphism
between the diagonal part of the motivic cohomology groups
of a field $L$ and the Milnor K-theory of $L$ and prove
our Gysin morphism induces the usual norm morphism on Milnor
K-theory (cf \cite[3.10]{Deg6}) -- after a limit process,
considering $L$ as a function field.

The present work is obviously linked with the fundamental book 
on algebraic cobordism by Levine and Morel \cite{LM}
(see also \cite{Lev2}), 
but here, we study oriented cohomology theories 
from the point of view of stable homotopy.
This point of view is precisely that of \cite{Lev1}.
It is more directly linked with the prepublication \cite{Panin} 
of I. Panin which was mainly concerned with the construction
of pushforwards in cohomology, corresponding to our Gysin morphism
(see also \cite{Smi} and \cite{Pim} for extensions of this work).
Our study gives a unified self-contained treatment of all these works,
 except that we have not considered here the theory of transfers
  and Chern classes with support (see \cite{Smi},
   \cite[part 5]{Lev1}).

The final work we would like to mention is the thesis of J. Ayoub
on cross functors (\cite{Ayoub}). In fact, it is now folklore that
the six functor formalism yields a construction of the Gysin morphism.
In the work of Ayoub however, the questions of orientability are not
treated. In particular, 
the Gysin morphism we obtain takes value in a certain Thom space.
To obtain the Gysin morphism in the usual form,
 we have to consider the Thom isomorphism introduced here. 
Moreover, we do not need the localization property in our study 
 whereas it is essentially used in the formalism of cross functors. 
This is a strong property which is not known in general 
for triangulated mixed motives.
Finally, the interest of this article relies in the study of the 
defect of the base change formula which is not covered by the six 
functor formalism.

\subsection*{Remerciements}
J'aimerais remercier tout sp\'ecialement Fabien Morel car ce travail,
commenc\'e \`a la fin de ma th\`ese, 
a b\'en\'efici\'e de ses nombreuses id\'ees et de son support. 
Aussi, l'excellent rapport d'une version pr\'eliminaire de l'article 
 \cite{Deg6}
 m'a engag\'e \`a le g\'en\'eraliser sous la forme pr\'esente;
 j'en remercie le rapporteur, ainsi que J\"org Wildeshaus pour son soutien.
Je tiens \`a remercier Geoffrey Powell pour m'avoir grandement aid\'e
\`a clarifier l'introduction de cet article et Jo\"el Riou pour m'avoir 
indiqu\'e une incoh\'erence dans une premi\`ere version de la formule
de ra\-mification.
Mes remerciements vont aussi au rapporteur de cet article
 pour sa lecture attentive qui m'a notamment aid\'ee \`a clarifier
 les axiomes.
Je souhaite enfin adresser un mot \`a Denis-Charles 
Cisinski pour notre amiti\'e math\'ematique qui a \'et\'e la meilleure
des muses.

%% file: axioms.tex
\section{The general setting : homotopy oriented triangulated systems}

\subsection{Axioms and notations} \label{general_setting}

Let $\D$ be the category whose objects are the cartesian squares
\begin{equation} \label{std}
\xymatrix@=10pt{W\ar[r]\ar[d]\ar@{}|\Delta[rd] & V\ar[d] \\ U\ar[r] & X}
 \tag{$*$}
\end{equation}
made of immersions between smooth schemes.
The morphisms in $\D$ are the evident commutative cubes.
We will define the \emph{transpose} of the square $\Delta$,
denoted by $\Delta'$, as the square
$$
\xymatrix@=10pt{W\ar[r]\ar[d]\ar@{}|{\Delta'}[rd] & U\ar[d] \\ V\ar[r] & X}
$$
made of the same immersions. This defines an endofunctor of $\D$.

In all this work, 
we consider a triangulated symmetric monoidal category 
$(\T,\otimes,\un)$ together with a covariant functor $\M:\D \rightarrow \T$. 
Objects of $\T$ are called \emph{premotives}.

Considering a square as in \eqref{std}, we adopt the suggesting notation
$$\M\left(\frac{X/U}{V/W}\right)=\M(\Delta).$$
We simplify this notation in the following two cases~:
\begin{enumerate}
\item If $V=W=\emptyset$, we put $\M(X/U)=\M(\Delta)$.
\item If $U=V=W=\emptyset$, we put $\M(X)=\M(\Delta)$.
\end{enumerate}

We call \emph{closed pair} any pair $(X,Z)$ of schemes such that $X$ is smooth 
and $Z$ is a closed (not necessarily smooth) subscheme of $X$.
As usual, we define the premotive of $X$ with support in $Z$ as $\M_Z(X)=\M(X/X-Z)$.

A \emph{pointed scheme} is a scheme $X$ together with an $S$-point
$x:S \rightarrow X$.
When $X$ is smooth, the reduced premotive associated with $(X,x)$
will be $\Mr(X,x)=M(S \xrightarrow x X)$. Let $n>0$ be an integer. 
We will always assume the smooth scheme $\PP^n_S$ is pointed by the infinity. 
We define the \emph{Tate twist} as the premotive $\un(1)=\Mr(\PP^1_S)[-2]$ of $\T$. 

\num \label{axioms}
We suppose the functor $\M$ satisfies the following axioms~:
\begin{enumerate}
\item[(Add)] For any finite family of smooth schemes $(X_i)_{i \in I}$, \\
\indent $\M(\sqcup_{i \in I} X_i)=\oplus_{i \in I} \M(X_i)$.
\item[(Htp)] For any smooth scheme $X$, the canonical projection of the
affine line induces an isomorphism $\M(\AA^1_X) \rightarrow \M(X)$.
\item[(Exc)] Let $(X,Z)$ be a closed pair and $f:V \rightarrow X$ be an \'etale morphism.
Put $T=f^{-1}(Z)$ and suppose the map $T_{red} \rightarrow Z_{red}$
obtained by restriction of $f$ is an isomorphism. 
Then the induced morphism $\phi:M_T(V) \rightarrow M_Z(X)$ is an isomorphism.
\item[(Stab)] The Tate premotive $\un(1)$ admits an inverse for the tensor product
denoted by $\un(-1)$.
\item[(Loc)] For any square $\Delta$ as in \eqref{std},
a morphism 
$\partial_\Delta:\M\left(\frac{X/U}{V/W}\right) \rightarrow \M(V/W)[1]$
is given natural in $\Delta$ and 
such that the sequence of morphisms
$$
\M(V/W) \rightarrow \M(X/U)
 \rightarrow \M\left(\frac{X/U}{V/W}\right) \xrightarrow{\partial_\Delta}
  \M(V/W)[1]
$$
made of the evident arrows is a distinguished triangle in $\T$.
\item[(Sym)] Let $\Delta$ be a square as in \eqref{std}
and consider its transpose $\Delta'$.
There is given a morphism 
$\epsilon_\Delta:\M\left(\frac{X/U}{V/W}\right) \rightarrow
 \M\left(\frac{X/V}{U/W}\right)$
natural in $\Delta$.

If in the square $\Delta$, $V=W=\emptyset$, we put 
$$
\partial_{X/U}=\partial_{\Delta'} \circ \epsilon_{\Delta}:
 M(X/U) \rightarrow M(U)[1].
$$
 
We ask the following coherence properties~:
\begin{enumerate}
\item[(a)] $\epsilon_{\Delta'} \circ \epsilon_{\Delta}=1$.
\item[(b)] If $\Delta=\Delta'$ then $\epsilon_{\Delta}=1$.
\item[(c)] The following diagram is anti-commutative~:
$$\xymatrix@R=11pt@C=14pt{
M\left(\frac{X/U}{V/W}\right)\ar^{\epsilon_{\Delta}}[r]\ar_{\partial_{\Delta}}[d]
 & M\left(\frac{X/V}{U/W}\right)\ar^/-5pt/{\partial_{\Delta'}}[r]
 & M(U/W)[1]\ar^{\partial_{U/W}[1]}[d] \\
M(V/W)[1]\ar^{\partial_{V/W}[1]}[rr] & & M(W)[2].
}$$
\end{enumerate}
\item[(Kun)] \begin{enumerate}
\item[(a)] For any open immersions $U \rightarrow X$ and $V \rightarrow Y$ of smooth schemes,
 there are canonical isomorphisms: \\
$\M(X/U) \otimes \M(Y/V)=\M(X \times Y/X \times V \cup U \times Y), \quad M(S)=\un
$ \\
satisfying the coherence conditions of a monoidal functor.
\item[(b)] Let $X$ and $Y$ be smooth schemes and $U \rightarrow X$ be an open immersion.
Then, $\partial_{X \times Y/U \times Y}=\partial_{X/U} \otimes 1_{Y*}$
through the preceding canonical isomorphism.
\end{enumerate}
\item[(Orient)] For any smooth scheme $X$,
there is an application, called the \emph{orientation},
$$
c_1:\pic(X) \rightarrow \Hom_\T(\M(X),\un(1)[2])
$$
which is functorial in $X$ and such that
the class $c_1(\L_1):\M(\PP^1_S) \rightarrow \un(1)[2]$
is the canonical projection.
\end{enumerate}
For any integer $n \in \NN$, we let $\un(n)$ 
(resp. $\un(-n)$) be the $n$-th tensor power of $\un(1)$ 
(resp. $\un(-1)$). 
Moreover, for an integer $n \in \ZZ$ and a premotive 
$\E$, we put $\E(n)=\E \otimes \un(n)$.

\num  \label{further_prop}
Using the excision axiom (Exc) and an easy noetherian induction,
we obtain from the homotopy axiom (Htp) the following stronger result~:
\begin{enumerate}
\item[(Htp')]
For any fiber bundle $E$ over a smooth scheme $X$,
the morphism induced by the canonical projection $\M(E) \rightarrow \M(X)$
is an isomorphism.
\end{enumerate}
We further obtain the following interesting property~:
\begin{enumerate}
\item[(Add')]
Let $X$ be a smooth scheme and $Z$, $T$ be disjoint closed subschemes of $X$. \\
Then the canonical map $\M_{Z \sqcup T}(X) \rightarrow \M_Z(X) \oplus \M_T(X)$
induced by naturality is an isomorphism.
\end{enumerate}
Indeed, using (Loc) with $V=X-T$, $W=X-(Z \sqcup T)$ and $U=W$, 
we get a distinguished triangle
$$
\M_Z(V) \rightarrow \M_{Z \sqcup T}(X) \xrightarrow{\pi}
 \M\left(\frac{X/W}{V/W}\right) \rightarrow \M_Z(V)[1].
$$
Using (Exc), we obtain $\M_Z(V)=\M_Z(X)$. The natural map
$\M_{Z \sqcup T}(X) \rightarrow \M_Z(X)$ induces a retraction of the first arrow.
Moreover, we get $\M\left(\frac{X/W}{V/W}\right)=\M_T(X)$ from the symmetry axiom (Sym).
Note that we need (Sym)(b) and the naturality of $\epsilon_\Delta$ 
to identify $\pi$ with the natural map $\M_{Z \sqcup T}(X) \rightarrow \M_T(X)$.

\begin{rem} \textit{About the axioms}.--- 
\begin{enumerate}
\item There is a stronger form
of the excision axiom (Exc) usually called the Brown-Gersten property (or
distinguished triangle). In the situation of axiom (Exc), with $U=X-Z$ and
$W=V-T$,
 we consider the cone in the sense of \cite{Nee}
  of the morphism of distinguished triangles
$$
\xymatrix@=10pt{
\M(W)\ar[r]\ar[d] & \M(V)\ar[r]\ar[d] & \M(V/W)\ar[r]\ar[d] & \M(V)[1]\ar[d] \\
\M(U)\ar[r] & \M(X)\ar[r] & \M(X/U)\ar[r] & \M(U)[1] \\
}
$$
This is a \emph{candidate triangle} in the sense of \textit{op. cit.} of the form
$$
\M(W) \rightarrow \M(U) \oplus \M(V) \rightarrow \M(X) \rightarrow \M(W).
$$
Thus, in our abstract setting, 
 it is not necessarily a distinguished triangle.
We call (BG) the hypothesis that in every such situation, 
the candidate triangle obtained above is a distinguished triangle.
We will not need the hypothesis (BG) ; however, in the applications, 
it is always true and the reader may use this stronger form for
simplification.
\item We can replace axiom (Kun)(a) by a weaker one
\begin{enumerate}
\item[(wKun)]
The restriction of $\M$ to the category of pairs of schemes $(X,U)$
is a \emph{lax} monoidal symmetric functor.
\end{enumerate}
(Kun)(b) is then replaced by an obvious coherence property of the
boundary operator in (Loc).
This hypothesis is sufficient for the needs of the article
with a notable exception of the duality pairing \ref{thm:duality}.
For example, if one wants to work with cohomology theories directly,
 one has to use rather this axiom, 
  replace $\T$ by an abelian category and
   "distinguished triangle" by "long exact sequence" everywhere.
The arguments given here covers equally this situation, 
except for the general duality pairing.
\item The symmetry axiom (Sym) encodes a part of a richer structure which
possess the usual examples
(all the ones considered in section \ref{sec:htp_funct}).
This is the structure of a derivator as the object $\M(\Delta)$
may be seen as a homotopy colimit. The coherence axioms which appear
in (Sym) are very natural from this point of view.
\end{enumerate}
\end{rem}

\begin{df}
Let $\E$ be a premotive.
For any smooth scheme $X$ and any  couple $(n,p) \in \ZZ \times \ZZ$,
we define respectively the cohomology and the homology groups of $X$ with 
coefficient in $\E$ as
\begin{align*}
\E^{n,p}(X)&=\Hom_\T\big(M(X),\E(p)[n]\big), \\
\text{resp. } \E_{n,p}(X)&=\Hom_\T\big(\un(p)[n],\E \otimes M(X)\big).
\end{align*}
We refer to the corresponding bigraded cohomology group (resp. homology group)
by $\E^{**}(X)$ (resp. $\E_{**}(X)$). The first index is usually refered to
as the cohomological (resp. homological) \emph{degree} 
and the second one as the cohomological (resp. homological) \emph{twist}.
We also define the module of coefficients
attached to $\E$ as $\E^{**}=\E^{**}(S)$. \\
When $\E=\un$, we use the notations
$H^{**}(X)$ (resp. $H_{**}(X)$) for the cohomology (resp. homology) 
with coefficients in $\un$.
Finally, we simply put $A=H^{**}(S)$.
\end{df}
Remark that, from axiom (Kun)(a), $A$ is a bigraded ring.
Moreover, using the axiom (Stab), $A=H_{**}(S)$.
Thus, there are two bigraduations on $A$, one cohomological and the
other homological, and the two are exchanged as usual by a change of sign.
The tensor product of morphisms in $\T$ induces a structure
of left bigraded $A$-module on $\E^{**}(X)$ (resp. $\E_{**}(X)$).
There is a lot more algebraic structures on these bigraded groups that 
we have gathered in section \ref{sec:products}.

The axiom (Orient) gives a natural transformation
$$
c_1:\pic \rightarrow H^{2,1}
$$
of presheaves of \emph{sets} on $\sm$, or in other words,
an \emph{orientation} on the fundamental cohomology $H^{*,*}$ assocciated with 
the functor $\M$. 
In our setting, cohomology classes are morphisms in $\T$~: 
for any element $L \in \pic(X)$,
we view $c_1(L)$ both as a cohomology class, the \emph{first Chern class},
and as a morphism in $\T$.

\rem In the previous definition, we can replace the premotive $M(X)$ 
by any premotive $\cM$. This allows to define as usual 
the cohomology/homology of 
an (arbitrary) pair $(X,U)$ made by a smooth scheme $X$ 
 and a smooth subscheme $U$ of $X$.
Particular cases of this general definition is
the cohomology/homology of a smooth scheme $X$
 with support in a closed subscheme $Z$ 
and the reduced cohomology/homology
 associated with a pointed smooth scheme.
 
\subsection{Products} \label{sec:products}
Let $X$ be a smooth scheme and
$\delta:X \rightarrow X \times X$ its associated diagonal embedding.
Using axiom (Kun)(a) and functoriality, we get a morphism
 $\delta'_*:M(X) \rightarrow M(X) \otimes M(X)$.
Given two morphisms $x:M(X) \rightarrow \E$ and $y:M(X) \rightarrow \F$
in $\T$, we can define a product
$$
x \gcup y=(x \otimes y) \circ \delta'_*:
 M(X) \rightarrow \E \otimes \F.
$$

\num \label{products&ringed_spectra}
By analogy with topology, we will call \emph{ringed premotive} 
any premotive $\E$ equipped with a commutative monoid structure 
in the symmetric  monoidal category $\T$. This means we have
a product map $\mu:\E \otimes \E \rightarrow \E$ and a unit map
$\eta:\un \rightarrow \E$ satisfying the formal properties of a 
commutative monoid. \\
For any smooth scheme $X$ and any couple of integer $(n,p) \in \ZZ^2$,
the unit map induces morphisms
\begin{align*}
\varphi_X:&H^{n,p}(X) \rightarrow \E^{n,p}(X) \\
\psi_X:&H_{n,p}(X) \rightarrow \E_{n,p}(X)
\end{align*}
which we call the \emph{regulator maps}.

Giving such a ringed premotive $\E$, we define\footnote{
We do not indicate the commutativity isomorphisms for the tensor product 
  and the twists in the formulas to make them shorter.}
the following products~:
\begin{itemize}
\item \emph{Exterior products}~:
\begin{align*}
\E^{n,p}(X) \otimes \E^{m,q}(Y) & \rightarrow \E^{n+m,p+q}(X \times Y), \\
  \ (x,y) & \mapsto x \times y:=\mu \circ x \otimes y \\
\E_{n,p}(X) \otimes \E_{m,q}(Y) & \rightarrow \E_{n+m,p+q}(X \times Y), \\
  \ (x,y) & \mapsto x \times y:=(\mu \otimes 1_{X \times Y*}) \circ (x \otimes y)
\end{align*}
\item \emph{Cup-product}~:
$$
\E^{n,p}(X) \otimes \E^{m,q}(X) \rightarrow \E^{n+m,p+q}(X),
(x,x') \mapsto x \ncup x':=\mu \circ (x \gcup x').
$$
Then $\E^{**}$ is a bigraded ring
 and $\E^{**}(X)$ is a bigraded $\E^{**}$-algebra.
Moreover, $\E^{**}$ is a bigraded $A$-algebra
and the regulator map is a morphism of bigraded $A$-algebra.
\item \emph{Slant products}\footnote{For the first slant product defined here,
we took a slightly different covention than \cite[13.50(ii)]{Swi} in order to
obtain formula \eqref{alexander_duality}. Of course, 
the two conventions coincide up to the isomorphism
 $X \times Y \simeq Y \times X$.}~:
\begin{align*}
\E^{n,p}(X \times Y) \otimes \E_{m,q}(X) & \rightarrow \E^{n-m,p-q}(Y), \\
(w,x) & \mapsto w/x
   :=\mu \circ (1_\E \otimes w) \circ (x \otimes 1_{Y*}) \\
\E^{n,p}(X) \otimes \E_{m,q}(X \times Y) & \rightarrow \E_{m-n,q-p}(Y), \\
(x,w) & \mapsto x \backslash w
  :=(\mu \otimes 1_{Y*}) \circ (x \otimes 1_\E \otimes 1_{Y*}) \circ w.
\end{align*}
\item \emph{Cap-product}~:
$$
\E^{n,p}(X) \otimes \E_{m,q}(X) \rightarrow \E_{m-n,q-p}(X),
 (x,x') \mapsto x \ncap x':=x \backslash \big((1_\E \otimes \delta_*) \circ x'\big).
$$
\item \emph{Kronecker product}~:
$$
\E^{n,p}(X) \otimes \E_{m,q}(X) \rightarrow A^{n-m,p-q},
 (x,x') \mapsto \langle x,x' \rangle:=x/x'
$$
where $y$ is identified to a homology class in $\E_{m,q}(S \times X)$.
\end{itemize}
 
The regulator maps (cohomological and homological) are compatible 
with these products in the obvious way.

\begin{rem}
These products satisfy a lot of formal properties.
We will not use them in this text but we refer the interested reader to
\cite[chap. 13]{Swi} for more details 
(see more precisely 13.57, 13.61, 13.62).
\end{rem}

\num \label{product_with_support}
We can extend the definition of these products
to the cohomology of an open pair $(X,U)$. 
We refer the reader to \emph{loc. cit.} for this extension 
but we give details for the cup-product in the case of cohomology with supports
as this will be used in the sequel. \\
Let $X$ be a smooth scheme and $Z$, $T$ be two closed subschemes of $X$.
Then the diagonal embedding of $X/S$ induces using once again  axiom (Kun)(a) 
a morphism
 $\delta''_*:\M_{Z \cap T}(X) \rightarrow \M_Z(X) \otimes \M_T(X)$.
This allows to define a product of motives with support.
Given two morphisms $x:M_Z(X) \rightarrow \E$ and $y:M_T(X) \rightarrow \F$
in $\T$, we define
$$
x \gcup y=(x \otimes y) \circ \delta''_*:
 M_{Z \cap T}(X) \rightarrow \E \otimes \F.
$$
In cohomology, we also define the cup-product with support~:
$$
\E^{n,p}_Z(X) \otimes E^{m,q}_T(X)
 \rightarrow \E^{n+m,p+q}_{Z \cap T}(X), 
(x,y) \mapsto x \ncup_{Z,T} y=\mu \circ (x \otimes y) \circ \delta''_*.
$$
Note that considering the canonical morphism
$\nu_{X,W}:\E^{n,p}_W(X) \rightarrow \E^{n,p}(X)$, for any closed
subscheme $W$ of $X$, we obtain easily~:
\begin{equation} \label{eq:cup_with_support}
\nu_{X,Z}(x) \ncup \nu_{X,T}(y)=\nu_{X,Z \cap T}(x \ncup_{Z,T} y).
\end{equation}

\num \label{general_cup&slant}
Suppose now that $\E$ has no ring structure.
It nethertheless always has a module structure 
over the ringed premotive $\un$
-- given by the structural map (isomorphism)
$\eta:\un \otimes \E \rightarrow \E$.

This induces in particular a structure of left $H^{**}(X)$-module
on $\E^{**}(X)$ for any smooth scheme $X$.
Moreover,
 it allows to extend the definition of slant products and cap products.
Explicitely, this gives in simplified terms~:
\begin{itemize}
\item \emph{Slant products}~:
\begin{align*}
H^{n,p}(X \times Y) \otimes \E_{m,q}(X) & \rightarrow \E^{n-m,p-q}(Y), \\
 (w,y) & \mapsto w/y:=\eta \circ (1_\E \otimes w) \circ (x \otimes 1_{Y*})
\end{align*}
\item \emph{Cap-products}~:
$$
\E^{n,p}(X) \otimes H_{m,q}(X) \rightarrow \E_{m-n,q-p}(X),
 (x,x') \mapsto x \ncap x':=(x \otimes 1_{X*}) \circ \delta_* \circ x'.
$$
\end{itemize}
These generalized products will be used at the end of the article
to formulate duality with coefficients in $\E$ 
(cf paragraph \ref{explicit_duality_coh_hom}).

Note finally that, analog to the cap-product,
 we have a $H^{**}(X)$-module structure on $E_{**}(X)$
 that can be used to describe the projective bundle formula in
   $\E$-homology (cf formula (2) of \ref{cor:projbdl_formula}).
 
\subsection{Examples} \label{sec:htp_funct}

%
%
%

\subsubsection{Motives} \label{motives}

Suppose $S$ is a regular scheme. 
Below, we give the full construction of the category of
geometric motives of Voevodsky over $S$,
and indicate how to check the axioms of \ref{axioms}.
Note however we will give a full construction of this category, 
together with the category of motivic complexes and spectra,
over any noetherian base $S$ in \cite{CD3}.
Here, the reader can find all the details for the proof of the
axioms \ref{axioms} (especially axiom (Orient)).

For any smooth schemes $X$ and $Y$,
we let $c_S(X,Y)$ be the abelian group of
cycles in $X \times_S Y$ whose support is
finite equidimensional over $X$.
As shown in \cite[sec. 4.1.2]{Deg7}, this defines the morphisms
of a category denoted by $\smc$. The category
$\smc$ is obviously additive. It has a
symmetric monoidal structure defined by
the cartesian product on schemes and
by the exterior product of cycles on morphisms.

Following Voevodsky, 
we define the category of effective geometric motives 
$\DMgme$ as the pseudo-abelian envelope\footnote{Recall
that according to the result of \cite{BS}, the pseudo-abelian
envelope of a triangulated category is still triangulated}
of the Verdier triangulated quotient
$$
\mathrm K^b(\smc)/\T
$$
where $\mathrm K^b(\smc)$ is the category of
bounded complexes up to chain homotopy equivalence
and $\T$ is the thick subcategory generated
by the following complexes~:
\begin{enumerate}
\item For any smooth scheme $X$,
$$
\hdots 0 \rightarrow \AA^1_X \xrightarrow p X \rightarrow 0 \hdots
$$
with $p$ the canonical projection.
\item For any cartesian square of smooth schemes
$$
\xymatrix@=10pt{
W\ar^k[r]\ar_g[d] & V\ar^f[d] \\
U\ar^j[r] & X
}
$$
such that $j$ is an open immersion, $f$ is \'etale and
the induced morphism 
$f^{-1}(X-U)_{red} \rightarrow (X-U)_{red}$ is an isomorphism,
\begin{equation} \label{eq:cone}
\hdots 0 \rightarrow W
 \xrightarrow{
\text{\scriptsize $\begin{pmatrix} g \\ -k \end{pmatrix}$}} U \oplus V
  \xrightarrow{(j,f)} X \rightarrow 0 \hdots
\end{equation}
\end{enumerate}
Consider a cartesian square of immersions 
$$
\xymatrix@=10pt{W\ar^k[r]\ar_g[d]\ar@{}|\Delta[rd] & V\ar^f[d] \\ U\ar^j[r] & X}
$$
This defines a morphism of complexes in $\smc$~:
$$
\psi:\left\{
\raisebox{0.6cm}{
\xymatrix@=10pt{
\hdots 0\ar[r] & W\ar^k[r]\ar_g[d] & V\ar^f[d]\ar[r] & 0 \hdots \\
\hdots 0\ar[r] & U\ar^j[r] & X\ar[r] & 0 \hdots
}}\right.
$$
We let $M(\Delta)$ be the cone of $\psi$
 and see it as an object
of $\DMgme$. To fix the convention, we define this cone 
as the triangle \eqref{eq:cone} above.
With this convention, we define $\epsilon_\Delta$
as the following morphism~:
$$
\xymatrix@R=18pt@C=30pt{
\hdots 0\ar[r] 
 & W\ar_{-1}[d]\ar[r]
 & U \oplus V
     \ar^{\text{\tiny $\begin{pmatrix} 0 & 1 \\ 1 & 0 \end{pmatrix}$}}[d]
    \ar[r] 
 & X\ar^{+1}[d]\ar[r]
 & 0 \hdots \\
\hdots 0\ar[r] 
 & W\ar[r]
 & V \oplus U\ar[r]
 & X\ar[r]
 & 0 \hdots \\
}
$$
The reader can now check easily that the resulting functor 
 $M:\D \rightarrow \DMgme$,
satisfies all the axioms of \ref{axioms} except (Stab) and (Orient).
We let $\ZZ=M(S)$ be the unit object for the monoidal structure
of $\DMgme$.

To force axiom (Stab), we formally invert the motive $\ZZ(1)$
in the monoidal category $\DMgme$. This defines
the \emph{triangulated category of (geometric) motives}
denoted by $\DMgm$.
Remark that according to the proof of \cite[lem. 4.8]{V1},
the cyclic permutation of the factors of $\ZZ(3)$ is
the identity. This implies the monoidal structure on
$\DMgme$ induces a unique monoidal structure on $\DMgm$
such that the obvious triangulated functor 
$\DMgme \rightarrow \DMgm$
is monoidal.
Now, the functor $M:\D \rightarrow \DMgm$ still satisfies
all axioms of \ref{axioms} mentioned above but also axiom (Stab).

To check the axiom (Orient), it is sufficient to construct a natural 
application
$$
\pic(X) \rightarrow \Hom_{\DMgme}(\M(X),\ZZ(1)[2]).
$$
We indicate how to obtain this map. 
Note moreover that, from the following construction,
 it is a morphism of abelian group.

Still following Voevodsky, we have defined in \cite{Deg7} 
the abelian category of sheaves with transfers over $S$, 
denoted by $\ftr$.
We define the cateogy $\DMe$ of motivic complexes as the
$\AA^1$-localization of the derived category of $\ftr$.
The Yoneda embedding $\smc \rightarrow \ftr$ sends smooth schemes
to \emph{free} abelian groups. For this reason,
the canonical functor
$$
\DMgme \rightarrow \DMe
$$
is fully faithful.
Let $\GG$ be the sheaf with transfers which associates
to a smooth scheme its group of invertible (global) functions.
Following Suslin and Voevodsky (cf also \cite[2.2.4]{Deg5}),
we construct a morphism in $\DMe$~:
$$
\GG \rightarrow M(\GG)=\ZZ \oplus \ZZ(1)[1]
$$
This allows to define the required morphism~:
\begin{align*}
\pic(X)& =H^1_\mathrm{Nis}(X;\GG) \simeq \Hom_{\DMe}(\M(X),\GG[1]) \\
&\rightarrow \Hom_{\DMe}(\M(X),\ZZ(1)[2])
   \simeq \Hom_{\DMgme}(\M(X),\ZZ(1)[2]).
\end{align*}
The first isomorphism uses that the sheaf $\GG$ is $\AA^1$-local
and that the functor forgetting transfers is exact
(cf \cite[prop. 2.9]{Deg7}).

\subsubsection{Stable homotopy exact functors}

In this example, $S$ is any noetherian scheme.
For any smooth scheme $X$, we let $X_+$ be the pointed sheaf of sets on
$\sm$ represented by $X$ with a (disjoint) base point added.

Consider an immersion $U \rightarrow X$ of smooth schemes.
We let $X/U$ be the pointed sheaf of sets which is the cokernel
of the pointed map $U_+ \rightarrow X_+$.

Suppose moreover given a square $\Delta$ as in \eqref{std}.
Then we obtain an induced morphism of pointed sheaves of sets
$V/W \rightarrow X/U$ which is injective.
We let $\frac{X/U}{V/W}$ be the cokernel of this monomorphism.
Thus, we obtain a cofiber sequence in $\H$
$$
V/W \rightarrow X/U \rightarrow \frac{X/U}{V/W}
 \xrightarrow{\partial_\Delta} S^1_s \wedge V/W.
$$
Moreover, the functor
$$
\D \rightarrow \H, \Delta \mapsto \frac{X/U}{V/W}
$$
satisfies axioms (Add), (Htp), (Exc) and (Kun) from \cite{MV}.

Consider now the stable homotopy category of schemes $\SH(S)$
(cf \cite{Jar})
together with the infinite suspension functor
$$
\Sigma^\infty:\H \rightarrow \SH(S).
$$
The category $\SH(S)$ is a triangulated symmetric monoidal category.
The canonical functor $\D \rightarrow \SH(S)$ satisfies all the
axioms of \ref{axioms} except axiom (Orient).
In fact, (Loc) and (Sym) follows easily from the definitions
and (Stab) was forced in the construction of $\SH(S)$.

Suppose we are given a triangulated symmetric monoidal category $\T$
together with a triangulated symmetric monoidal functor
$$
R:\SH(S) \rightarrow \T.
$$
This induces a canonical functor
$$
M:\D \rightarrow \T, \frac{X/U}{V/W} \mapsto
 M\left(\frac{X/U}{V/W}\right)
  :=R\left(\Sigma^\infty \left(\frac{X/U}{V/W}\right)\right)
$$
and $(M,\T)$ satisfies formally all the axioms \ref{axioms} except (Orient).

Let $B\GG$ be the classifying space of $\GG$ defined in \cite[section 4]{MV}.
It is an object of the simplicial homotopy category $\Hs$
 and from \emph{loc. cit.}, proposition 1.16, $$Pic(X)=\Hom_{\Hs}(X_+,B\GG).$$ 

Let $\pi:\Hs \rightarrow \H$ be the canonical $\AA^1$-localisation functor.
Applying proposition 3.7 of \emph{loc. cit.}, $\pi(B\GG)=\PP^\infty$
where $\PP^\infty$ is the tower of pointed schemes
$$
\PP^1 \rightarrow ..
 \rightarrow \PP^n \xrightarrow{\iota_n} \PP^{n+1} \rightarrow ...
$$
made of the inclusions onto the corresponding hyperplane at infinity.
We let $\M(\PP^\infty)$ (resp. $\Mr(\PP^\infty)$) be the ind-object
of $\T$ obtained by applying $\M$ (resp. $\Mr$) on each degree of the
tower above.

Using this, we can define an application
\begin{align*}
\rho_X:& Pic(X)=\Hom_{\Hs}(X_+,B\GG) \\
 & \rightarrow \Hom_{\H}\left(X_+,\pi(B\GG)\right)=\Hom_{\H}(X_+,\PP^\infty) \\
 & \rightarrow \Hom_\T\left(\M(X),\Mr(\PP^\infty)\right)
\end{align*}
where the last group of morphisms denotes by abuse of notations the 
group of morphisms in the category of ind-objects of $\T$ -- and similarly in what
follows.

\rem \label{pfou}
Note that the sequence $(\L_n)_{n \in \NN}$ of line bundles is sent by
$\rho_{\PP^\infty}$ to the canonical projection
$\M(\PP^\infty) \rightarrow \Mr(\PP^\infty)$ -- this follows
from the construction of the isomorphism of \emph{loc. cit.}, prop. 1.16.

Recall that $\un(1)[2]=\Mr(\PP^1)$ in $\T$.
Let $\pi:\M(\PP^1) \rightarrow \Mr(\PP^1)$ be the canonical
projection and
 $\iota:\PP^1 \rightarrow \PP^\infty$ be the canonical 
 morphism of pointed ind-schemes.
We introduce the following two sets~:
\begin{enumerate}
\item[$(S_1)$] The transformations 
$c_1:\pic(X) \rightarrow \Hom_\T(\M(X),\un(1)[2])$
natural in the smooth scheme $X$
such that $c_1(\L_1)=\pi$.
\item[$(S_2)$] The morphisms $c'_1:\Mr(\PP^\infty) \rightarrow \Mr(\PP^1)$
such that $c'_1 \circ \iota_*=1$.
\end{enumerate}
We define the following applications~:
\begin{enumerate}
\item $\varphi:(S_1) \rightarrow (S_2)$. \\
Consider an element $c_1$ of $(S_1)$.
The collection $\big(c_1(\L_n)\big)_{n \in \NN}$
defines a morphism $\M(\PP^\infty) \rightarrow \Mr(\PP^1)$.
Moreover, the restriction of this latter morphism
$\Mr(\PP^\infty) \rightarrow \Mr(\PP^1)$
is obviously an element of $(S_2)$,
denoted by $\varphi(c_1)$.
\item $\psi:(S_2) \rightarrow (S_1)$. \\
Let $c'_1$ be an element of $(S_2)$. For any smooth scheme $X$, 
we define 
$$
\psi(c'_1):\pic(X) \xrightarrow{\rho_X} \Hom_\T\left(\M(X),\Mr(\PP^\infty)\right)
 \xrightarrow{c'_{1*}} \Hom_\T\left(\M(X),\Mr(\PP^1)\right).
$$
Using remark \ref{pfou}, we check easily that $\psi(c'_1)$ belongs to $(S_1)$.
\end{enumerate}
The following lemma is obvious from these definitions~:
\begin{lm}
Given, the hypothesis and definitions above,
 $\varphi \circ \psi=1$.
\end{lm}
Thus, an element of $(S_2)$ determines canonically an element of $(S_1)$.
This gives a way to check the axiom (Orient) for a functor $R$ as above.
Moreover, 
 we will see below (cf paragraph \ref{fgl}) that given an element of $(S_1)$,
  we obtain a canonical isomorphism $H^{**}(\PP^\infty)=A[[t]]$
   of bigraded algebra, $t$ having bidegree $(2,1)$.
Then elements of $(S_2)$ are in bijection with the set of generators
of the the bigraded algebra $H^{**}(\PP^\infty)$.
Thus in this case, elements of $(S_2)$ are equivalent to 
\emph{orientations} of the cohomology $H^{**}$ in the classical sense
of algebraic topology.

\begin{ex}
\begin{enumerate}
\item Let $S=\spec k$ be the spectrum of a field,
 or more generally any regular scheme.
In \cite[2.1.4]{CD2},
 D.C. Cisinski and the author introduce the notion of
mixed Weil theory (and more generally of stable theory)
as axioms for cohomology theories on smooth $S$-schemes
which extends the classical axioms of Weil.
Examples of such cohomology theories are algebraic De Rham 
cohomology if $k$ has characteristic $0$,
rigid cohomology if $k$ has caracteristic $p$
and \'etale $l$-adic cohomology in any case,
$l$ being invertible in $k$ (cf part 3 of \emph{loc. cit.}).
To a mixed Weil theory (or more generally a stable theory)
is associated a commutative ring spectrum
(cf \emph{loc. cit.} 2.1.5) and a triangulated closed 
symmetric monoidal category $D_{\AA^1}(S,\mathcal E)$ --
which is obtained by localization of a derived category.
By construction (see \emph{loc. cit.} (1.5.3.1)), we have
a triangulated monoidal symmetric functor
$$
\SH(S) \rightarrow D_{\AA^1}(S,\mathcal E).
$$
In \emph{loc. cit.} 2.2.9,
we associate a canonical element of the set $(S_2)$ for this functor.
Thus the resulting functor $\D \rightarrow D_{\AA^1}(S,\mathcal E)$
 satisfies all the axioms of \ref{axioms}.
\item Consider a noetherian scheme $S$
and the model category of symmetric $T$-spectra $Sp_S$ over $S$
defined by R.~Jardine in \cite{Jar}.
It is a cofibrantly generated, symmetric monoidal model category 
which satisfies the monoid axiom of \cite[3.1]{SwSh} 
(cf \cite[4.19]{Jar} for this latter fact).

A commutative monoid $\E$ in the category $Sp_S$ will be called 
a (homotopy) \emph{coherent} ring spectrum.
Given such a ring spectrum, according to \cite[4.1(2)]{SwSh},
the category of $\E$-modules in 
the symmetric monoidal category $Sp_S$
carries a structure of a cofibrantly generated, 
symmetric monoidal model category such that the pair
of adjoint functors $(F,\mathcal O)$
given by the free $\E$-module functor and the
obvious forgetful functor is a Quillen adjunction.
We denote by $\SH(S;\E)$ the associated homotopy category
and consider the left derived free $\E$-module functor
$$
\SH(S) \rightarrow \SH(S;\E).
$$
It is a triangulated symmetric monoidal functor.
Then, as indicated in the previous remark, 
 an element of $(S_2)$ relative to this functor is
  equivalent to an orientation on the ring spectrum 
   $\E$ in the classical sense (see \cite[3.1]{Vez}).

The basic example of such a ring spectrum is the cobordism ring 
spectrum $\MGl$. Indeed, $\MGl$ has a structure of a coherent
ring spectrum in our sense and is evidently oriented 
(see \cite[1.2.3 and 2.1]{PPR} for details).
Thus the homotopy category $\SH(S;\MGl)$
of $\MGl$-modules satisfies the axioms \ref{axioms}.

Another example is given by the spectrum $\BGL$ 
introduced by Voevodsky in \cite[par. 6.2]{V2}.
According to \emph{loc. cit.}, th. 6.9, it
represents the homotopy invariant algebraic K-theory defined
by Weibel (cf \cite{Wei}). 
However, it is not at all clear to get a coherent structure
on the ring spectrum $\BGL$ with the definition given in 
 \emph{loc. cit}.
To obtain such a coherent ring structure on $\BGL$
we invoke a recent result of Gepner and Snaith
which construct a coherent ring spectrum homotopy equivalent
to $\BGL$ in \cite[5.9]{GS}.
\end{enumerate}
\end{ex}

%% file: chern.tex
\section{Chern classes}

\subsection{The projective bundle theorem}

Let $X$ be a smooth scheme
 and $P$ be a projective bundle over $X$ of rank $n$. 
We denote by $p:P \rightarrow X$ the canonical projection
 and by $\L$ the canonical line bundle on $P$.
Put $c=c_1(\L):\M(P) \rightarrow \un(1)[2]$.
We can define a canonical map~:
$$
\epsilon_{P}:=\sum_{0 \leq i \leq n} p_* \gcup c^i:
M(P) \rightarrow \bigoplus_{0 \leq i \leq n} M(X)(i)[2i]
$$
Consider moreover an open subscheme $U \subset X$, $P_U=P \times_X U$.
We let $\pi:P/P_U \rightarrow X/U$ be the canonical projection
 and $\nu:P/P_U \rightarrow (P \times P)/(P \times P_U)$ the
  morphism induced by the diagonal embedding and the graph of 
   the immersion $P_U \rightarrow P$.
Using the product of motives with support 
 (cf \ref{product_with_support}),
  we also define a canonical map~:
$$
\epsilon_{P/P_U}:=\sum_{0 \leq i \leq n} \pi_* \gcup c^i:
M(P/P_U) \rightarrow \bigoplus_{0 \leq i \leq n} M(X/U)(i)[2i]
$$
\begin{lm}
Using the above notations, the following diagram is commutative~:
$$
\xymatrix@R=16pt@C=9pt{
M(P_U)\ar[r]\ar_{\epsilon_{P_U}}[d]\ar@{}|{(1)}[rd]
 & M(P)\ar[r]\ar_{\epsilon_P}[d]\ar@{}|{(2)}[rd]
 & M(P/P_U)\ar[r]\ar_{\epsilon_{P/P_U}}[d]\ar@{}|{(3)}[rd]
 & M(P_U)[1]\ar^{\epsilon_{P_U}}[d] \\
\bigoplus_i M(U)(i)[2i]\ar[r]
 & \bigoplus_i M(X)(i)[2i]\ar[r]
 & \bigoplus_i M(X/U)(i)[2i]\ar[r]
 & \bigoplus_i M(U)(i)[2i+1]
}
$$
where the top (resp. bottom) line is the distinguished 
triangle (resp. sum of distinguished triangles)
 obtained using \emph{(Loc)} (resp. and tensoring with $\un(i)[2i]$).
\end{lm}
\begin{proof}
Coming back to the definition of product and product with supports,
 squares (1) and (2) are commutative by functoriality of $\M$.
For square (3), besides this functoriality, 
we have to use axiom (Kun)(b).
\end{proof}

\begin{thm} \label{th:projbdl}
With the above hypothesis and notations,
 the morphism $\epsilon_{P}:M(P) \rightarrow \bigoplus_{0 \leq i \leq n} M(X)(i)[2i]$
 is an isomorphism in $\T$.
\end{thm}
\begin{proof}
Consider an open cover $X=U \cup V$, $W=U \cap V$. 
Assume that $\epsilon_{P_U}$, $\epsilon_{P_V}$ and
$\epsilon_{P_W}$ are isomorphisms.
Then according to the previous lemma, $\epsilon_{P_V/P_W}$ is
an isomorphism. Using the compatibility of the first
Chern class with pullback, we obtain a commutative diagram
$$
\xymatrix@C=16pt@R=10pt{
M(P_V/P_W)\ar[r]\ar_{\epsilon_{P_V/P_W}}[d]
 & M(P/P_U)\ar^{\epsilon_{P/P_U}}[d] \\
\bigoplus_i M(V/W)(i)[2i]\ar[r]
 & \bigoplus_i M(X/U)(i)[2i]
}
$$
where the horizontal maps are obtained by functoriality.
According to axiom (Exc), these maps are isomorphisms which
implies $\epsilon_{P/P_U}$ is an isomorphism.
Applying ance again the previous lemma, we deduce that
$\epsilon_P$ is an isomorphism. \\
This reasoning shows that we can argue locally on $X$
and assume $P$ is trivializable as a projective bundle over $X$. 
Then, 
as the map depends only on the isomorphism class of the projective bundle $P$,
we can assume $P=\PP^n_X$.
Finally, by property (Kun)(a),
$\epsilon_{\PP^n_X}=M(X) \otimes \epsilon_{\PP^n}$ and we can assume $X=S$.
Put simply $\epsilon_n=\epsilon_{\PP^n}$.

For $n=0$, the statement is trivial. Assume $n>0$.
Recall we consider the scheme $\PP^n$ pointed by the infinite point.
The morphism $\epsilon_n$
induces a map
$\Mr(\PP^n) \rightarrow \oplus_{0 < i \leq n} \Mr(\PP^1)^{\otimes,i}$
still denoted by $\epsilon_n$
and we have to prove this later is an isomorphism.
Put $c_{1,n}=c_1(\L_n)$ for any integer $n \geq 0$.

The canonical inclusion $\PP^{n-1} \rightarrow \PP^n-\{0\}$ is the zero section 
of a vector bundle. 
For any integer $i \in [1,n]$, we put
 $U_i=\{(x_1,...,x_n) \in \AA^n \mid x_i \neq 0\}$
  considered as an open subscheme of $\AA^n$.
We obtain the canonical isomorphism denoted by $\tau_n$~:
\begin{align*}
\M(\PP^n/\PP^{n-1}) & \stackrel{(1)}{\simeq} \M(\PP^n/\PP^n-\{0\})
 \stackrel{(2)}\simeq \M(\AA^n/\AA^n-\{0\}) \\
 & =\M(\AA^n/\cup_i U_i)
 \stackrel{(3)}=\M(\AA^1/\AA^1-\{0\})^{\otimes,n}
  \stackrel{(4)}\simeq \Mr(\PP^1)^{\otimes,n}
\end{align*}
where (1) follows from (Htp) and (Loc), (2) from (Exc), (3) from (Kun)(a)
and (4) from (Exc), (Htp) and (Loc).

Consider the following diagram
\begin{equation} \label{diag_proof_prjbdlthm}
\xymatrix@R=14pt@C=25pt{
\Mr(\PP^{n-1})\ar^{\iota_{n-1*}}[r]\ar_{\epsilon_{n-1}}[d]\ar@{}|{(a)}[rd]
 & \Mr(\PP^n)\ar^/-7pt/{\pi_n}[r]\ar_{\epsilon_{n}}[d]\ar@{}|{(b)}[rd]
 & \M(\PP^n/\PP^{n-1})\ar^{\tau_n}[d] \\
\oplus_{0 < i < n} \Mr(\PP^1)^{\otimes,i}\ar[r]
 & \oplus_{0 < i \leq n} \Mr(\PP^1)^{\otimes,i}\ar[r]
 & \Mr(\PP^1)^{\otimes,n}
}
\end{equation}
where $\iota_{n-1}$ is the canonical inclusion, 
$\pi_n$ is the obvious morphism obtained by functoriality in $\D$,
and the bottom line is made up of the evident split distinguished triangle.
We prove by induction on $n>0$ the following statement~:
\begin{equation} \label{proof_projbdlthm}
\begin{split}
\text{(i) } & \iota_{n-1*} \text{ is a split monomorphism.} \\
\text{(ii) } & c_{1,n-1}^n=0 \text{ which means square (a) is commutative.} \\
\text{(iii) } & c_{1,n}^n=\tau_n \pi_n \text{ which means square (b) is commutative.} \\
\text{(iv) } & \epsilon_n \text{ is an isomorphism.}
\end{split}
\end{equation}
For $n=1$, this is obvious as (iii) is a part of axiom (Orient).

The induction relies on the following lemma due to Morel.
\begin{lm}
Let $\delta_n:\PP^n \rightarrow (\PP^n)^n$ be the iterated $n$-th diagonal
of $\PP^n/S$ and denote by $\delta_{n*}:\Mr(\PP^n) \rightarrow \Mr(\PP^n)^{\otimes,n}$
the morphism induced by $\delta_n$ and axiom \emph{(Kun)(a)}. 
Let $\iota_{1,n}:\PP^1 \rightarrow \PP^n$ be the canonical inclusion. \\
\indent Then the following square commutes~:
$$
\xymatrix@=12pt{
\Mr(\PP^n)\ar^/-4pt/{\delta_{n*}}[r]\ar_{\pi_n}[d]
 & \Mr(\PP^n)^{\otimes,n} \\
\M(\PP^n/\PP^{n-1})\ar^/2pt/{\tau_n}[r]
 & \Mr(\PP^1)^{\otimes,n}.\ar_{(\iota_{1,n*})^{\otimes,n}}[u]
 }
$$
\end{lm}
Consider an integer $i \in [1,n]$ and let $\bar U_i$ be
 the open subscheme of $\PP^n$ made of points $(x_1:...:x_n:x_{n+1})$
such that $x_i \neq 0$ and put $\Omega_i=\PP^{i-1} \times U_i \times \PP^{n-i}$. \\
\indent 
We consider the following commutative diagram~:
$$
\xymatrix@R=12pt@C=20pt{
\Mr(\PP^n)\ar^/-12pt/{(1)}[r]\ar^{\pi_n}[d]
 & \M\big((\PP^n)^n/\cup_i \Omega_i\big) & \\
\M(\PP^n/\PP^{n-1})\ar^\sim[r]
 & \M\big(\PP^n/\cup_i \bar U_i\big)\ar[u]
 & \M(\AA^n/\cup_i U_i)\ar_/-6pt/{(2)}[lu]\ar_\sim[l]
 }
$$
where the map (1) is induced by $\delta_n$, 
the maps on the lower horizontal line are isomorphisms given
 respectively by the inclusions
  $\PP^{n-1} \subset \cup_i \bar U_i$ and $U_i \subset \bar U_i$. \\
\indent
Consequently, the map (2) is induced by the restriction of $\delta_n$.
However, this map is $\AA^1$-homotopic to the product 
$\iota^{(1)}\times ... \times \iota^{(n)}$ where
$\iota^{(i)}:\AA^1 \rightarrow \PP^n$ is the embedding defined by $\iota^{(i)}(x)=(x_1:...:x_{n+1})$ with $x_j=0$ if $j \notin \{i,n+1\}$, $x_i=x$, $x_{n+1}=1$. It follows from property (Htp) and (Kun)(a) that
the map (2) is equal to the morphism
$$
\M(\AA^1/\AA^1-\{0\})^{\otimes,n}
 \xrightarrow{\iota^{(1)}_* \otimes ... \otimes \iota^{(n)}_*}
  \M(\PP^n/\bar U_1) \otimes ... \otimes \M(\PP^n/\bar U_n).
$$
Note finally the scheme $\bar U_i \simeq \AA^n$ is contractible
and, from property (Htp), the corresponding map 
$\iota^{(i)}_*:\M(\AA^1/\AA^1-\{0\}) \rightarrow \Mr(\PP^n)$
does not depend on the integer $i$. 
Thus the preceding commutative diagram together with the identifications 
just described allows to conclude.

With that lemma in hand, we conclude as follows. 
Suppose the property \eqref{proof_projbdlthm} is true for $n-1$. \\
\noindent
The composite map $\left(\sum_{0<i<n} p_* \gcup c_n^i\right) \circ \iota_{n-1}$
is equal to $\epsilon_{n-1}$ as $c_n \circ \iota_{n-1*}=c_{n-1}$. 
This shows \eqref{proof_projbdlthm}(i). 
Then, the preceding lemma implies properties (ii) and (iii).
Now, using (Loc) and (Sym), the upper horizontal line of 
diagram \eqref{diag_proof_prjbdlthm} is a split distinguished triangle which
concludes.
\end{proof}

Using axiom (Stab), we obtain the following corollary~:
\begin{cor}
\label{cor:projbdl_formula}
Consider the hypothesis and notations of the previous theorem. \\
Then $H^{**}(P)$ is a free $H^{**}(X)$-module with base $1,...,c^n$.

Let $\E$ be a motive.
\begin{enumerate}
\item The map
$$
\E^{**}(X) \otimes_{H^{**}(X)} H^{**}(P) \rightarrow \E^{**}(P),
x \otimes \lambda \rightarrow \lambda.p^*(x)
$$
is an isomorphism. If moreover $\E$ has a ringed motive structure,
 it is an isomorphism of $\E^{**}(X)$-algebra.
\item Considering the $H^{**}(X)$-module structure on $\E_{**}(X)$
 (cf the end of \ref{general_cup&slant}), the map
$$
\E_{**}(P) \rightarrow \bigoplus_{0 \leq i \leq n} \E_{**}(X),
 \varphi \mapsto \sum_i c^i \ncap p_*(\varphi)
$$
is an isomorphism.
\end{enumerate}
\end{cor}

\begin{rem}
It can be seen actually that the first assertion of this corollary
is equivalent to the fact $H^{n,m}(\M(X)(r))=H^{n,m-r}(\M(X))$
which is a weak form of the stability axiom (Stab).
\end{rem}

A corollary of the projective bundle theorem is the following result,
classical in topology and first exploited in the homotopy category of schemes
by Morel~:
\begin{cor}
Consider the permutation isomorphism
$\eta:\un(1) \otimes \un(1) \rightarrow \un(1) \otimes \un(1)$
in the symmetric monoidal category $\T$.
Then $\eta=1$.

Let $\E$ be a ringed motive and $X$ be a smooth scheme. \\
 For any $x \in \E^{n,p}(X)$ and $y \in \E^{m,q}(X)$,
$x \ncup y=(-1)^{nm}.y \ncup x$.
\end{cor}
\begin{proof}
In general, for $x \in \E^{n,p}(X)$ and $y \in \E^{m,q}(X)$, we have
$x \ncup y=(-1)^{nm} \eta^{pq}.y \ncup x$.
In particular, when $X=\PP^2$ and $c=c_1(\L_2)$, we get $c^2=\eta.c^2$.
This implies $\eta=1$ from the previous corollary and the other assertion
follows.
\end{proof}

\subsection{The associated formal group law}

\num \label{fgl}
Put $H^{**}(\PP^\infty)=\plim{n>0} H^{**}(\PP^n)$.
Then corollary \ref{cor:projbdl_formula} together with the relation
\eqref{proof_projbdlthm}(ii) implies $H^{**}(\PP^\infty)=A[[c]]$,
free ring of power series over $A$ with generator
 $c=(c_{1,n})_{n>0}$ of degree $(2,1)$.
Moreover, $H^{**}(\PP^\infty \times \PP^\infty)=A[[x,y]]$.
 
Consider the Segre embeddings
 $\sigma_{n,m}:\PP^n \times \PP^m \rightarrow \PP^{n+m+nm}$
for $(n,m) \in \NN^2$
and the induced map on ind-schemes
 $\sigma:\PP^\infty \times \PP^\infty \rightarrow \PP^\infty$.
Then the map
$\sigma^*:H^{**}(\PP^\infty)
  \rightarrow H^{**}(\PP^\infty \times \PP^\infty)$
corresponds to a power series
$$
F=\sum_{i,j} a_{ij}.x^iy^j \in A[[x,y]]
$$
which according to the classical situation\footnote{Recall these properties follows 
from the fact that the coefficients $a_{ij}$ for $i\leq n$, $j\leq m$ 
are determined by the map $\sigma_{n,m}$. The reader can find a more detailed proof 
in \cite{LM}, proof of cor. 10.6.}
in algebraic topology
is a commutative formal group law~:
$$
F(x,0)=x, F(x,y)=F(y,x), F(x,F(y,z))=F(F(x,y),z).
$$
For any $(i,j) \in \NN^2$, the element $a_{i,j} \in A$
 is of homological degree $(2(i+j-1),i+j-1)$
 and the first two relations above are equivalent to
$$
a_{0,1}=1, a_{0,i}=0 \text{ if } i \neq 1, a_{i,j}=a_{j,i}.
$$
Recall also there is a formal inverse associated to $F$,
 that is a formal power series $m \in A[[x]]$ such
  that $F(x,m(x))=0$.
We can find the notation $x+_Fy=F(x,y)$ in the litterature.
For an integer $n \geq 0$,
we put $[n]_F \cdot x=x+_F...+_Fx$,
that is the power series in $x$ equal to
the formal $n$-th addition of $x$ with itself.
These notations will be fixed through the rest of the article.

\begin{prop} \label{nilpotence&FGL}
Let $X$ be a smooth scheme.
\begin{enumerate}
\item For any line bundle $L/X$,
 the class $c_1(L)$ is nilpotent in $H^{**}(X)$.
\item Suppose $X$ admits an ample line bundle.
 For any line bundles $L,L'$ over $X$,
$$
c_1(L_1 \otimes L_2)=F(c_1(L_1),c_1(L_2)) \in H^{2,1}(X).
$$
\end{enumerate}
\end{prop}
\begin{proof}
For the first point, we first remark the question is local in $X$.
As $X$ is noetherian,
 we are reduced by induction to consider an open covering $X=U \cup V$, 
 such that $c_1(L_U)$ (resp. $c_1(L_V)$) is nilpotent
  in $H^{**}(U)$ (resp. $H^{**}(V)$)
   where $L_U$ (resp. $L_V$) is the restrion of $L$ to $U$ (resp. $V$).
Let $n$ (resp. $m$) be the order of nilpotency of $c_1(L_U)$
 (resp. $c_1(L_V)$).
Let $Z=X-U$ (resp. $T=X-V$) and consider the canonical morphism
$\nu_{X,W}:H^{**}_W(X) \rightarrow H^{**}(X)$ for $W=Z,T$.
From axiom (Loc), there exists a class $a$ (resp. $b$) in 
$H^{**}_Z(X)$ (resp. $H^{**}_T(X)$ such that
$a=c_1(L)^n$ (resp. $b=c_1(L)^m$).
As $Z \cap T=\emptyset$, axiom (Loc) implies $a \ncup_{Z,T} b=0$.
Thus, relation \eqref{eq:cup_with_support} implies
 $c_1(L)^{n+m}=0$ as wanted. \\
The first point follows, as $\L$ is locally trivial and the Chern
class of a trivial line bundle is $0$ by definition.

For the second point, the assumption implies there is 
a torsor $\pi:X' \rightarrow X$ under a vector bundle over $X$
such that $X'$ is affine.
From axioms (Htp') and (Exc), we obtain that
 $\pi_*:\M(X') \rightarrow \M(X)$ is an isomorphism.
Thus we are reduced to the case where $X$ is affine. \\
Then, the line bundle $L$ is generated by its section
 (cf \cite[5.1.2,e]{EGA2}),
which means there is a closed immersion
$L \xrightarrow \iota \AA^{n+1}_X$ where $n+1$ is the cardinal
of a generating family. In particular, we get a morphism
$$
f:X \simeq \PP(L) \xrightarrow \iota \PP^n_X \rightarrow \PP^n
$$
with the property that $f^{-1}(\L_n)=L$.
In the same way, we can find a morphism $g:X \rightarrow \PP^m$
such that $g^{-1}(\L_m)=L'$. We consider the morphism
$$
\varphi:X \rightarrow X \times X
 \xrightarrow{f \times g} \PP^n \times \PP^m
  \xrightarrow{\sigma_{n,m}} \PP^{nm+n+m}.
$$
By construction, $\varphi^{-1}(\lambda_{nm+n+m})=L \otimes L'$
and this concludes, computing in two ways the Chern class
 of this line bundle.
\end{proof}

Consider a ringed motive $\E$ 
 with regulator map $\varphi:H^{**} \rightarrow \E^{**}$. \\
The map $\sigma^*:\E^{**}(\PP^\infty)
  \rightarrow \E^{**}(\PP^\infty \times \PP^\infty)$
defines a formal group law $F_\E$ with coefficients in $\E^{**}$
and $F_\E=\sum_{i,j} \varphi_S(a_{i,j}).x^iy^j$.
Thus the regulator map induces a morphism of formal group law
$(A,F) \rightarrow (\E^{**},F_\E)$.

\begin{rem}
In case $F$ is the additive formal group law, $F(x,y)=x+y$, 
for any ringed motive $\E$, 
$F_\E$ is the additive formal group law. 
This is the case for example if $\T=\DMgm$
 or $\T$ is the category of modules over a mixed Weil theory. \\
When $F$ is the universal multiplicative formal group law
$F=x+y+\beta.xy$, the obstruction for $F_\E$ to be additive
is the element $\varphi(\beta)$.
\end{rem}

\subsection{Higher Chern classes}

We now follow the classical approach of Grothendieck
 to define higher Chern classes.
Consider a vector bundle $E$ of rank $n>0$ over a smooth scheme $X$.
Let $\L$ (resp. $p$) be the canonical invertible sheaf 
(resp. projection) of the projective bundle $\PP(E)/X$.
From corollary \ref{cor:projbdl_formula},
there are unique classes $c_i(E) \in H^{2i,i}(X)$ for $i=0,...,n$,
such that
\begin{equation} \label{eq_chern}
\sum_{i=0}^{n} p^*(c_i(E)) \ncup \big(-c_1(\L)\big)^{n-i}=0
\end{equation}
and $c_0(E)=1$.
\begin{df} \label{df:Chern}
With the above notations, we call $c_i(E)$ the $i$-th Chern class of $E$.
We also put $c_i(E)=0$ for any integer $i>n$.
\end{df}

\rem In the case $n=1$, due to our choice of conventions, $\L=p^{-1}(E)$.
The previous relation is not a definition, but a tautology.
This enlighten particularly our choice of sign in the previous relation.
Besides, when $c_1(\L^\vee)=-c_1(\L)$ (in particular when the formal group law
F is additive),
 relation \eqref{eq_chern} agrees precisely with that of \cite{Gro}.

\rem Considering any ringed motive $\E$, with regulator map
  $\varphi:H \rightarrow \E$, $\varphi \circ c_i$ defines
   Chern classes for cohomology with coefficients in $\E$.
When no ringed structure is given on $\E$, we still get an action
of the former Chern classes on the $\E$-cohomology using the action of
the cohomology theory $H$ (cf \ref{general_cup&slant}).

The Chern classes are obviously functorial with respect to pullback
 and invariant under isomorphism of vector bundles.
They also satisfy the Whitney sum formula ; we recall the proof to the reader
as it uses the axiom (Kun)(a) in an essential way.
\begin{lm} \label{additivity_Chern}
Let $X$ be a smooth scheme and consider an exact sequence
of vector bundles over $X$~:
$$
0 \rightarrow E' \rightarrow E \rightarrow E'' \rightarrow 0
$$

Then for any $k \in \NN$,
 $c_k(E)=\sum_{i+j=k} c_i(E') \ncup c_j(F'')$.
\end{lm}
\begin{proof}
By compatibility of Chern classes with pullback
 we can assume the sequence above is split.
Let $n$ (resp. $m$) be the rank of $E'/X$ (resp. $E''/X$).
Put $P=\PP(E)$ and consider $c \in H^{2,1}(P)$ 
(resp. $p:P \rightarrow X$)
the first Chern class of the canonical line bundle on
(resp. canonical projection of) $P/X$. \\
Put $a=\sum_{i=0}^n p^*(c_i(E')).c^{n-i}$
 and $b=\sum_{j=0}^m p^*(c_j(E'')).c^{m-j}$
as cohomology classes in $H^{**}(P)$. We have to prove $a \ncup b=0$. \\
Consider the canonical embeddings $i:\PP(E') \rightarrow P$
and $j:P-\PP(E'') \rightarrow P$. 
Then $i^*(a)=0$ which implies by property (Htp') that $j^*(a)=0$. 
Thus there exists $a' \in H^{*,*}_{\PP(E'')}(P)$ such that $a=\nu_F(a')$
where $\nu_F:H^{*,*}_{\PP(E'')}(P) \rightarrow H^{**}(P)$ is 
the canonical morphism. 
Similarly, there exists $b' \in H^{*,*}_{\PP(E')}(P)$ such that $b=\nu_E(b')$
where $\nu_E:H^{*,*}_{\PP(E')}(P) \rightarrow H^{**}(P)$ is the canonical morphism.
Then, relation \eqref{eq:cup_with_support} allows\footnote{This is where axiom (Kun)(a) 
is used.} to conclude because
$\PP(E') \cap \PP(E'')=\emptyset$ in $P$ and $H^{*,*}_\emptyset(P)=0$ 
from property (Loc).
\end{proof}

\rem \label{Chern&nilpotence}
Suppose $X$ admits an ample line bundle
 and consider a vector bundle $E/X$.
As a corollary of the first point of proposition 
\ref{nilpotence&FGL} and the usual splitting principle,
we obtain that the class $c_n(E)$ is nilpotent in $H^{**}(X)$
for any integer $n \geq 0$.

%% file: purity.tex
\section{The Gysin triangle}

In this section, we consider closed pairs $(X,Z)$ --
recall $X$ is assumed to be smooth and $Z$ is a closed subscheme of $X$.
We say $(X,Z)$ is \emph{smooth} (resp. of \emph{codimension} $n$) 
if $Z$ is smooth
(resp. has everywhere codimension $n$ in $X$).
A \emph{morphism} of closed pair $(f,g):(Y,T) \rightarrow (X,Z)$ 
is a commutative square
$$
\xymatrix@=10pt{
T\ar[r]\ar_g[d] & Y\ar^f[d] \\
Z\ar[r] & X
}
$$
which is cartesian on the underlying topological space.
This means the canonical embedding $T \rightarrow Z \times_X Y$
is a thickening. We say the morphism is \emph{cartesian}
 if the square is cartesian.
 
The premotive $\M_Z(X)$ is functorial with respect to 
morphisms of closed pairs.

\subsection{Purity isomorphism} \label{sec:purity}

Consider a projective bundle over a smooth scheme $X$ of rank $n$.
For any integer $0\leq r \leq n$, 
we will consider the embedding\footnote{
The change of sign which appears
 in this formula amounts to take $-c$ instead of $c$
  as a generator of the algebra $H^{**}(P)$.}
$$
\lef r {P}:\M(X)(r)[2r] \xrightarrow{(-1)^r}
 \bigoplus_{0 \leq i \leq n} \M(X)(i)[2i]
 \xrightarrow{\epsilon_{P/X}^{-1}} \M(P).
$$
where the first map is the canonical embedding time $(-1)^r$
 and the second one is induced by the isomorphism of theorem \ref{th:projbdl}.

\num \label{def_diags}
Consider a smooth closed pair $(X,Z)$.
Let $N_ZX$ (resp. $B_ZX$) be the normal bundle (resp. blow-up) of $(X,Z)$
and $P_ZX$ be the projective completion of $N_ZX$.
We denote by $B_Z(\AA^1_X)$
the blow-up of $\AA^1_X$ with center $\{0\} \times Z$.
It contains as a closed subscheme the trivial blow-up 
$\AA^1_Z=B_Z(\AA^1_Z)$. 
We consider the closed pair $(B_Z(\AA^1_X),\AA^1_Z)$ over 
$\AA^1$. Its fiber over $1$ is the closed pair $(X,Z)$
and its fiber over $0$ is $(B_ZX \cup P_ZX,Z)$.
Thus we can consider the following deformation diagram~:
\begin{equation} \label{1st_def_diag}
(X,Z) \xrightarrow{\bar \sigma_1} (B_Z(\AA^1_X),\AA^1_Z)
 \xleftarrow{\bar \sigma_0} (P_ZX,Z).
\end{equation}
We will also consider the open subscheme $D_ZX=B_Z(\AA^1_X)-B_ZX$,
which still contains $\AA^1_Z$ as a closed subscheme.
The previous diagram then gives by restriction a second deformation 
diagram~:
\begin{equation} \label{2nd_def_diag}
(X,Z) \xrightarrow{\sigma_1} (D_ZX,\AA^1_Z)
 \xleftarrow{\sigma_0} (N_ZX,Z).
\end{equation}
Note these two deformation diagrams are functorial in $(X,Z)$ with respect
to cartesian morphisms of closed pairs.

\rem \label{rem:comp_2nd_def_space}
As we will see in the followings, one of the advantage to consider 
the deformation space $D_ZX$ is that, when $X$ is a vector bundle over
$Z$ and the embedding $Z \subset X$ is the $0$-section,
 we can define a canonical isomorphism $D_ZX \simeq \AA^1 \times X$.
In fact, when $X=\spec A$ and $Z=\spec{A/I}$,
 $D_ZX=\spec{\oplus_{n \in \ZZ} I^n.t^{-n}}$
with the convention that for $n<0$, $I^n=A$ ($t$ is an indeterminate).
Thus, if $A=A_0[x_1,...,x_n]$, $I=(x_1,...,x_n)$, we get an isomorphism
defined on the affine level by 
$$
A[t',x'_1,..x'_n] \rightarrow \oplus_{n \in \ZZ} I^n.t^{-n},
 t' \mapsto t, x'_i \mapsto t^{-1}x_i.
$$
This isomorphism is independant on the regular sequence parametrizing $I$.
Thus, in the case when $X$ is an arbitrary vector bundle, 
we can glue the isomorphisms obtained by choosing local parametrizations.
\begin{prop}
\label{prop:purity}
Let $n$ be a natural integer.

There exists a unique family of isomorphisms of the form
$$
\pur{X,Z}:\M_Z(X) \rightarrow \M(Z)(n)[2n]
$$
indexed by smooth closed pairs of codimension $n$ such that~:
\begin{enumerate}
\item for every cartesian morphism $(f,g):(Y,T) \rightarrow (X,Z)$ of
smooth closed pairs of codimension $n$, the following diagram is
commutative~:
$$
\xymatrix@R=16pt@C=40pt{
\M_T(Y)\ar^{(f,g)_*}[r]\ar_{\pur{Y,T}}[d]
 & \M_Z(X)\ar^{\pur{X,Z}}[d] \\
\M(T)(n)[2n]\ar^{g_*(n)[2n]}[r] &  \M(Z)(n)[2n].
}
$$
\item Let $X$ be a smooth scheme,
$E$ be a vector bundle over $X$ of rank $n$.
Put $P=\PP(E \oplus 1)$.
Consider the closed pair $(P,X)$ corresponding to the canonical section of $P/X$. 
Then $\pur{P,X}$ is the inverse of the following composition
$$
\M(X)(n)[2n]
 \xrightarrow{\lef n P} \M(P)
 \xrightarrow \pi \M_X(P)
$$
where the second arrow is obtained by functoriality in $\D$.
\end{enumerate}
\end{prop}
\begin{proof}
Uniqueness~: Consider a smooth closed pair $(X,Z)$ of codimension $n$.
Applying property (1) above to the deformation 
diagram \eqref{1st_def_diag}, we obtain the following commutative 
diagram~:
$$
\xymatrix@C=40pt@R=20pt{
\M_Z(X)\ar^/-6pt/{\bar \sigma_{1*}}[r]\ar_{\pur{X,Z}}[d]
 & \M_{\AA^1_Z}(B_Z(\AA^1_X))\ar|{\pur{B_Z(\AA^1_X),\AA^1_Z}}[d]
 & \M_Z(P_ZX)\ar^{\ \pur{P_ZX,Z}}[d]
                           \ar_{\bar \sigma_{0*}}[l] \\
\M(Z)(n)[2n]\ar^/-2pt/{s_{1*}(n)[2n]}[r]
 & \M(\AA^1_Z)(n)[2n]
 & \M(Z)(n)[2n]\ar_/-3pt/{s_{0*}(n)[2n]}[l] 
}
$$
The morphisms $s_0,s_1:Z \rightarrow \AA^1_Z$ are respectively
the zero section and the unit section of $\AA^1_Z/Z$.
Using axiom (Htp), $s_{0*}=s_{1*}$.
Thus in the above diagram, all morphisms are isomorphisms. 
Now, property (2) stated previously determines uniquely $\pur{P_ZX,Z}$, 
thus $\pur{X,Z}$ is also uniquely determined.

\noindent Existence~: Consider property (2). 
Let $i:\PP(E) \rightarrow P$ be the canonical embedding.
Its corestriction $i':\PP(E) \rightarrow P-X$ is the zero section 
of a vector bundle,
thus it induces an isomorphism on premotives from property (Htp'). 
By (Loc), we then obtain the distinguished triangle~:
$$
\M(\PP(E)) \xrightarrow{i_*} \M(P) \xrightarrow{\pi} \M_X(P)
 \xrightarrow{+1}
$$
We easily obtain $\lef r {\PP(E)} \circ i_*=\lef r P$ for any integer $r<n$.
Thus the composite $\lef n P \circ \pi$ is an isomorphism as required.
We put: $\pur{P,X}=\big( \lef n P \circ \pi \big)^{-1}$.

Considering the proof of uniqueness, we have to show that 
 $\bar \sigma_{0*}$ and $\bar \sigma_{1*}$ are isomorphisms.
Considering the excision axiom (Exc), this is equivalent
to prove the morphisms
$$
\M_Z(X)
 \xrightarrow{\sigma_{1*}} \M_{\AA^1_Z}(D_Z(X))
 \xleftarrow{\sigma_{0*}} \M_Z(N_ZX)
$$
induced by diagram \eqref{2nd_def_diag} are isomorphisms.
In the case $X=\AA^n_Z$ and the inclusion $Z \subset X$ is the $0$-section,
 the result follows from remark \ref{rem:comp_2nd_def_space}
  and axiom (Htp).

We can argue locally for the Zariski topology on $X$.
In fact, consider an open cover $X=U \cup V$, $W=U \cap V$,
 such that the case of 
 $(U,Z \cap U)$, $(V,Z \cap V)$ and $(W,Z \cap W)$ are known.
Using axiom (Sym), (Exc) and (Loc),
the canonical map
$$
\MD{V/V-Z \cap V}{W/W-Z \cap W}
 \rightarrow \MD{X/X-Z}{U/U-Z \cap U}
$$
is an isomorphism, and the same is true when we replace $(X,Z)$
by $(D_ZX,\AA^1_Z)$. This fact, together with the above three assumptions
and axiom (Loc), allows to obtain the result for $(D_ZX,\AA^1_Z)$. \\
Thus we can assume there exists a parametrisation of the closed pair $(X,Z)$,
that is to say a cartesian morphism 
$(f,g):(X,Z) \rightarrow (\AA^{d+n}_S,\AA^d_S)$
such that $f$ is \'etale.
Consider the pullback square
$$
\xymatrix@R=10pt@C=18pt{
X'\ar^p[r]\ar_q[d] & X\ar^f[d] \\
{}\AA^n_Z\ar^{1 \times g}[r] & {}\AA^{n+d}_S. }
$$
There is an obvious closed immersion $Z \rightarrow X'$
and its image is contained in $q^{-1}(Z)$.
As $q$ is \'etale, $Z$ is a direct factor of $q^{-1}(Z)$.
Put $W=q^{-1}(Z)-Z$ and $\Omega=X'-W$. 
Thus $\Omega$ is an open subscheme of $X'$, 
and the reader can check that $p$ and $q$ induce cartesian \'etale 
morphisms
$$
(X,Z) \leftarrow (\Omega,Z) \rightarrow (\AA^n_Z,Z).
$$
The functorialty of \eqref{2nd_def_diag}
 and axiom (Exc) allow to conclude in view of the previous case.

To sum up, the purity isomorphism $\pur{X,Z}$ is defined as the
composite
$$
\M_ZX \xrightarrow{\bar \sigma_{0*}} \M_{\AA^1_Z}\big(B_Z(\AA^1_X)\big)
 \xrightarrow{\bar \sigma_{1*}^{-1}} \M_Z(P_ZX)
  \xrightarrow{\pur{Z,P_ZX}} M(Z)(n)[2n].
$$
We finally have to check the coherence of this definition
 in the case of the closed pair $(P,X)$, $P=\PP(E \oplus 1)$,
  appearing in property (2).
Explicitely, we have to check that in this case
 $\bar \sigma_{1*}^{-1} \circ \bar \sigma_{0*}=1$.
This is easily seen considering the commutative diagram~:
$$
\xymatrix@R=10pt{
\M_X(P)\ar^/-6pt/{\bar \sigma_{1*}}[r] & \M_{\AA^1_X}(B_Z(\AA^1_X))
 & \M_X(P)\ar_/-9pt/{\bar \sigma_{0*}}[l] \\
\M_X(E)\ar^/-4pt/{\sigma_{1*}}[r]\ar[u] & \M_{\AA^1_X}(D_XE)\ar[u]
 & \M_X(E).\ar_/-7pt/{\sigma_{0*}}[l]\ar[u]
}
$$
We have identified the projective normal bundle of $(P,X)$
 (resp. the normal bundle of $(E,X)$) with $P$ (resp. $E$).
According to remark \ref{rem:comp_2nd_def_space}, there is a canonical isomorphism
$D_XE \simeq \AA^1 \times E$ through which $\sigma_0$ (resp. $\sigma_1$)
corresponds to the zero (resp. unit) section.
The homotopy axiom (Htp) allows to conclude.
\end{proof}

\num \label{Thom_class}
Let $X$ be a smooth scheme, $E$ be a vector bundle over $X$ of rank $n$
and put $P=\PP(E \oplus 1)$.
Let $\L$ be the canonical line bundle on $P$,
and $p:P \rightarrow X$ be the canonical projection.
We define the \emph{Thom class} of $E/X$ as the cohomology class
$$
t(E)=
\sum_{i=0}^n p^*(c_i(E)) \ncup \left(-c_1(\L)\right)^{n-i}
$$
in $H^{2n}(P)$. 
This is in fact a morphism $\M(P) \rightarrow \un(n)[2n]$
whose restriction to $\M(\PP(E))$ is zero. This implies
the morphism
$$
p_* \gcup t(E):\M(P) \rightarrow \M(X)(n)[2n]
$$
factors as a morphism $\M_X(P) \rightarrow \M(X)(n)[2n]$
and this latter is equal to $\pur{P,X}$. 
Indeed, $p_* \gcup t(E)$ is a split epimorphism with splitting 
$\lef n P$. \\
We introduce the 
\emph{Thom premotive}\footnote{Analog of the Thom space in algebraic topology.}
as $\MTh(E):=\M_X(E)$ - remark it is functorial with respect to monomorphisms
of vector bundles. Using property (Exc), 
the natural morphism $\MTh(E) \rightarrow M_X(P)$ is an isomorphism.
As a consequence, the morphism $p_* \gcup t(E)$ induces an isomorphism
$\MTh(E):\MTh(E) \rightarrow \M(X)(n)[2n]$ which is precisely
the purity isomorphism $\pur{E,X}$. 
In the litterature, this arrow is called the \emph{Thom isomorphism}.

\begin{rem} \label{Thom&quotient}
Recall the universal quotient bundle $\Q$ on $P$ is defined by 
the exact sequence
$$
0 \rightarrow \L \rightarrow p^{-1}(E \oplus 1) \rightarrow \Q \rightarrow 0.
$$
Thus the Whitney sum formula \ref{additivity_Chern} 
gives: $t(E)=c_n(\Q).$
\end{rem}

\begin{df} \label{df_purity&Gysin}
Let $(X,Z)$ be a smooth closed pair of codimension $n$.
Put $U=X-Z$ and consider the obvious immersions
$i:Z \rightarrow X$ and $j:U \rightarrow X$. \\
\indent Considering the notations of the previous proposition,
we call $\pur{X,Z}$ the \emph{purity isomorphism} associated with $(X,Z)$.
Using this isomorphism together with property (Loc) we obtain a 
distinguished triangle
$$
\M(X-Z) \xrightarrow{j_*} \M(X)
 \xrightarrow{i^*} \M(Z)(n)[2n]
  \xrightarrow{\partial_{X,Z}} \M(X-Z)[1]
$$
called the \emph{Gysin triangle}.
The morphism $i^*$ (resp. $\partial_{X,Z}$)
is called the \emph{Gysin morphism} (resp. \emph{residue morphism})
associated with $(X,Z)$.
\end{df}

\ex \label{ex:Gysin&Thom}
Let $X$ be a smooth scheme and $E/X$ be a vector bundle of rank $n$.
Put $P=\PP(E \oplus 1)$ 
and consider the canonical section $s:X \rightarrow P$ of $P/X$.
Then property (2) of proposition \ref{prop:purity}
implies $s^* \circ \lef n P=1$~:
the Gysin triangle of $(P,X)$ is split and $\partial_{P,X}=0$.
Moreover, remark \ref{Thom_class} and the previous definition implies 
that
$$s^*=p_* \gcup t(E).
$$

\subsection{Base change}

\begin{df} \label{df:refined_Gysin}
Let $(X,Z)$ (resp. $(Y,T)$) be a smooth closed pair of codimension $n$
(resp. $m$). 
Let $(f,g):(Y,T) \rightarrow (X,Z)$ be a morphism of closed pairs.
We define the morphism $(f,g)_!:\M(T)(m)[2m] \rightarrow \M(Z)(n)[2n]$
by the equality $(f,g)_!:=\pur{X,Z} \circ (f,g)_* \circ \pur{Y,T}^{-1}$.
\end{df}
Thus we obtain a commutative diagram
\begin{equation} \label{refind_Gysin}
\xymatrix@R=16pt@C=22pt{
\M(Y-T) \ar^/4pt/{l_*}[r]\ar_{h_*}[d]
 & \M(Y)\ar^/-7pt/{k^*}[r]\ar_{f_*}[d]
 & \M(T)(n)[2n]\ar^{\partial_{Y,T}}[r]\ar^{(f,g)_!}[d]
 & \M(Y-T)[1]\ar^{h_*[1]}[d] \\
\M(X-Z) \ar^/4pt/{j_*}[r]
 & \M(X)\ar^/-7pt/{i^*}[r]
 & \M(Z)(n)[2n]\ar^{\partial_{X,Z}}[r]
 & \M(X-Z)[1]
}\end{equation}
where $i$, $j$, $k$, $l$ are the obvious immersions 
and $h$ is the restriction of $f$.

In what follows, we will compute the morphism $(f,g)_!$
in various cases. 
The commutativity of the second square will give us 
\emph{refined projection formulas}.
The new thing in our study is that any such formula
corresponds to another formula involving residue morphisms 
as we see by considering the third commutative square.

\rem The notation $(f,g)_!$ is to be compared with the notation of \cite{Ful}
for the "refined Gysin morphism". In fact, the reader will notice that
in the case of motivic cohomology, our formulas extend the formulas of Fulton
to the case of arbitrary weights (and arbitrary base). 
Be careful however that our Gysin morphism $i^*:\M(X) \rightarrow \M(Z)(n)[2n]$ 
corresponds to the usual pushout on Chow groups (cf \cite{Deg6}[1.21]). 
The Gysin morphism considered by Fulton is induced 
by the usual functoriality of motives. 
This fact can be understand if we thought of Chow groups over a field studied 
by Fulton as motivic homology with compact support.

\subsubsection{The transversal case}

\begin{prop} \label{transversal}
Consider the hypothesis of definition \ref{df:refined_Gysin}.

Suppose $(f,g)$ is cartesian and $n=m$. Then $(f,g)_!=g_*(n)[2n]$.
\end{prop}
\begin{proof}
Diagram \eqref{1st_def_diag} is functorial with respect to cartesian morphism.
Let $p:P_TY \rightarrow P_ZX$ be the morphism induced by $(f,g)$ on
the projective completions of the normal bundles.
Through the morphisms $\bar \sigma_{0*}$ and $\bar \sigma_{1*}$ for the
closed pairs $(X,Z)$ and $(Y,T)$, the morphism $(f,g)_*$ is isomorphic to 
$$
(p,g)_*:\M(P_TY,T) \rightarrow \M(P_ZX,Z).
$$
As $n=m$ and $Y=X \times_Z T$, one has $P_TY=P_ZX \times_Z T$. 
Using the compatibility of the projective bundle isomorphism with base change,
we see that the following diagram commutes
$$
\xymatrix@=10pt@C=42pt{
\M(T)(n)[2n]\ar^/5pt/{\lef n {P_TY}}[r]\ar_{g_*(n)[2n]}[d]
 & \M(P_TY)\ar^{p_*}[d] \\
\M(Z)(n)[2n]\ar^/5pt/{\lef n {P_ZX}}[r] & \M(P_ZX)
}
$$
which concludes in view of the property (2) in proposition 
\ref{prop:purity}.
\end{proof}

\begin{cor} \label{second_proj_formula}
Consider a smooth closed pair $(X,Z)$ of codimension $n$
 and $i:Z \rightarrow X$ the corresponding immersion. Put $U=X-Z$.

Then $(1_{Z*} \gcup i_*) \circ i^*=i^* \gcup 1_{X*}$
as a morphism $\M(X) \rightarrow \M(Z \times X)(n)[2n]$, \\
and $(j_* \gcup 1_{U*}) \circ \partial_{X,Z}=\partial_{X,Z} \gcup i_*$
as a morphism $\M(Z)(n)[2n] \rightarrow \M(U \times X)[1]$.
\end{cor}
\begin{proof}
We consider the cartesian square
$$
\xymatrix@R=12pt@C=20pt{
Z\ar_{\gamma_i}[d]\ar^i[r] & X\ar^{\delta_X}[d] \\
Z \times X\ar^{i \times 1_X}[r] & X \times X
}
$$
where $\delta_X$ is the diagonal embedding of $X/S$.
The two formulas then follow from the previous proposition
applied to the morphism of closed pairs 
$(\delta_X,\gamma_i):(X,Z) \rightarrow (X \times X,Z \times X)$
with the help of the following elementary lemma~:
\begin{lm} \label{lm:Gysin&product}
Let $(X,Z)$ be a smooth closed pair of codimension $n$
and $Y$ be a smooth scheme.

Then $(i \times 1_Y)^*=i^* \otimes 1_{Y*}$
 and $\partial_{X \times Y,Z \times Y}=\partial_{X,Z} \otimes 1_{Y*}$.
\end{lm}
\noindent Using axiom (Kun)(a) and (Kun)(b),
 the lemma is reduced to prove that
$\pur{X \times Y,Z \times Y}=\pur{X,Z} \otimes Y$.
From the construction of the purity isomorphism,
 we are reduced to show that for a projective bundle $P/X$,
 $\epsilon_{P \times Y}=\epsilon_{P} \otimes 1_{X*}$
  using the notations of theorem \ref{th:projbdl}.
This last equality follows finally from axiom (Kun)(a) and 
the functoriality of the first Chern class in axiom (Orient).
\end{proof}

\rem
\begin{enumerate}
\item In the formula of this lemma,
 there is hidden a permutation isomorphism for the tensor product. 
In this paper, we will not need to care about this isomorphism. 
However, in some cases, it may result in a change of sign 
 (see \cite{Deg5}, rem. 2.6.2).
\item Considering a ringed premotive $\E$, the previous corollary gives
the usual projection formula for $i$~: for any $z \in \E^{**}(Z)$
and any $x \in \E^{**}(X)$, $i_*(z \ncup i^*(x))=i_*(z) \ncup x$.
\end{enumerate}

\begin{num}\label{fund_class}
Let $(X,Z)$ be a smooth closed pair of codimension $n$,
$i:Z \rightarrow X$ the corres\-ponding closed immersion.
Following Grothendieck (see \cite{Gro}), 
we define the \emph{fundamental class} of $Z$ in $X$ 
as the cohomology class $\eta_X(Z)=i_*(1)$ in $H^{2n,n}(X)$.
As a morphism, it is equal to the composite
$$
\M(X) \xrightarrow{i^*} \M(Z)(n)[2n]
 \xrightarrow{\pi_{Z*}} \un(n)[2n]
$$
where $\pi_Z:Z \rightarrow S$ is the structural morphism of $Z/S$. \\
Suppose that $i$ admits a retraction $p:X \rightarrow Z$.
Then corollary 3.10 gives the following computation\footnote{
Considered in cohomology, this is a well known formula.}
 of the Gysin morphism~:
\begin{equation} \label{Gysin&Thom}
i^*=p_* \gcup \eta_X(Z).
\end{equation}
Suppose given a vector bundle $E/X$ and put $P=\PP(E \oplus 1)$.
Applying example \ref{ex:Gysin&Thom}, we get
$$
\eta_{P}(X)=t(E)
$$
where $X$ is embedded in $P$ through the canonical 
section. Indeed example \ref{ex:Gysin&Thom} is a particular case
of the formula \eqref{Gysin&Thom}. \\
More generally, we can define the \emph{localised fundamental class}
of $Z$ in $X$ as the cohomology class $\bar \eta_X(Z) \in H^{2n,n}_Z(X)$
equal to the composite
$$
M_Z(X) \xrightarrow{\pur{X,Z}} M(Z)(n)[2n]
 \xrightarrow{\pi_{Z*}} \un(n)[2n].
$$
Considering the canonical morphism
 $\nu_{X,Z}:H^{2n,n}_Z(X) \rightarrow H^{2n,n}(X)$,
  we have tautologically $\nu_{X,Z}(\bar \eta_X(Z))=\eta_X(Z)$. \\
For any vector bundle $E/X$ of rank $n$, $P=\PP(E \oplus 1)$,
 the \emph{localised Thom class} $\bar t(E)=\bar \eta_P(X)$ is uniquely
  determined by the Thom class $t(E)$. 
Usually, $\bar t(E)$ is considered as an element of $H^{2n,n}_X(E)$
 using axiom (Exc).
\end{num}

As a last application of the previous corollary, let us remark
the following~:

\begin{cor} \label{Gysin&proj_bdl_iso}
Let $(X,Z)$ be a smooth closed pair of codimension $m$,
and $P$ be a projective bundle of rank $n$ over $X$.

Then for any integer $r \in [0,n]$, the following diagram
is commutative~:
$$
\xymatrix@R=30pt@C=14pt{
\M(P_V)\ar^{\nu_*}[r]
    \ar|{p_{V*} \gcup c_1(\lambda_V)^r}[d]
 & \M(P)\ar^/-6pt/{\iota^*}[r]
    \ar|{p_* \gcup c_1(\lambda)^r}[d]
 & \M(P_Z)(m)[2m]\ar^/2pt/{\partial_\iota}[r]
    \ar|{p_{Z*} \gcup c_1(\lambda_Z)^r}[d]
 & \M(P_V)[1]
    \ar|{p_{V*}[1] \gcup c_1(\lambda_V)^r}[d] \\
\M(V)(r)[2r]\ar^{j_*}[r]
 & \M(Y)(r)[2r]\ar^/-13pt/{i^*}[r]
 & \M(Z)(r+m)[2(r+m)]\ar^/7pt/{\partial_i}[r]
 & \M(V)(r)[2r+1].
}
$$
\end{cor}
In particular, the Gysin triangle is compatible with the
projective bundle isomorphisms and with the induced
embeddings $\lef r {P_?}$.

\subsubsection{The excess intersection case}

Remark that in the hypothesis of definition \ref{df:refined_Gysin},
we have a canonical closed immersion
$$
N_TY \xrightarrow{\nu} g^*(N_ZX).
$$
In particular, we have necessarily the inequality $n \geq m$. 
\begin{prop} \label{excess}
Consider the hypothesis of definition \ref{df:refined_Gysin}.
Suppose $(f,g)$ is cartesian. \\
Put $e=n-m$ and consider $\xi=g^{-1}(N_ZX)/N_TY$,
 quotient vector bundle over $T$.

Then $(f,g)_!=\left(g_* \gcup_T c_e(\xi)\right)(m)[2m]$.
\end{prop}
\rem The integer $e$ is usually called the \emph{excess of intersection},
and $\xi$ the \emph{excess intersection bundle}.
\begin{proof}
The morphism $(f,g)$ induces the following composite morphism on normal
bundles~:
$$
N_TY \xrightarrow{\nu} g^{-1}(N_ZX) \xrightarrow{g'} N_ZX.
$$
Thus, considering now the functoriality of diagram \eqref{2nd_def_diag}
with respect to the cartesian morphism $(f,g)$, 
we obtain $(f,g)_!=(\nu,1_T)_! (g',g)_!$.
From proposition \ref{transversal}, $(g',g)_!=g_*(n)[2n]$.
We conclude using the following lemma~:
\begin{lm}
Let $E$ and $F$ be vector bundles over a smooth scheme $T$
of respective rank $n$ and $m$.
Consider a monomorphism $\nu:F \rightarrow E$ of vector bundles
and put $e=n-m$.

Then $(\nu,1_T)_!=\big(1_{T*} \gcup c_e(E/F)\big)(m)[2m]$.
%
\end{lm}
To prove the lemma, we use the description of $\pur{F,T}$ and $\pur{E,T}$
using the Thom class (cf \ref{Thom_class}).
Let $P$, $Q$ and $\bar \nu:Q \rightarrow P$ be the respective projective completions 
of $E$, $F$ and $\nu$. 
Let $p:P \rightarrow T$ and $q:Q \rightarrow T$ be the canonical projections.
We are reduced to prove the relation 
$\bar \nu^*(t(E))=(q^* c_e(E/F)) \ncup t(F)$ in $H^{2n,n}(Q)$. \\
From remark \ref{Thom&quotient}, we get $t(E)=c_n(\xi_P)$ 
(resp. $t(F)=c_m(\xi_Q)$) where $\xi_P$ (resp. $\xi_Q$) 
is the universal quotient bundle on $P$ (resp. $Q$).
Thus, the relation follows from the Whitney sum formula
\ref{additivity_Chern} and the following exact sequence
of vector bundles over $Q$~:
$$
0 \rightarrow \xi_Q \rightarrow \bar \nu^{-1}\xi_P
 \rightarrow q^{-1}(E/F) \rightarrow 0.
$$
\end{proof}


\begin{cor} \label{self_intersection}
Let $(X,Z)$ be a smooth closed pair of codimension $n$. Then~: \\
(1) $i^* i_*=1_Z \gcup c_n(N_ZX)$ as a morphism $\M(Z) \rightarrow \M(Z)(n)[2n]$. \\
(2) $\partial_{X,Z} \circ  (1_Z \gcup c_n(N_ZX))=0$.
\end{cor}
This follows from the previous proposition applied with
$(f,g)=(i,1_Z)$. 
We usually refer to the first formula as the \emph{self-intersection formula}.

\num Consider a vector bundle $E$ over a smooth scheme $X$ of rank $n$.
Let $E^\times$ be the complement of the zero section in $E$ and
$\pi:E^\times \rightarrow X$ be the obvious projection.
Then using property (Htp') and the previous corollary, 
we obtain from the Gysin triangle for $(E,X)$
the following distinguished triangle
$$
\M(E^\times) \xrightarrow{\pi_*} \M(X) \xrightarrow{1_X \gcup c_n(E)}
 \M(X)(n)[2n] \xrightarrow{\partial_{E,X}} \M(E^\times)[1]
$$
which we shall call the \emph{Euler distinguished triangle}.
Indeed, in cohomology with coefficients in a ringed premotive $\E$,
it corresponds to a long exact sequence where one of the arrow 
is the cup product by $c_n(E)$.

As a corollary of the self-intersection formula \ref{self_intersection},
 we obtain the following tool to compute fundamental classes 
 which generalises in our setting a theorem of Grothendieck
  (cf \cite[th. 2]{Gro}).
\begin{cor} \label{fundamental_class_as_chern_class}
Consider a smooth closed pair $(X,Z)$ of codimension $n$.
Let $i$ be the corresponding closed immersion
and $\eta_X(Z)=i_*(1) \in H^{2n,n}(X)$ be the fundamental class of $Z$ in $X$ 
 (cf \ref{fund_class}). \\
Suppose there exists a vector bundle $E$ on $X$ and a section $s$
of $E/X$ such that $s$ is transversal to the zero section $s_0$ of $E$
and $Z=s^{-1}(s_0(X))$.

Then, $\eta_X(Z)=c_n(E)$.
\end{cor}
It simply follows from corollary \ref{self_intersection} applied to $s_0$ 
together with proposition \ref{transversal} applied to the following 
transversal square~:
$$
\xymatrix@=10pt{
Z\ar^i[r]\ar[d] & X\ar^{s_0}[d] \\
X\ar^{s}[r] & E.
}
$$

\begin{ex} \label{Panin}
Let $E$ be a vector bundle of rank $n$ over a smooth scheme $X$.
Put $P=\PP(E \oplus 1)$
 and consider $p:P \rightarrow X$ (resp. $s:X \rightarrow P$, $\L$)
  the canonical projection (resp. section, line bundle) of $P/X$.
Consider finally the vector bundle
 $F=\L^\vee \otimes p^{-1}(E)$ over $P$.
The sequence of morphisms of vector bundles over $P$,
$$
\L \rightarrow p^{-1}(E \oplus 1) \rightarrow p^{-1}(E)
$$
gives a section $\sigma$ of $F/P$. We check easily
it is transversal to the zero section and we have $\sigma^{-1}(0)=X$,
while the embedding $\sigma^{-1}(0) \rightarrow P$ is $s$.
Thus we obtain from the previous corollary $\eta_X(P)=c_n(F)$.
Considering paragraph \ref{fund_class} and remark \ref{Thom&quotient}
we thus obtain three expressions of the fundamental class of $X$ in $P$~:
$$
t(E)=c_n\big(p^{-1}(E \oplus 1)/\L\big)
 =c_n\big(\L^\vee \otimes p^{-1}(E)\big).
$$
Note the last equality,
 though obvious in the case where $F$ is the additive formal group,
  is not evident to check directly in the general case.
However, we left as an exercice to the reader to check it 
 using the inverse series of the formal group law $F$ 
  in the case of a line bundle. This implies the general case by the
   splitting principle.
\end{ex}


\subsubsection{The ramified case}

In this section, we study the case of a morphism
$(f,g):(Y,T) \rightarrow (X,Z)$ of smooth closed pairs
of same codimension $n$.
This corresponds to the \emph{proper case} in the operation of pullback
of $Z$ along $f$. 
We put $T'=Z \times_X Y$ and consider the canonical thickening $T' \rightarrow T$
induced by $(f,g)$. \\
We first need an assumtion.
Let $T'=\bigcup_{i \in I} T'_i$ be the decomposition into connected components. 
For any $i \in I$,
we also consider the decomposition $T'_i=\bigcup_{j \in J_i} T'_{ij}$
into irreducible components.
Put $T_{ij}=T'_{ij} \times_{T'} T$.
As $T \rightarrow T'$ is a thickening,
the geometric multiplicity $m(T'_{ij})$ of $T'_{ij}$
is an integral multiple of the geometric multiplicity
$m(T_{ij})$ of $T_{ij}$. We introduce the following condition on 
the morphism $(f,g)$:
\begin{enumerate}
\item[(Special)] For any $i \in I$, there exists an integer $r_i \geq 0$
such that for any $j \in J_i$, $m(T'_{ij})=r_i.m(T_{ij})$.
\end{enumerate}
The integer $r_i$ will be called the 
\emph{ramification index}
of $f$ along $T_i$.

\begin{rem}
When $S$ is irreducible, this condition is always fulfilled.
When $S$ is integral, $T'_i$ is irreducible and
the integer $r_i$ is nothing else than the geometric multiplicity 
of $T'_i$.
\end{rem}

Under this assumption, we define intersection multiplicities which
take into account the formal group law $F$ introduced in paragraph 
\ref{fgl}. \\
Let $B$ be the blow-up of $\AA^1_X$ with center $\{0\} \times Z$,
 and $P$ its exceptional divisor.
Put $C=B \times_X Y$, and for any $i \in I$, $Q_i=P \times_T T_i$.
Remark that $Q_i/T_i$ admits a canonical section $s_i$.
We denote by $L_i$ the line bundle over $T_i$ obtain by the 
pullback of the normal bundle $N_{Q_i}(C)$ along $s_i$.
We consider the localised Thom class
$\bar t(L_i) \in H^{2,1}_{T_i}(L_i)$ (cf \ref{fund_class});
we recall it is sent to $1$ by the purity isomorphism
$\pur{L_i,T_i}^*:H^{2,1}_{T_i}(L_i) \rightarrow H^{0,0}(T_i)$. \\
Note that, according to remark \ref{Chern&nilpotence},
 the Thom class $t(L_i)$ is nilpotent. Thus, the same is true
  for $\bar t(L_i)$.
In particular, we can apply the power series $[r_i]_F$ (see paragraph \ref{fgl})
to the element $\bar t(L_i)$ of the $A$-algebra $H^{**}_{T_i}(L_i)$.
This defines an element $[r_i]_F \cdot \bar t(L_i) \in H^{**}_{T_i}(L_i)$
of bidegree $(2,1)$.
\begin{df} \label{inter_mult}
Consider a morphism $(f,g):(Y,T) \rightarrow (X,Z)$
 which satisfies the condition (Special).
Assume $T$ admits an ample line bundle.

We consider the notations introduced above.
For any $i \in I$, 
we define the \emph{$F$-intersection multiplicity} of $T_i$ in $f^{-1}(Z)$
as the element 
$$
r(T_i;f,g)=\pur{L_i,T_i}^*\big( [r_i]_F \cdot \bar t(L) \big)
 \in H^{0,0}(T_i)
$$
where $r_i$ is the ramification index of $f$ along $T_i$.
\end{df}

A straightforward check shows the $F$-intersection multiplicities
 are compatible with flat base change. 
When the formal group law $F$ is additive,
 we easily get that $r(T_i;f,g)=r_i$. \\
In the codimension $n=1$ case, we can also consider the localised
fundamental class $\bar \eta_Y(T_i) \in H^{2,1}_{T_i}(Y)$
introduced in paragraph \ref{fund_class}.
It corresponds to the localised Thom class $\bar t(N_{T_i}(Y))$
under the isomorphisms given by the deformation diagram
\eqref{2nd_def_diag}. Thus applying remark \ref{Chern&nilpotence}
as above, we obtain that the class $\bar \eta_Y(T_i)$ is nilpotent.
In particular, we can consider the class 
$[r_i]_F \cdot \bar \eta_Y(T_i) \in H^{2,1}_{T_i}(Y)$
obtained by applying the power series $[r_i]_F$ of \ref{fgl}.
We then obtain a natural expression of the $F$-intersection multiplicity~:
\begin{lm} \label{mult_divisors}
Consider the hypothesis and assumptions of the previous definition
 and assume $n=1$.
Let $\bar \eta_Y(T_i) \in H^{2,1}_{T_i}(Y)$ be the localised fundamental class 
of $T_i$ in $Y$ (cf paragraph \ref{fund_class}) and 
$\pur{Y,T_i}^*:H^{2,1}_{T_i}(Y) \rightarrow H^{0,0}(T_i)$ be the purity
isomorphism in cohomology.

Then, $r(T_i;f,g)=\pur{Y,T_i}^*\big([r_i]_F \cdot \bar \eta_Y(T_i)\big)$.
\end{lm}
\begin{proof}
We may assume $T$ is connected. Thus $I=\{i\}$ and we put $L=L_i$, 
$r=r_i$ with the notations of the previous definition.
As $n=1$, the zero section of $\AA^1_X/X$ induces the following transversal 
square
$$
\xymatrix@=10pt{
Z=\PP(N_ZX)\ar^-s[r]\ar[d] & \PP(N_ZX \oplus 1)=P\ar[d] \\
X=B_ZX\ar[r] & B_Z(\AA^1_X)=B
}
$$
which, after pullback above $Y$ gives a cartesian square, still transversal,
$
\xymatrix@=10pt{
T\ar^-t[r]\ar[d] & Q\ar[d] \\
Y\ar[r] & C}$
with $t$ the canonical section of $Q/T$. Thus we get~:
\begin{align*}
\pur{L,T}^*([r]_F \cdot \bar t(L))
 &=t^* \pur{N_QC,Q}^*([r]_F \cdot \bar t(N_QC))
 =t^* \pur{C,Q}^*([r]_F \cdot \bar \eta_C(Q)) \\
 &=\pur{Y,T}^*([r]_F \cdot \bar \eta_Y(T))
\end{align*}
where the last equality follows from the transversal square above
and proposition \ref{transversal} whereas the other equalities follow 
from the definitions.
\end{proof}

Before stating the main result of this section,
 we need to recall an extension of the functoriality
of the deformation diagram \eqref{2nd_def_diag} to certain morphisms of closed
pairs (see also \cite[proof of 3.3]{Deg2}). \\
Consider a morphism $(f,g):(Y,T) \rightarrow (X,Z)$ of smooth closed pairs
of codimension $1$.
Let $\I$ (resp. $\J$, $\J'$) be the ideal defining 
$Z$ in $X$ (resp. $T$ in $Y$, $T'$ in $Y$).
The map $f$ induces a morphism $\varphi:\I \rightarrow f_* \J'$ of sheaves
over $X$. \\
We consider the second deformation space $D_ZX=B_Z(\AA^1_X)-B_ZX$ as in
\ref{def_diags}. An easy computation shows
$$
D_ZX=\specx X{\bigoplus_{n \in \ZZ} \I^n.u^{-n}}
$$
where $\I^n=\mathcal O_X$ for $n<0$, and $u$ is an indeterminate. \\
Assume\footnote{
As the immersion $T \rightarrow Y$ is regular,
 this can happen only in the codimension $1$ case.
Note it implies $(f,g)$ satisfies the condition (Special)
 and the ramification indexes are all equal to $r$.}
 \underline{$\J'=\J^r$}.
Then we can define a morphism of sheaves of rings over $X$~:
$$
\bigoplus_{n \in \ZZ} \I^{\,n}.u^{-n}
 \rightarrow \bigoplus_{n \in \ZZ} f_*({\J'}^{\,n}).u^{-n}
  \rightarrow \bigoplus_{m \in \ZZ} f_*(\J^{\,m}).v^{-m}
$$
where the first arrow is induced by $\varphi$ and the second
is the obvious inclusion which maps $u$ to $v^r$ as $\J'=\J^r$. \\
Taking the spectrum of these morphisms over $X$, we get a morphism
$$
\rho_r(f,g):D_TY \rightarrow D_ZX
$$
of schemes over $\AA^1$. The fibre of $\rho_r(f,g)$ over $1$ is simply $f$
 and one can check that its fiber over $0$ is a composite morphism
$$
\sigma_r(f,g):N_TY \xrightarrow{\nu} g^*N_ZX \xrightarrow{\mu} N_ZX
$$
such that $\mu$ is induced by $g$
 and $\nu$ is a homogenous morphism of degree $r$.
Thus, considering the respective deformation diagrams \eqref{2nd_def_diag}
for $(X,Z)$ and $(Y,T)$ 
we obtain a commutative diagram of closed pairs
\begin{equation} \label{ext_funct_2nd_def_diag}
\xymatrix@R=18pt@C=26pt{
(Y,T)\ar^/-4pt/{\sigma'_1}[r]\ar_{(f,g)}[d]
 & (D_TY,\AA^1_T)\ar|{(\rho_r(f,g),1 \times g)}[d]
  & (N_TY,T)\ar_/-4pt/{\sigma'_0}[l]\ar^{(\sigma_r(f,g),g)}[d] \\
(X,Z)\ar^/-4pt/{\sigma_1}[r] & (D_ZX,\AA^1_Z)
 & (N_ZX,Z).\ar_/-4pt/{\sigma_0}[l]
}
\end{equation}

\begin{thm} \label{ramification}
Let $(f,g):(Y,T) \rightarrow (X,Z)$ be a morphism of smooth closed
pairs of codimension $n$.
We assume $T$ admits an ample line bundle and $(f,g)$ satisfies
condition (Special).

Then 
$$(f,g)_!=\sum_{i \in I} r(T_i;f,g) \gcup_T g_{i*}$$
where $T=\bigcup_{i \in I} T_i$ is the decomposition into connected component, 
$g_i=g|_{T_i}$ and $r(T_i;f,g)$ is the $F$-intersection multiplicity of $T_i$ in $f^{-1}(Z)$.
\end{thm}
\begin{proof}
Using axiom (Add'), we can assume $T$ is connected. \\
We first reduce to the codimension $n=1$ case.
Consider the blow-up $B=B_Z(\AA^1_X)$
 and its exceptional divisor $P=\PP(N_ZX \oplus 1)$.
Consider also the cartesian morphism $(p,q):(B,P) \rightarrow (X,Z)$.
If we put $B_Y=B \times_X Y$, $Q=P \times_Z T$, we obtain the following 
commutative diagram of morphisms closed pairs~:
$$
\xymatrix@R=10pt@C=24pt{
(B_Y,Q)\ar^{(f',g')}[r]\ar_{(\pi',q)}[d]
 & (B,P)\ar^{(\pi,p)}[d] \\
(Y,T)\ar^{(f,g)}[r] & (X,Z).
}
$$
By definition, $(f,g)_!(\pi',q)_!=(\pi,p)_!(f',g')_!$. \\
Note that $(\pi,p)$ and $(\pi',q)$ are cartesians.
We can apply proposition \ref{excess} to $(\pi,p)$~:
the excess intersection bundle is the universal quotient bundle
$\xi_0$ on $P$ and $(\pi,p)_!=p_* \gcup c_{n}(\xi_0)$.
Thus, according to remark \ref{Thom&quotient}
 and paragraph \ref{fund_class}, $(\pi,p)_!=s^*$
  where $s:X \rightarrow P$ is the canonical section. \\
Similarly, if we put $\xi={g'}^{-1}(\xi_0)$,
we get $(\pi',q)_!=q_* \gcup c_{n}(\xi)=t^*$ with $t:T \rightarrow Q$
the canonical section.
Note this latter morphism is a split epimorphism with splitting 
 $\lef{n} {Q}$. Thus we get
$$
(f,g)_!=s^* \circ (f',g')_! \circ \lef{n}{Q}.
$$
Remark that $Q=P \times_B B_Y$.
Thus the morphism $(f',g')$ of smooth closed pairs of codimension $1$
satifies the condition (Special)
 and the ramification indexes
  of $f$ along $T$ and $f'$ along $Q$ are equal.
Assume $(f',g')=r(Q;f',g') \gcup g'_*$.
According to the expression above, we get
\begin{align*}
(f,g)_!&=\big(r(Q;f',g') \gcup s^*g'_*\big) \circ \lef{n}{Q}
\stackrel{(1)}=\big(r(Q;f',g') \gcup g_*t^*\big) \circ \lef{n}{Q} \\
&\stackrel{(2)}=\big((r(Q;f',g') \circ t_*) \gcup g_*\big) t^* \circ \lef{n}{Q}
=(r(Q;f',g') \circ t_*) \gcup g_*.
\end{align*}
where equality (1) follows from the projection formula of proposition \ref{transversal}
 and equality (2) from the other projection formula of corollary \ref{second_proj_formula}.
From definition \ref{inter_mult}, the reader can now easily check the equality of the
cohomological classes $t^*\lbrack r(Q;f',g') \rbrack=r(T;f,g)$.

Thus we are reduced to the case $n=1$, $T$ still being connected.
Let $r$ be the ramification index of $f$ along $T$.
Let $\J$ (resp. $\J'$) be the ideal sheaf of $T$ (resp. $T'$) in $Y$.
As $Z \rightarrow X$ and $T \rightarrow Y$ are regular immersions of 
a divisor, we see that necessarily, $\J'=\J^r$.
Considering now diagram \eqref{ext_funct_2nd_def_diag}, we obtain that
$(f,g)_!=(\sigma_r(f,g),g)_!$. 
In view of the factorization of the morphism $\sigma_r(f,g)$,
we then are reduced to the following lemma~:
\begin{lm}
Let $T$ be a smooth scheme which admits an ample line bundle.
Consider a line bundle $N$ over $T$ and
 $N^{\otimes r}$ be its $r$-th tensor power over $T$. \\
Let $\nu:N \rightarrow N^{\otimes r}$ be the obvious homogenous morphism of degree $r$,
 and $(\nu,1_T):(N,T) \rightarrow (N^{\otimes r},T)$ be the
corresponding morphism of closed pairs.

Then $(\nu,1_T)_!=\rho \otimes 1_{T*}$ where 
$\rho$ is the unique element of $H^{00}(T)$ such that
$[r]_F \cdot t(N)=\rho.t(N)$.
\end{lm}
Put $P=\PP(N \oplus 1)$, $P'=\PP(N^{\otimes r} \oplus 1)$
and consider the projective completion $\bar \nu:P \rightarrow P'$
of $\nu$. Let $\L$ (resp. $\L'$) be the canonical line bundle
and $p$ (resp. $p'$) be the canonical projection of $P/T$ (resp. $P'/T$).
An easy computation shows that $\bar \nu^*(\L')=\L^{\otimes r}$.
Recall from \ref{Panin} that the Thom class of $N$ (resp. $L^{\otimes r}$)
is equal to $t(N)=c_1(\L^\vee \otimes p^{-1} N)$
(resp. $t(N^{\otimes r})=c_1(\L'^\vee \otimes p'^{-1} N^{\otimes r})$).
Thus, from the second point of proposition \ref{nilpotence&FGL},
 $\bar \nu^* t(N^{\otimes r})=[r]_F \cdot t(N)$.
This latter class is zero on $P-T$, thus we get the relation
$[r]_F \cdot t(N)=\rho.t(N)$ in $H^{2,1}(P)$.
The conclusion now follows according to the computation of the Thom
isomorphism \ref{Thom_class}.

To finish the proof with that lemma, we remark that
 $\rho=r(T;\nu,1_T)=r(T;f,g)$.
\end{proof}

\begin{cor} \label{pullback_fund_class}
Let $(f,g):(Y,T) \rightarrow (X,Z)$ be a morphism of smooth closed
pairs of codimension $1$.
We assume $T$ admits an ample line bundle and $(f,g)$ satisfies
condition (Special). Let $(T_i)_{i \in I}$ be the connected
components of $T$, and $r_i \in \NN$ be the ramification index of 
$f$ along $T_i$.

Then, for any $i \in I$, the fundamental class $\eta_Y(T_i)$ is 
nilpotent and
$$
f^*(\eta_X(Z))=\sum_{i \in I} [r_i]_F \cdot \eta_T(T_i)
$$
where $[r_i]_F$ is the power series equal to the $r_i$-th formal
sum with respect to the formal group law $F$.
\end{cor}
\begin{proof}
Let $i:Z \rightarrow X$ and $j_i:T_i \rightarrow Y$ be the canonical
immersions. We simply put $\rho_i=r(T_i;f,g) \in H^{0,0}(T_i)$.
In cohomology, the preceding theorem applied to $(f,g)$ gives the
relation
$$
f^*i_*(z)=\sum_{i \in I} j_{i*}(\rho_i \gcup g_i^*(z))
$$
for $z \in H^{**}(Z)$.
Applied with $z=1$, this gives $f^*(\eta_X(Z))=\sum_i j_{i*}(\rho_i)$.
Recall from lemma \ref{mult_divisors} that 
$\rho_i=\pur{Y,T_i}^*([r_i]_F \cdot \bar \eta_Y(T_i))$.
Tautologically, the composition $j_{i*} \pur{Y,T_i}^*$
is equal to the canonical morphism 
$H^{**}_{T_i}(Y) \rightarrow H^{**}(Y)$ simply obtained
by functoriality. For conclusion, it is sufficient to recall
this latter is a morphism of $A$-algebra 
(cf paragraph \ref{product_with_support}).
\end{proof}

\begin{rem}
In the previous corollary, the integers $r_i$ can be understood as follows:
locally, $Z$ is parametrized by a $S$-regular function 
 $a:X \rightarrow \AA^1$.
Then, $(f,g)$ is special if $a \circ f$ can be written locally
$u.\prod_{i \in I} b_i^{r_i}$ where $u$ is a unit and 
$b_i:Y \rightarrow \AA^1$
is a $S$-regular function parametrizing $T_i$ -- this expression should 
remain the same when we change any of the parameters $b_i$ or $a$.
\end{rem}

%% file: associativity.tex
\subsection{Crossing Gysin triangles}

The following lemma will be the key point of the main result of this section.
Though it will appear finally as a particular case, we begin by proving it 
to enlighten the proof of theorem \ref{thm:assoc}.
\begin{lm}
\label{lm:proj&Gysin_mot2}
Let $Z$ be a smooth scheme, $E$ and $E'$ be vector bundles over $Z$
of respective ranks $n$ and $m$.
Put $Q=\PP(E \oplus 1)$, $Q'=\PP(E' \oplus 1)$ and $P=Q \times_Z Q'$.

Consider the fundamental class (see paragraph \ref{fund_class}) $\eta_P(Z)$ 
 (resp. $\eta_P(Q)$, $\eta_P(Q')$) of the canonical embedding of $Z$ 
 (resp. $Q$, $Q'$) in $P$, as an element of $H^{**}(P)$.

Then $\eta_P(Z)=\eta_P(Q) \ncup \eta_{P}(Q')$.
\end{lm}
\begin{proof}
Put $d=n+m$.
Let $\bar \eta_P(Z)$ be the localised fundamental class of $Z$ in $P$
(cf paragraph \ref{fund_class}).
Consider the deformation diagram \eqref{1st_def_diag} for
the closed pair $(P,Z)$, with $B=B_Z(\AA^1_P)$~:
$$
(P,Z) \xrightarrow{\bar \sigma_1} (B,\AA^1_Z)
 \xleftarrow{\bar \sigma_0} (P,Z).
$$
As $\bar \sigma_0^*$ and $\bar \sigma_1^*$ are isomorphisms,
$\bar \eta_P(Z)$ is uniquely determined by the class 
$\bar t=\bar \sigma_1^*(\bar \eta_P(Z))$
and $\bar t$ is uniquely determined by the fact that
$\bar \sigma_0^*(\bar t)$ corresponds to the Thom class 
$t(E \oplus E')$ in $H^{2d,d}(P)$.

Consider the divisor
$D=B_Z(\AA^1 \times \PP(E) \times \PP(E' \oplus 1))$
(resp. $D'=B_Z(\AA^1 \times \PP(E \oplus 1) \times \PP(E'))$ in $B$
and the class $c=-c_1(D)$ (resp. $c'=-c_1(D')$) in $H^{2,1}(B)$.
Let $\pi$ be the canonical projection of $P/X$.
We define a cohomology class in $H^{2,1}(B)$~:
$$
t=\left(\sum_{0 \leq i \leq n} \pi^*(c_i(E)) \ncup c^{n-i}\right)
 \ncup \left(\sum_{0 \leq j \leq m} \pi^*(c_j(E')) \ncup {c'}^{m-j}\right).
$$
Then $t$ vanishes on $B-\AA^1_Z$ and, by construction, its pullback by $\bar \sigma_0$
is equal to $t(E \oplus E')$. Thus $t$ corresponds to the class $\bar t$
mentionned above, through the map
 $H^{2d,d}_{\AA^1_Z}(B) \rightarrow H^{2d,d}(B)$.
The computation of its pullback by $\bar \sigma_1$ gives the desired formula.
\end{proof}

\begin{rem}
Another way to obtain this lemma is to apply corollary
\ref{fundamental_class_as_chern_class} with $X=P$
and $E=\xi \times_Z \xi'$ where $\xi$ (resp. $\xi'$) is the universal quotient 
bundle of $Q$ (resp. $Q'$) --- compare with remark \ref{Thom&quotient}.
\end{rem}

\begin{thm}
\label{thm:assoc}
Consider a cartesian square of smooth schemes
$\xymatrix@=10pt{
Z\ar^k[r]\ar_l[d] & Y'\ar^j[d] \\
Y\ar^i[r] & X
}
$
such that $i$,$j$,$k$,$l$ are closed immersions of respective pure
codimension $n$, $m$, $s$, $t$. We put $d=n+s=m+t$ and consider the
closed immersion $h:(Y-Z) \longrightarrow (X-Y')$ induced by $i$.

Then, in the following diagram~:
$$
\xymatrix@R=16pt@C=30pt{
\M(X)\ar^{j^*}[r]\ar_{i^*}[d]\ar@{}|{(1)}[rd]
 & \M(Y')(m)[2m]\ar^/-2pt/{\partial_{X,Y'}}[r]\ar^{k^*}[d]
    \ar@{}^{(2)}[rd]
 & \M(X-Y')[1]\ar^{h^*}[d] \\
\M(Y)(n)[2n]\ar_{l^*}[r]
 & \M(Z)(d)[2d]\ar^/-5pt/{\partial_{Y,Z}}[r]
     \ar_{\partial_{Y',Z}}[d]\ar@{}|{(3)}[rd]
 & \M(Y-Z)(n)[2n+1]\ar^{\partial_{X\!-\!Y',Y\!-\!Z}}[d] \\
& \M(Y-Z)(m)[2m+1]\ar_/6pt/{\partial_{X\!-\!Y,Y\!-\!Z}}[r]
 & \M(X-Y \cup Y')[2]
}
$$
squares (1) and (2) are commutative and square (3) is anti-commutative.
\end{thm}
\begin{proof} Put $Y''=Y \cup Y'$.
Using axiom (Loc) and (Sym)(c), we obtain the following diagram~:
$$
\xymatrix@C=8pt@R=10pt{
(\mathcal D):\M(X-Y'')\ar[r]\ar[d]
 & \M(X-Y)\ar[r]\ar[d]
 & \MD{X-Y}{X-Y''}\ar[r]\ar[d]
 & \M(X-Y'')[1]\ar[d] \\
\M(X-Y')\ar[r]\ar[d]
 & \M(X)\ar[r]\ar[d]\ar@{}|{(1)}[rd]
 & \MD X {X-Y'}\ar[d]\ar[r]\ar@{}|{(2)}[rd]
 & \M(X-Y')[1]\ar[d] \\
\MD{X-Y'}{X-Y''}\ar[r]\ar[d]
 & \MD X {X-Y}\ar[r]\ar[d]
 & \MD{X/X-Y}{X-Y'/X-Y''}
    \ar[r]\ar[r]\ar[d]\ar@{}|{(3)}[rd]
 & \MD{X-Y'}{X-Y''}[1]\ar[d] \\
\M(X-Y'')[1]\ar[r]
 & \M(X-Y)[1]\ar[r]
 & \MD{X-Y}{X-Y''}[1]\ar[r]
 & \M(X-Y'')[2],
}
$$
in which any line or any row is a distinguished triangle,
every square is commutative except square (3)
 which is anticommutative.

We put $\M(X;Y,Y')=\MD{X/X-Y}{X-Y'/X-Y''}$
for short.
The proof will consist in
constructing a purity isomorphism  
$\pur{X;Y,Y'}:\M(X;Y,Y') \rightarrow \M(Z)(d)[2d]$ which
satisfies the following properties~:
\begin{enumerate}
\item[(i)] \textit{Functoriality~:} The morphism $\pur{X;Y,Y'}$ 
is functorial
with respect to morphisms in $X$ which are transversal to $Y$,
$Y'$ and $Z$ respectively. 
\item[(ii)] \textit{Symmetry~:} The following diagram is 
commutative~:
$$
\xymatrix@R=12pt@C=0pt{
\M(X;Y,Y')\ar_/-8pt/{\pur{X;Y,Y'}}[rd]\ar^{\epsilon}[rr] &
 & \M(X;Y',Y)\ar^/-8pt/{\ \ \pur{X;Y',Y}}[ld] \\
& \M(Z)(d)[2d],
}
$$
where $\epsilon$ is the isomorphism given in axiom (Sym).
\item[(iii)] \textit{Compatibility~:} The following diagram 
is commutative~:
$$
\xymatrix@C=16pt@R=24pt{
\M(\frac{X-Y'}{X-\Omega})\ar[r]\ar|{\pur{X-Y',Y-Z}}[d]
 & \M(\frac X {X-Y})\ar[r]\ar|{\pur{X,Y}}[d]
 & \M(X;Y,Y')\ar[r]
     \ar|{\pur{X;Y,Y'}}[d]
 & \M(\frac{X-Y'}{X-\Omega})[1]\ar|{\pur{X-Y',Y-Z}}[d] \\
\M(Y-Z)(n)[2n]\ar[r]
 & \M(Y)(n)[2n]\ar^/1pt/{j^*}[r]
 & \M(Z)(d)[2d]\ar^/-10pt/{\partial_{Y,Z}}[r]
 & \M(Y-Z)(n)[2n+1]
}
$$
\end{enumerate}
With this isomorphism, we can deduce the three relations of the
theorem by considering squares $(1)$, $(2)$, $(3)$ in the above
diagram when we apply the evident purity isomorphisms where we can.
We then are reduced to construct the isomorphism and to prove the
above relations. The difficult one is the second relation because we
have to show that two isomorphisms in a triangulated category
are equal. This forces to be very precise in the construction of the
isomorphism.

We use a construction analog to the construction of the purity 
isomorphism in proposition \ref{prop:purity}.
The first deformation space (cf paragraph \ref{def_diags}) 
for the pair $(X,Y)$ is $B=B_Y(\AA^1_X)$.
We let $P=P_YX$ be the projective completion of the normal bundle of $(X,Y)$.
Consider also the closed pair $(U,V)=(X-Y',Y-Z)$. 
The analog deformation space for $(U,V)$ is $B_U=B \times_X U$
and the projective completion of its normal bundle is $P_V=P \times_Y V$.

The deformation diagrams \eqref{1st_def_diag} for $(X,Y)$ and $(U,V)$ 
induce the following morphisms
$$
\M(X;Y,Y')\!=\!\MD{X/X-Y}{U/U-V}
 \xrightarrow{\bar \sigma_{1*}} \MD{B/B-\AA^1_Y}{B_U/B_U-\AA^1_V}
 \xleftarrow{\bar \sigma_{0*}} \MD{P/P-Y}{P_V/P_V-V}
$$
and the axiom (Loc) together with the purity theorem \ref{prop:purity}
shows $\bar \sigma_{0*}$ and $\bar \sigma_{1*}$ are isomorphisms.

Using the compatibility of the Gysin triangle with the projective bundle
isomorphism (cf corollary \ref{Gysin&proj_bdl_iso}), we
obtain a commutative diagram~:
$$
\xymatrix@C=26pt@R=16pt{
\M(P_V/P_V-V)\ar[r]
 & \M(P/P-U)\ar[r]
 & \MD{P/P-Y}{P_V/P_V-V}\ar^/16pt/{+1}[r]
 & \\
\M(P_V)\ar[r]\ar@{=}[d]\ar[u]
 & \M(P)\ar[r]\ar@{=}[d]\ar[u]
 & \M(\frac{P}{P_V})\ar^/16pt/{+1}[r]\ar[u]
 & \\
\M(P_V)\ar[r]
 & \M(P)\ar[r]
 & \M(P_Z)(s)[2s]\ar^/18pt/{+1}[r]
     \ar_/-3pt/{\pur{P,P_Z}^{-1}}[u]
 & \\
\M(Y-Z)(n)[2n]\ar[r]\ar^{\lef n {P_V}}[u]
 & \M(Y)(n)[2n]\ar^{j^*}[r]\ar^{\lef n {P}}[u]
 & \M(Z)(d)[2d]\ar^/14pt/{+1}[r]
    \ar_{\lef n {P_Z}(s)[2s]}[u]
 &
}
$$
The composite of the vertical maps thus gives a morphism of triangles.
Using property (2) of proposition \ref{prop:purity}, the first two maps
of this morphism are isomorphisms and so is the third.
This last isomorphism together with the maps $\bar \sigma_{1*}$ and $\bar \sigma_{0*}$
gives the desired isomorphism $\pur{X;Y,Y'}$.

Note that property (iii) is obvious by construction. 
Property (i) is easily obtained as in proposition \ref{transversal}.

Thus we have only to prove property (ii).
First of all, we remark that the previous construction implies immediately the
commutativity of the diagram~:
$$
\xymatrix@R=12pt@C=0pt{
\M(X;Y,Y')\ar_/-8pt/{\pur{X;Y,Y'}}[rd]\ar^{\alpha_{(X;Y,Y')}}[rr] &
 & \M(X;Y,Z)\ar^/-8pt/{\ \ \pur{X;Y,Z}}[ld] \\
& \M(Z)(d)[2d],
}
$$
where $\alpha_{(X;Y,Y')}$ is induced by the evident open immersions.

Consider the following map
$$
\beta_{(X;Y,Y')}:
\M_Z(X) \xrightarrow{\pi_{(X,Y,Z)}} \M(X;Y,Z)
 \xrightarrow{\alpha_{(X;Y,Y')}^{-1}} \M(X;Y,Y')
$$
where $\pi_{(X,Y,Z)}$ is obtained by functoriality as usual -- it is an 
isomorphism from axioms (Loc) and (Sym).
Using the coherence axiom (Sym)(b), one checks that 
the following diagram is commutative
$$
\xymatrix@R=12pt@C=0pt{
& \M_Z(X)\ar_/8pt/{\beta_{(X;Y,Y')}}[ld]\ar^/8pt/{\ \ \beta_{(X;Y',Y)}}[rd]
 & \\
\M(X;Y,Y')\ar^\epsilon[rr] & & \M(X;Y',Y). \\
}
$$
Thus, it will be sufficient to prove the commutativity of the
following diagram~:
$$
\xymatrix@R=12pt@C=0pt{
\M_Z(X)\ar_/-8pt/{\pur{X,Z}}[rd]\ar^{\pi_{(X,Y,Z)}}[rr]
 & \ar@{}|/-3pt/{(*)}[d]
 & \M(X;Y,Z)\ar^/-8pt/{\ \ \pur{X;Y,Z}}[ld] \\
& \M(Z)(d)[2d].
}
$$
In the remainings of the proof, 
we consider the triples of smooth schemes
$(X',Y',Z')$ such that $Z' \subset Y' \subset X'$ are closed subschemes.
A morphism of triples $(f,g,h):(X'',Y'',Z'') \rightarrow (X',Y',Z')$
is a morphism of schemes $f:X'' \rightarrow X'$ which is transversal
to $Y'$ and $Z'$, and such that $Y''=f^{-}(Y')$, $Z''=f^{-1}(Z')$.
Using the functoriality of $\pur{X;Y,Z}$, we remark that
diagram $(*)$ is natural with respect to morphisms of triples.

We use the notations of paragraph \ref{def_diags}.
We also put $B(X',Z'):=B_{Z'}(X')$, 
for a closed pair $(X',Z')$, 
and so on for the other schemes depending on a closed pair, 
to clarify the following considerations.
We consider the evident closed pair $(D_ZX,D_ZX|_Y)$ 
and we put $D(X,Y,Z)=D(D_ZX,D_ZX|_Y)$. This scheme is
in fact fibered over $\AA^2$. The fiber over $(1,1)$ is $X$ and the fiber
over $(0,0)$ is $B(B_ZX \cup P_ZX,B_ZX|_Y \cup P_ZX|_Y)$.
In particular, the $(0,0)$-fiber contains the scheme $P(P_ZX,P_ZY)$.

We now put~: $D=D(X,Y,Z)$, $D'=D(Y,Y,Z)$. 
Remark that $D(Z,Z,Z)=\AA^2_Z$. 
Similarly, we put $P=P(P_ZX,P_ZY)$, $Q=P_ZY$.
Remark finally that if we consider $Q'=P_YX|_Z$, then\footnote{This 
is equivalent to the canonical isomorphism $N(N_ZX,N_ZY)=N_ZY \oplus N_YX|_Z$.} 
$P=Q \times_Z Q'$.

From the above description of fibers, we obtain a deformation diagram
of triples~:
$$
(X,Y,Z) \rightarrow (D,D',\AA^2_Z) \leftarrow (E,G,Z).
$$
Note that these morphisms are on the smaller closed subscheme the
$(0,0)$-section and $(1,1)$-section of $\AA^2_Z$ over $Z$, denoted 
respectively by $s_1$ and $s_0$.
Now we apply these morphisms to diagram $(*)$ obtaining the following 
commutative diagram~:
$$
\xymatrix@R=20pt@C=-20pt{
\M_Z(X)\ar|/-15pt/{\pur{X,Z}}[dd]\ar^{\pi_{X,Y,Z}}[rd]\ar[rr]
 && \M_{\AA^2_Z}(D)\ar|/-15pt/{\pur{D,\AA^2_Z}}[dd]\ar[rd]
 && \M_Z(P)\ar|/-15pt/{\pur{P,Z}}[dd]\ar^{\pi_{P,Q,Z}}[rd]\ar[ll]
 \\
& \M(X;Y,Z)\ar|/-0pt/{\quad \pur{X,Y,Z}}[ld]\ar[rr] &
 & \M(D;D',\AA^2_Z)\ar|/-0pt/{\quad \pur{D,D',Z}}[ld] &
 & \M(P;Q,Z)\ar|/-0pt/{\quad \pur{P,Q,Z}}[ld]\ar[ll]
 \\
\M(Z)(d)[2d]\ar@<-3pt>@{}_/-2pt/{{s_1}_*}[rr]\ar[rr]
 && \M(\AA^2_Z)(d)[2d]
 && \M(Z)(d)[2d].\ar^/-2pt/{{s_0}_*}[ll]
}
$$
The square parts of this prism are commutative. As morphisms
${s_1}_*$ and ${s_{0}}_*$ are isomorphisms, the
commutativity of the left triangle is equivalent to the
commutativity of the right one.

Thus, we are reduced to the case of the smooth triple $(P,Q,Z)$.
Now, using the canonical split epimorphism
 $\M(P) \rightarrow \M(P/P-Z)$, we are reduced to prove
the commutativity of the diagram~:
$$
\xymatrix@R=-6pt@C=50pt{
\M(P)\ar_{i^*}[dd]\ar[rd] & \\
 & \MD{P/P-Q}{P-Z/P-Q}
   \ar^/-0pt/{\quad \pur{P,Q,Z}}[ld] \\
\M(Z)(d)[2d]
}
$$
where $i:Z \rightarrow P$ denotes the canonical closed immersion.

Using property (iii) of the
isomorphism $\pur{P,Q,Z}$, we are finally reduced to prove the
commutativity of the triangle
$$
\xymatrix@R=-4pt@C=50pt{
\M(P)\ar_{i^*}[rd]\ar^{j^*}[rr] & 
 & \M(Q)(n)[2n]\ar^{k^*}[ld] \\
& \M(Z)(d)[2d] &
}
$$
where we considered $Z \xrightarrow k Q \xrightarrow j P$ the
canonical closed embeddings. This now simply follows 
from paragraph \ref{fund_class} and lemma \ref{lm:proj&Gysin_mot2}.
\end{proof}

As a corollary (apply commutativity of square (1) in the case $Y'=Z$),
we get the functoriality of the Gysin morphism of a closed immersion~:
\begin{cor}
\label{cor:assoc}
Let $Z \xrightarrow l Y \xrightarrow i X$ be closed immersions between
smooth schemes of respective pure codimension $n$ and $m$.

Then, $l^* \circ i^*=(i \circ l)^*$ as a morphism
$\M(X) \rightarrow \M(Z)(n+m)[2(n+m)]$.
\end{cor}

A corollary of this result, using lemma \ref{lm:Gysin&product},
is the compatibility of the Gysin morphism with products~:
\begin{cor} \label{gysin&tensor}
Consider a closed immersion $i:Z \rightarrow X$ 
(resp. $k:T \rightarrow Y$) between smooth schemes of pure codimension $n$
(resp. $m$).

Then $(i \times k)^*=i^* \otimes k^*$ as a morphism\footnote{
When we identify $\M(Z)(n)[2n] \otimes \M(T)(m)[2m]$ with
$\M(Z) \otimes \M(T)(n+m)[2(n+m)]$
through the canonical isomorphism.}~:
\begin{center}
$\M(X) \otimes \M(Y) \rightarrow \M(Z) \otimes \M(T)(n+m)[2(n+m)]$.
\end{center}
\end{cor}

\begin{rem}
In the hypothesis of the previous corollary,
we obtain in terms of fundamental classes~:
$$\eta_{X \times Y}(Z \times T)=\eta_X(Z) \otimes \eta_Y(T).$$
\end{rem}

We also obtain a result of intersection of fundamental classes
 in the case of smooth cycles~:
\begin{cor}
Let $X$ be a smooth scheme,
 $Z$ and $T$ be smooth closed subschemes of $X$. 
We assume that~:
\begin{enumerate}
\item The intersection of $Z$ and $T$ in $X$ is proper.
\item There is a closed subscheme $W$ in $Z \cap T$
 which is smooth, homeomorphic to $Z \cap T$
  and admits an ample line bundle.
\item The induced morphism of closed pairs $(T,W) \rightarrow (X,Z)$
 satisfies condition (Special).
\end{enumerate}
Let $\nu_{X,W}:H^{**}_W(X) \rightarrow H^{**}(X)$ be the canonical morphism.
According to $(Add')$, $H^{**}_W(X)=\bigoplus_{i \in I} H^{**}_{W_i}(X)$
where $(W_i)_{i \in I}$ are the connected components of $W$.
For any $i \in I$, 
we can consider the localised fundamental class $\bar \eta_X(W_i)$
 as an element of $H^{**}_{W}(X)$ (see paragraph \ref{fund_class}). 
We let $\rho_i \in H^{0,0}(W_i)$ be the $F$-intersection multiplicity of $W_i$ 
in $Z \cap T$ (see definition \ref{inter_mult}).
Then, 
$$
\eta_X(Z) \ncup \eta_X(T)
 =\nu_{X,W}\Big({\sum}_{i \in I} \ \rho_i.\bar \eta_X(W_i)\Big)
$$
using the $H^{0,0}(W)$-module structure of $H^{**}_W(X)$ obtained through the
purity isomorphism.
\end{cor}
\begin{proof}
We apply theorem \ref{ramification} to the obvious square~:
$$
\xymatrix@=10pt{
W\ar^{\nu'}[r]\ar_g[d] & T\ar^f[d] \\
Z\ar^\nu[r] & X.
}
$$
For any $i \in I$, we let $\nu'_i$ (resp. $\mu_i$) be the immersion
 of $W_i$ in $T$ (resp. $X$).
We thus obtain the formula in $H^{**}(T)$~:
$f^*\nu_*(1)=\sum_{i \in I} \nu'_{i*}(\rho_i)$. \\
Applying $f_*$ to this formula
 and using corollary \ref{second_proj_formula} for the left hand side,
  corollary \ref{cor:assoc} for the right hand side,
   we obtain
    $\eta_X(Z) \gcup \eta_X(T)=\sum_{i \in I} \mu_{i*}(\rho_i)$.
By the very definition now,
 $\mu_{i*}(\rho_i)=\nu_{X,W}(\rho_i.\bar \eta_X(W_i))$.
\end{proof}

%% file: duality.tex
\section{Duality and Gysin morphism}

\subsection{Preliminaries}

For the rest of the section, 
we fix a monoidal category $\C$ with unit $\un$.
\begin{df}
Let $M$ an object of $\C$.

We say $M$ is \emph{strongly dualizable} if the following conditions
are fulfilled~:
\begin{enumerate}
\item The functor $M \otimes .$ admits a right adjoint $\uHom(M,.)$.
\item For any object $N$ of $\C$, consider the map
$$
M \otimes \uHom(M,\un) \otimes N \xrightarrow{ad \otimes 1_N} N
$$
induced by the evident adjunction morphism.
Then the adjoint map
$$
\uHom(M,\un) \otimes N \rightarrow \uHom(M,N)
$$
is an isomorphism.
\end{enumerate}
\end{df}
This definition coincides with definition 1.2 of \cite{DP}.
Obviously, strongly dualizable objects are stable by finite sums
and tensor product.
Remark also that any invertible object of $\C$
 for the tensor product is \emph{a fortiori} strongly dualizable.

\begin{df} \label{df:strong_dual}
Consider an object $M$ of $\C$.

A \emph{strong dual} of $M$ is an object $M^\vee$ of $\C$
and two morphisms $\mu:M \otimes M^\vee \rightarrow \un$,
$\epsilon:\un \rightarrow M^\vee \otimes M$
 such that the following composites
\begin{align*}
(i) & \quad M \xrightarrow{1 \otimes \epsilon} M \otimes M^\vee \otimes M
 \xrightarrow{\mu \otimes 1} M \\
(ii) & \quad M^\vee \xrightarrow{\epsilon \otimes 1} M^\vee \otimes M \otimes M^\vee
 \xrightarrow{1 \otimes \mu} M^\vee
\end{align*}
are the identity morphisms.
\end{df}
The conditions of the definition imply that $M^\vee \otimes .$
 is right adjoint to $M \otimes .$ and the natural transformations
  $\epsilon \otimes .$ and $\mu \otimes .$ are the adjunction transformations.
Moreover, $M$ is strongly dualizable
 as condition (2) of the first definition simply follows from
the structural isomorphism
 $(M^\vee \otimes \un) \otimes N \simeq M^\vee \otimes N$
(see also \cite[1.3]{DP}).

Remark we also obtain that $. \otimes M^\vee$ is left adjoint to
 $. \otimes M$ with natural transformation $. \otimes \mu$
  and $. \otimes \epsilon$.
This gives the following reciprocal
  isomorphisms which we describe for future needs~:
\begin{equation} \label{duality&adjunction}
\begin{split}
\Hom_\C(M^\vee,\E) \rightarrow \Hom_\C(\un,\E \otimes M), 
 & \varphi \mapsto (\varphi \otimes 1_M) \circ \epsilon \\
\Hom_\C(\un,\E \otimes M) \rightarrow \Hom_\C(M^\vee,\E),
 & \psi \mapsto (1_\E \otimes \mu) \circ (\psi \otimes 1_{M^\vee}),
 \end{split}
\end{equation}
where $\E$ is any object of $\C$.

The following lemma gives some precisions on the relation between
 "strongly dualizable" and "strong dual"~:
\begin{lm} \label{lm:strong_dual}
Consider a strongly dualizable object $M$ of $\C$.
Let $M^\vee$ be an object of $\C$.

Consider the following sets~:
\begin{enumerate}
\item Couples of morphisms $\mu:M \otimes M^\vee \rightarrow \un$ and 
$\epsilon:M^\vee \otimes M \rightarrow \un$ such that $(M^\vee,\mu,\epsilon)$
is a strong dual of $M$.
\item Morphisms $\mu:M \otimes M^\vee \rightarrow \un$ such that the adjoint map
$\phi:M^\vee \rightarrow \uHom(M,\un)$ is an isomorphism.
\end{enumerate}
We associate to any morphism $\mu$ in (2) the following composite
$$
\epsilon_\mu:\un \xrightarrow{ad'} \uHom(M,M)
 \rightarrow \uHom(M,\un) \otimes M
  \xrightarrow{\phi^{-1} \otimes 1} M^\vee \otimes M
$$
where the first map is the evident adjunction morphism and the second one
is induced by the isomorphism obtained by the property of the
strongly dualizable object $M$.

Then $(\mu,\epsilon_\mu)$ is an element of $(1)$ and the application
$$
(2) \rightarrow (1), \mu \mapsto (\mu,\epsilon_\mu)
$$
is a bijection.
\end{lm}
We left the easy check to the reader.

\begin{df} \label{df:transpose}
Let $M$ (resp. $N$) be an object of $\C$
 and $(M^\vee,\mu_M,\epsilon_M)$ (resp. $(N^\vee,\mu_N,\epsilon_N)$)
  be a strong dual of $M$ (resp. $N$).

For any morphism $f:M \rightarrow N$, we define the transpose morphism
of $f$ (with respect to the chosen strong duals) as the composite
$$
\tra f:N^\vee \xrightarrow{\mu_M \otimes 1} M^\vee \otimes M \otimes N^\vee
 \xrightarrow{1 \otimes f \otimes 1} M^\vee \otimes N \otimes N^\vee
  \xrightarrow{1 \otimes \epsilon_N} M^\vee.
$$
\end{df}
Remark that the morphism $\tra f$ in the previous definition is characterized
by either one of the next two properties~:
\begin{enumerate}
\item[(i)] The following diagram is commutative~:
$$
\xymatrix@R=13pt@C=22pt{
M \otimes N^\vee\ar^{f \otimes 1}[r]\ar_{1 \otimes \tra f}[d]
 & N \otimes N^\vee\ar^{\mu_N}[d] \\
M \otimes M^\vee\ar^/8pt/{\mu_M}[r] & {}\un.
}
$$
\item[(ii)] The following diagram is commutative~:
$$
\xymatrix@C=42pt@R=14pt{
N^\vee\ar^{\tra f}[r]\ar[d] & M^\vee\ar[d] \\
\uHom(N,\un)\ar^{\uHom(f,\un)}[r] & \uHom(M,\un)
}
$$
where the vertical maps are induced by adjunction from $\mu_N$
 and $\mu_M$ -- cf lemma \ref{lm:strong_dual}.
\end{enumerate}

\subsection{The projective bundle case}

Fix an integer $n \geq 0$.
Using the projective bundle theorem \ref{th:projbdl}
 and axiom (Stab),
we obtain that the motive $\M(\PP^n)$ is strongly dualizable, 
as a finite sum of invertible motives.

Let $\L_n$ be the canonical line bundle on $\PP^n$,
$c'=c_1(\L_n)$.
From the projective bundle theorem \ref{th:projbdl},
$c'$ is a generator of the $A$-algebra $H^{**}(\PP^n)$.
Let $c=c_1(\L^\vee_n)$. According to paragraph \ref{fgl},
$c=m(c')=-c' \mod c'^2$ where $m$ is the inverse series
 associated to the formal group law $F$.
Thus, the class $c$ is still a generator of $H^{**}(\PP^n)$ and
also satisfies the relation $c^{n+1}=0$.
In all this section on duality, we systematically use this
generator.

We consider the following morphism
$$
\mu_{n}:\M(\PP^n) \otimes \M(\PP^n)(-n)[-2n]
 \xrightarrow{\delta^*} \M(\PP^n)
  \xrightarrow{p_*} \un
$$
where $p:\PP^n \rightarrow S$ is the canonical projection
 and $\delta:\PP^n \rightarrow \PP^n \times \PP^n$ the diagonal embedding
  of $\PP^n/S$.

If we consider this morphism as a cohomological class in 
 $H^{2n,n}(\PP^n \times \PP^n)$, 
it is the fundamental class 
$\eta_{\PP^n \times \PP^n}(\PP^n)=\delta_*(1) \in H^{2n,n}(\PP^n \times \PP^n)$
of the diagonal.
Using the projective bundle theorem \ref{th:projbdl},
it can be written
$$
\eta_{\PP^n \times \PP^n}(\PP^n)
=\sum_{0 \leq i,j \leq n} \eta_{i,j}^{(n)}.c^i \ncup d^j
$$
where $\eta_{i,j}^{(n)}$ is an element in
$A^{2(n-i-j),n-i-j}$ and $c$ (resp. $d$) is the first
Chern class of the canonical dual line bundle on the first 
(resp. second) factor of $\PP^n \times \PP^n$.

We define the $(n+1)$-dimensional square matrix
 $M_n=\big(\eta_{i,j}^{(n)}\big)_{0 \leq i,j \leq n}$
over the bigraded ring $A$.
Note that $M_n$ is symmetric.
Remark finally that the morphism induced by adjunction from 
$\mu_{n}$ gives by another application of theorem \ref{th:projbdl}
a morphism
$$
\bigoplus_{i=0}^n \un(n-i)[2(n-i)]
 \rightarrow \bigoplus_{j=0}^n \un(j)[2j]
$$
whose matrix is precisely $M_{n}$.
\begin{lm} \label{pre_comput_cobordism_c}
For any integer $i \geq 0$, put $\eta_i=\eta_{ii}^{(i)} \in A^{-2i,-i}$.
The matrix $M_n$ has the form
$$
\left(
\raisebox{0.5\depth}{
\xymatrix@=1ex{
0\ar@{.}[rr]\ar@{.}[dd]
 & & 0\ar@{.}[lldd]
 & 1\ar@{-}[lllddd] \\
 & & & \eta_1\ar@{-}[lldd]\ar@{.}[dd] \\
0 & & & \\
1 & \eta_1\ar@{.}[rr] & & \eta_n
}
}
\right)
$$
\end{lm}
\begin{proof}
First remark the lemma is clear when $n=0$.

Consider the canonical embedding
$\sigma:\PP^n \rightarrow \PP^{n+1}$.
We apply the excess intersection formula \ref{excess} in the
case of the following square
$$
\xymatrix@R=15pt@C=22pt{
{\PP^n}\ar^/-9pt/{\delta}[r]\ar_\sigma[d]
 & {\PP^n \times \PP^n}\ar^{\sigma \times \sigma}[d] \\
{\PP^{n+1}}\ar^/-9pt/{\delta'}[r] & {\PP^{n+1} \times \PP^{n+1}}.
}
$$
In this case, the excess of codimension is $1$
and the excess intersection bundle on $\PP^n$ is the canonical
dual line bundle $\lambda_n^\vee$. 
Proposition \ref{excess} then gives the formula
$(\sigma \times \sigma)^*(\delta'_*(1))=\delta_*(c_1(\lambda_n^\vee))$.

The projection on the first factor
 $p_1:\PP^n \times \PP^n \rightarrow \PP^n$
gives a retraction of $\delta$, and consequently, 
 $\delta_*(c_1(\lambda_n^\vee))=c \ncup \delta_*(1)$.
Thus the previous relation reads~:
$$
\sum_{0 \leq i,j \leq n} \eta_{i,j}^{({n+1})}.c^i \ncup d^j
 =\sum_{0 \leq i,j \leq n} \eta_{i,j}^{(n)}.c^{i+1} \ncup d^j
$$
with the notations which precede the lemma. This in turn gives
the relations
$$
\begin{cases}
 \eta_{0,j}^{(n+1)}=0 & \text{if } 0 \leq j \leq n, \\
 \eta_{i,j}^{(n+1)}=\eta_{i-1,j}^{(n)}
  & \text{if } 0 < i \leq n\text{ and }0 \leq j \leq n,
\end{cases}
$$
which allow to conclude by induction on the integer $n$.
\end{proof}

As a corollary, we obtain from lemma \ref{lm:strong_dual} that
$\mu_n:\M(\PP^n) \otimes \M(\PP^n)(-n)[-2n] \rightarrow \un$
turns $\M(\PP^n)(-n)[-2n]$ into a strong dual of $\M(\PP^n)$.
\begin{df} \label{gysin4projection}
We define the Gysin morphism
 $p^*:\un \rightarrow \M(\PP^n)(-n)[-2n]$
associated to the projection
 $p:\PP^n \rightarrow S$ as the transpose of the morphism
  $p_*:\M(\PP^n) \rightarrow \un$ with respect to the strong duality 
   on $\M(\PP^n)$ induced by $\mu_n$.
   
Moreover, for any smooth scheme $X$, considering the projection
$p_X:\PP^n_X \rightarrow X$, we define the Gysin morphism associated
to $p_X$ as the morphism 
$$
p_X^*:=1 \otimes p^*:\M(X) \rightarrow \M(\PP^n_X)(-n)[-2n].
$$
\end{df}
Using property (ii) after definition \ref{df:transpose}, we obtain
the following way to compute $p^*$.
Consider the inverse matrix
$$
M_n^{-1}=
\left(
\raisebox{0.5\depth}{
\xymatrix@=1ex{
\eta'_n\ar@{.}[rr]\ar@{.}[dd]
 & & \eta'_1\ar@{-}[lldd]
 & 1\ar@{-}[lllddd] \\
 & & & 0\ar@{.}[lldd]\ar@{.}[dd] \\
\eta'_1 & & & \\
1 & 0\ar@{.}[rr] & & 0
}
}
\right)
$$
where $\eta'_i \in A^{-2i,-i}$ is given by the determinant of the
matrix obtained by removing line $0$ and column $n-i$ from $M_n$
times $(-1)^i$. Then
$$
p^*:\un \rightarrow \bigoplus_{i=0}^n \un(i-n)[2(i-n)]
$$
is given by the vector 
 \raisebox{-0.8\depth}{$\begin{pmatrix}\eta'_n \\ \vdots \\ \eta'_1 \\ 1\end{pmatrix}$}.

Note we have the fundamental relation in $A^{-2n,-n}$~:
\begin{equation} \label{relation_cobordism}
\sum_{i=0}^n \eta_i \ncup \eta'_{n-i}=
\begin{cases}
1 & \text{if } n=0 \\
0 & \text{otherwise}
\end{cases}
\end{equation}

\begin{rem}
The coefficients $\eta_i$ and $\eta'_i$ will be determined
 in proposition \ref{fdl_class_diagonal} and corollary \ref{cor:Myschenko}.
\end{rem}

\subsection{The Gysin morphism associated to a projective morphism}

\subsubsection{Preliminary lemmas}

\begin{lm} \label{lm:Gysin1}
Fix a couple of integers $n,m \in \NN$ and a smooth scheme $X$.
Consider the projection morphisms
$$
\xymatrix@=12pt@C=20pt{
{\PP^n_X \times_X \PP^m_X}\ar^/10pt/{q'}[r]\ar_{p'}[d] & {\PP^n_X}\ar^{p}[d] \\
{\PP^m_X}\ar^{q}[r] & X.
}
$$
Then ${q'}^* p^*={p'}^* q^*$.
\end{lm}
Obvious from definition \ref{gysin4projection}.

\begin{lm} \label{lm:Gysin2}
Consider a closed immersion $i:Z \rightarrow X$ between smooth schemes
 and an integer $n \geq 0$.
Consider the pullback square
$$
\xymatrix@=10pt{
{}\PP^n_Z\ar^l[r]\ar_{q}[d] & {}\PP^n_X\ar^{p}[d] \\
Z\ar^i[r] & X.
}
$$
Then $l^* p^*=q^* i^*$.
\end{lm}
It follows easily from definition \ref{gysin4projection}
 and lemma \ref{lm:Gysin&product}.
 
\begin{lm} \label{lm:Gysin3}
Consider and integer $n \geq 0$ and a smooth scheme $X$.
Consider the canonical projection $p:\PP^n_X \rightarrow X$.

Then for any section $s:X \rightarrow \PP^n_X$ of $p$,
 we have~: $s^*p^*=1$.
\end{lm}
\begin{proof}
Recall from paragraph \ref{fund_class} that 
 $s^*=p_* \gcup_{\PP^n_X} (\pi_{X*} s^*)$ where $\pi_X:X \rightarrow S$
  is the structural morphism of $X/S$.
We easily obtain the following relation~:
$$
(p_* \gcup_{\PP^n_X} 1) \circ p^*=1 \gcup_X p^*.
$$
Thus~: $s^* p^*=1 \gcup_X (\pi_{X*} s^* p^*)$.

As $s$ is a section of $p$, it can be written $s=\nu \times 1_X$
 for a closed immersion $\nu:X \rightarrow \PP^n_S$.
Consider the following cartesian squares~:
$$
\xymatrix@R=14pt@C=24pt{
X\ar^/-5pt/s[r]\ar_\nu[d]
 & {}\PP^n_S \times X\ar^/5pt/p[r]\ar^{1_{\PP^n_S} \times \nu}[d]
 & X\ar^\nu[d] \\
{}\PP^n_S\ar^/-5pt/\delta[r] & {}\PP^n_S \times {}\PP^n_S\ar^/5pt/\pi[r] & {}\PP^n_S
}
$$
where $\delta$ is the diagonal embedding and $\pi$ the canonical 
projection on the second factor.
Using the projection formula for each square 
-- for the first square, this is \ref{transversal}, for the second square, it follows
easily from definition \ref{gysin4projection} -- we obtain~:
$\nu_*s^*p^*=\delta^*\pi^*\nu_*$.

As $\pi_{X*}=\pi_{\PP^n_S*} \nu_*$, we thus are reduced to prove
 $\delta^* \pi^*=1$. To conclude, the reader has the choice~:
\begin{enumerate}
\item A direct computation shows that the matrix of
$\pi^*$ (resp. $\delta^*$),
 through the projective bundle isomorphism \ref{th:projbdl}, is 
\begin{align*}
& \left(\delta_i^k.\eta'_{n-j}\right)_{(j,k) \in [0,n]^2, \ i \in [0,n]} \\
\text{resp. } & \left(\eta_{j+k-l-n}\right)_{l \in [0,n], \ (j,k) \in [0,n]^2}.
\end{align*}
The fundamental relation \eqref{relation_cobordism} allows to conclude.
\item Use definition \ref{df:transpose} to compute $\pi^*=1 \otimes p^*$
 in terms of the duality pairing $(\mu_n,\epsilon_{\mu_n})$
 (cf lemma \ref{lm:strong_dual}).
Apply the projection formula \ref{transversal} to compute
directly $\delta^* \pi^*$ ; the second relation of 
definition \ref{df:strong_dual} concludes.
\item Prove $\delta^*=\tra(\delta_*)$ using characterization (i) after definition
 \ref{df:transpose} (and the usual projection formula \ref{transversal}).
\end{enumerate}
\end{proof}

\subsubsection{Definition}

\begin{lm} \label{lm:Gysin4}
Consider a commutative diagram~:
$$
\xymatrix@C=14pt@R=-4pt{
& {}\PP^n_X\ar^p[rd] & \\
Y\ar^k[ru]\ar_i[rd] & & X \\
& {}\PP^m_X\ar_q[ru] &
}
$$
where $i$ (resp. $k$) is a closed immersion of codimension $r$
 (resp. $s$) and $p$ (resp. $q$) is the canonical projection.
Then, $k^*p^*=i^*q^*$.
\end{lm}
\begin{proof}
Let us introduce the following morphisms~:
$$
\xymatrix@C=22pt@R=8pt{
&& & {}\PP^n_X\ar^p[rd] & \\
Y\ar@/^5pt/^k[rrru]\ar@/_5pt/_i[rrrd]\ar|\nu[rr]
 && {}\PP^n_X \times_X {}\PP^m_X\ar_/8pt/{q'}[ru]\ar^/8pt/{p'}[rd]
 & & X. \\
&& & {}\PP^m_X\ar_q[ru] &
}
$$
Applying lemma \ref{lm:Gysin1},
 we are reduced to prove $k^*=\nu^* {q'}^*$ and $i^*=\nu^* {p'}^*$.
In other words, we are reduced to the case $m=0$ and $q=1_X$.

In this case, we introduce the following morphisms~:
$$
\xymatrix@R=5pt@C=7pt{
Y\ar@{=}[rrddd]\ar^/4pt/s[rrd]\ar@/^8pt/^k[rrrrd] && & \\
&& {}\PP^n_Y\ar_l[rr]\ar^q[dd] && {}\PP^n_X\ar^p[dd] \\
&& && \\
&& Y\ar^i[rr] && X
}
$$
Then the lemma follows from lemma \ref{lm:Gysin2}, 
 lemma \ref{lm:Gysin3} and corollary \ref{cor:assoc}.
\end{proof}

Consider smooth schemes $X$ and $Y$ and a projective morphism
$f:Y \rightarrow X$ of codimension $d$.
Consider an arbitrary factorization
 $Y \xrightarrow i \PP^n_X \xrightarrow p X$ of $f$
 into a closed immersion of codimension $d+n$ and
 the canonical projection.
The preceding lemma shows that the composite morphism
$$
\M(X) \xrightarrow{p^*} \M(\PP^n_X)(-n)[-2n]
 \xrightarrow{i^*} \M(Y)(d)[2d]
$$
is independent of the chosen factorization.
\begin{df} \label{df:Gysin}
Considering the above notations, we define 
the Gysin morphism associated to $f$ as the morphism
$$
f^*:=i^*p^*:\M(X) \rightarrow \M(Y)(d)[2d].
$$
\end{df}

\subsubsection{Properties}

\num \label{general_Gysin&product}
Let us first remark that, 
as a corollary of \ref{gysin&tensor},
we obtain~: $(f \times g)^*=f^* \otimes g^*$
for any projective morphisms $f$ and $g$.

\begin{prop} \label{functoriality4general_Gysin}
Consider projective morphisms 
$Z \xrightarrow g Y \xrightarrow f X$
between smooth schemes.

Then $g^*f^*=(fg)^*$.
\end{prop}
\begin{proof}
We choose a factorization
$Y \xrightarrow i \PP^n_X \xrightarrow p X$
(resp. $Z \xrightarrow j \PP^m_X \xrightarrow q X$)
 of $f$ (resp. $fg$) and we introduce the diagram
$$
\xymatrix@C=16pt@R=6pt{
&& {}\PP^m_X\ar^q@/^18pt/[rrddd] && \\
& & {}\PP^n_X \times_X {}\PP^m_X\ar^/5pt/{q'}[rd]\ar_{p'}[u] & & \\
& {}\PP^m_Y\ar^{q''}[rd]\ar^/-6pt/{i'}[ru] & & {}\PP^n_X\ar|p[rd] & \\
Z\ar|g[rr]\ar|k[ru]\ar^j@/^18pt/[rruuu] && Y\ar|f[rr]\ar^i[ru] && X.
}
$$
in which $p'$ is deduced from $p$ by base change and so on for $q'$ and $q''$.

Then, by using the factorizations given in the preceding diagram,
the proposition follows directly using \ref{lm:Gysin2}, \ref{lm:Gysin1},
\ref{cor:assoc} and finally \ref{lm:Gysin4}.
\end{proof}

\begin{prop} \label{compatibility_Residues/Gysin}
Consider a commutative square of smooth schemes
$$
\xymatrix@=10pt{
T\ar^k[r]\ar_g[d] & Y\ar^f[d] \\
Z\ar^i[r] & X
}
$$
such that $i$ is a closed immersion
 and $f$ is a projective morphism.
Let $h$ be the pullback of $f$ on $X-Z$.
Let $n$, $m$, $s$, $t$ be the respective codimension
 of $i$, $k$, $f$, $g$.
Note that $n+s=m+t$ and put $d=n+s$.
 
Then the following square is commutative~:
$$
\xymatrix@R=14pt@C=24pt{
\M(T)(d)[2d]\ar^/-10pt/{\partial_{Y,T}}[r]
 & \M(Y-T)(s)[2s+1] \\
\M(Z)(n)[2n]\ar^{\partial_{X,Z}}[r]\ar^{g^*}[u]
 & \M(X-Z)[1]\ar_{h^*}[u]
}
$$
\end{prop}
\begin{proof}
By construction of the Gysin morphism, we have only to consider
the case where $f$ is the projection of a projective bundle or 
a closed immersion. 
It follows from lemma \ref{lm:Gysin&product} in the first case
 and from theorem \ref{thm:assoc} in the second.
\end{proof}

\begin{rem}
Applying the two preceding propositions and case (i) of the following
proposition, we obtain that the Gysin triangle is functorial with respect
to the Gysin morphism of a projective morphism in the case of
a cartesian square as in the preceding statement.
\end{rem}

\begin{prop} \label{projection4general_Gysin}
Consider a cartesian square of smooth schemes
$$
\xymatrix@=10pt{
Y'\ar^g[r]\ar_q[d] & X'\ar^p[d] \\
Y\ar^f[r] & X
}
$$
such that $f$ (resp. $g$) is a projective morphism
of codimension $n$ (resp. $m$). Note that necessarily,
 $n \geq m$.
\begin{enumerate}
\item[(i)]
Suppose $n=m$ and $Y \times_X X'$ is smooth (\emph{i.e.} $Y'=Y \times_X X'$).

Then $f^*p_*=q_*g^*$.
\item[(ii)] Suppose $Y \times_X X'$ is smooth and $n>m$.
Put $e=n-m$. We attach to the above square a vector bundle $\xi$
of rank $e$ called the excess intersection bundle~:
choose a projective bundle $P/X$
 and a factorization $Y \xrightarrow i P \xrightarrow p X$
 of $f$ into a closed immersion followed by the canonical projection.
We obtain a canonical embedding $N_{Y'}(P \times_X X') \rightarrow q^* N_Y(P)$
and denote by $\xi$ the quotient bundle over $Y'$.
This definition is independent of the choice of the factorization as
shown in \cite{Ful}, proof of prop. 6.6.

Then, $f^*p_*=\big(q_* \gcup c_e(\xi)\big) \circ g^*$.
\end{enumerate}
\end{prop}
\begin{proof}
In each case,
 we reduce to the corresponding assertion for a closed immersion 
 (\ref{transversal}, \ref{excess} and \ref{ramification}) 
 by choosing a factorization of $f$ into a closed immersion
  followed by a projection and by considering its pullback on $X'$. 
Indeed, the assertion (i) when $f$ is the projection of a projective bundle
 is trivial.
\end{proof}

We obtain finally the analog of corollary \ref{second_proj_formula}~:
\begin{cor} \label{2ndproj_formula4general_Gysin}
Let $f:Y \rightarrow X$ be a projective morphism between smooth scheme
 of pure codimension $d$.

Then $(1_{Y*} \gcup f_*) \circ f^*=f^* \gcup 1_{X*}$
as a morphism $\M(X) \rightarrow \M(Y \times X)(d)[2d]$.
\end{cor}
The proof is the same as for \ref{second_proj_formula} using assertion (i)
 of the proposition above and the formula of \ref{general_Gysin&product}.

\num \label{hyp:ramif_general_Gysin}
We now consider the analog of the ramification formula 
\ref{ramification}.
Consider a commutative square of smooth schemes
$$
\xymatrix@=10pt{
T\ar^q[r]\ar_g[d]\ar@{}|\Delta[rd] & Y\ar^f[d] \\
Z\ar_p[r] & X
}
$$
which is cartesian on the underlying topological spaces
 and such that $p$ and $q$ are projective morphisms of codimension $n$. 
We assume $T$ admits an ample line bundle. \\
Put $T'=T \times_X Y$ and 
note the morphism $T \rightarrow T'$ induce by $\Delta$ is a thickening. 
Let $T'=\bigcup_{i \in I} T'_i$ 
 (resp. $T'_i=\bigcup_{j \in J_j} T'_{ij}$) be the decomposition into
  connected (resp. irreducible) components. 
Put $T_i=T'_i \times_{T'} T$ and $T_{ij}=T'_{ij} \times_{T'} T$.
We introduce the following condition on $\Delta$~:
\begin{enumerate}
\item[(Special)] For any $i \in I$, there exists an integer $r_i \geq 0$
such that for any $j \in J_i$, $m(T'_{ij})=r_i.m(T_{ij})$.
\end{enumerate}
In this case, the integer $r_i$ will be called the 
\emph{ramification index}
of $f$ along $T_i$.

Consider a factorization $Z \xrightarrow i P \xrightarrow \pi X$
of $p$ into a closed immersion and the projection of a projective bundle.
We put $Q=P \times_X Y$ and consider the obvious morphism of closed pairs
$(h,g):(Q,T) \rightarrow (P,Z)$. Of course, $\Delta$ satisfies (Special)
if and only if $(h,g)$ satisfies (Special). Moreover, for any $i \in I$,
the element $r(T_i;h,g)$ is independent of the chosen factorization.
Indeed, taking into account the compatibility of $F$-intersection multiplicity
with flat base change, this boils down to the following lemma~:
\begin{lm}
Consider a commutative diagram of smooth schemes
$$
\xymatrix@C=15pt@R=6pt{
& T'\ar[rrd]\ar_t[ld]\ar'[d]^/2pt/{g'}[dd] & & \\
T\ar[rrr]\ar_g[dd] & & & Y\ar^f[dd] \\
& Z'\ar_/-2pt/s[ld]\ar[rrd] & & \\
Z\ar[rrr] & & & X
}
$$
such that $T$ and $T'$ are connected and admits an ample line bundle,
 $t=s \times_Z T$ and $(f,g)$ (resp. $(f,g')$)
 is a morphism of smooth closed pairs satisfying condition
  (Special) with ramification index $r$. \\
Then, $r(T';f,g')=t^*r(T;f,g) \in H^{0,0}(T')$.
\end{lm}
\begin{proof}
Consider the blow-up $B=B_Z(\AA^1_X)$ (resp. $B'=B_{Z'}(\AA^1_X)$)
and its exceptional divisor $P$ 
 (resp. $P'$).
As $Z' \subset Z$,
 we get a cartesian transversal square, together with its pullback
  over $Y$
$$
\begin{array}{ccc}
\xymatrix@=10pt{
P'\ar[r]\ar[d] & B',\ar[d] \\
P\ar[r] & B.
} &
\text{pullback above $Y$ :}
&
\xymatrix@=10pt{
Q'\ar[r]\ar[d] & C'.\ar[d] \\
Q\ar[r] & C
}
\end{array}
$$
The second square is still transversal.
Put $L=N_QC|_Z$ and $L'=N_{Q'}(C')|_{Z'}$. Thus, $L'=L|_{Z'}$.
According to this equality, the lemma follows from the definition
of $F$-intersection multiplicities.
\end{proof}

\begin{df}
Consider the notations and hypothesis of \ref{hyp:ramif_general_Gysin},
 assuming the square $\Delta$ satisfies condition (Special).
For any $i \in I$, we define the $F$-intersection multiplicity $r(T_i;\Delta)$ 
 of $T_i$ in the pullback square $\Delta$ as the well defined element 
  $r(T_i;h,g)$ according to the notations above.
\end{df}

The following proposition is now a corollary of \ref{ramification}~:
\begin{prop}
Consider the hypothesis and notations of the preceding definition.
Put $g_i=g|_{T_i}$ and $q_i=q|_{T_i}$.

Then, $p^*f_*=\sum_{i \in I} \big(r(T_i;\Delta) \gcup g_{i*}\big)q_i^*$.
\end{prop}
 
\subsection{The duality pairing}

Let $X/S$ be a smooth projective scheme of pure dimension $n$.
Let $p:X \rightarrow S$ (resp. $\delta:X \rightarrow X \times X$)
 be its structural morphism (resp. its diagonal embedding).

Then we define morphisms
\begin{align*}
\mu_X: \ & \un \!\xrightarrow{p^*} \!\M(X)(-n)[-2n]
 \!\xrightarrow{\delta_*}\! \M(X \times X)(-n)[-2n]
  \!=\!\M(X)(-n)[-2n] \otimes \M(X) \\
\epsilon_X: \ & \M(X) \otimes \M(X)(-n)[-2n]
 \xrightarrow{\delta^*} \M(X) \xrightarrow{p_*} \un.
\end{align*}

The following result is now a formality~:
\begin{thm} \label{thm:duality}
Consider the notations above.

Then $\big(\M(X)(-n)[-2n],\mu_X,\epsilon_X\big)$
 is a strong dual of $\M(X)$.
\end{thm}
\begin{proof}
Each identity of definition \ref{df:strong_dual} 
is an easy application of \ref{general_Gysin&product},
\ref{projection4general_Gysin}(i)
 (the usual projection formula)
  and proposition \ref{functoriality4general_Gysin}.
\end{proof}

\num \label{explicit_duality_coh_hom}
\textit{Applications}~: Consider the notations of the previous
proposition and let $\E$ be a motive.
\begin{enumerate}
\item We define the fundamental class $\tau_X \in H_{2n,n}(X)$ of $X$
 as the element
$$
p^*:\un \rightarrow \M(X)(-n)[-2n].
$$
We also consider $\eta \in H^{2n,n}(X \times X)$ the fundamental class 
of the diagonal $\delta$.

Then the isomorphisms of \eqref{duality&adjunction} with $M=\M(X)$
 gives exactly, considering the definitions of cap-product
  and slant product (cf \ref{general_cup&slant}),
the following reciprocal isomorphisms~:
\begin{equation} \label{alexander_duality}
\begin{split}
\E^{r,p}(X) & \leftrightarrows \E_{2n-r,p-n}(X) \\
x & \mapsto x \ncap \tau_X \\
\eta/y & \mapsfrom y.
\end{split}
\end{equation}
This is the \emph{Poincar\'e duality isomorphism},
 as it appears in algebraic topology (cf \cite{Ada},
 \cite[14.41, 14.42]{Swi}). 
To the knowledge of the author, the first appearance
 of this precise form of duality in algebraic geometry
 is in \cite{PY}.
\item Suppose $\E$ is a ringed motive.
In this case, the regulator maps
\begin{align*}
\varphi_X:&H^{2n,n}(X) \rightarrow \E^{2n,n}(X) \\
\psi_X:&H_{2n,n}(X) \rightarrow \E_{2n,n}(X)
\end{align*}
allow to define the fundamental class of $X$ (resp. the fundamental
class of the diagonal)
with coefficients in $\E$ as the image $\psi_X(\tau_X)$ (resp. $\varphi_X(\eta)$)
 of the corresponding class with coefficients in $H$.
Moreover, we can obviously express the isomorphisms above
 with this classes (cf number (1) above),
  obtaining a Poincar\'e duality purely
   in terms of the cohomology theory $\E^{**}$.
\item Suppose $\E$ is a ringed motive.

The morphism
$$
p_*:\E^{**}(X) \rightarrow A
$$
induced by the Gysin morphism of $p$ is usually called the \emph{trace morphism}
(relative to $S$).

We suppose the cohomology $\E^{**}$ satisfies the following K\"unneth
 property~:
for any motives $M,N,P \in \{\un,\M(X),\M(X)(-n)[-2n]\}$,
the pairing
$$
\E^{**}(M) \otimes_A \E^{**}(N) \otimes_A \E^{**}(P)
 \rightarrow \E^{**}(M \otimes N \otimes P)
$$
is an isomorphism.

Then it follows formally that 
$$
\big(\E^{**}(\M(X)(-n)[-2n]),\E^{**}(\mu),\E^{**}(\epsilon)\big)
$$
is a strong dual of $\E^{**}(\M(X))$ in the category
of graded $A$-modules.

More concretely, the pairing (induced by $\E^{**}(\mu)$)
$$
\E^{**}(X) \otimes_A \E^{**}(X) \rightarrow A,
 x \otimes y \mapsto p_*(x \ncup y)
$$
is a perfect pairing of graded $A$-modules.
This is usually called the \emph{Poincar\'e duality pairing}\footnote{
This way of deducing the usual Poincar\'e duality from the
Alexander duality and the K\"unneth formula was explained
to me by D.C.Cisinski.}
for the cohomology theory $\E^{**}$.

Note it implies that $E^{**}(X)$ is a projective finitely
generated graded $A$-module (see \cite[1.4]{DP}).
\end{enumerate}

\begin{ex}
The conditions of point (3) are fulfilled when $X$ is a Grassmanian scheme over $S$,
 or more generally a cellular variety over $S$, without any assumption on $\E$.
In \cite{CD2}, we study cohomology theories $\E^{**}$ which satisfies
the K\"unneth formula.
\end{ex}

The Gysin morphism determine the duality pairing defined above.
Reciprocally, this duality determines the Gysin morphism as shown in the
next proposition.
\begin{prop} \label{Gysin=transpose}
Let $f:Y \rightarrow X$ be a morphism between smooth projective $S$-schemes.
Suppose $X$ (resp. $Y$) is of constant relative dimension $n$ (resp. $m$)
over $S$.

Then
$$
f^*=\tra(f_*)(-n)[-2n]
$$
where the transpose morphism on the right hand side is taken with respect
to the strong duals of $\M(X)$ and $\M(Y)$ obtained in
the previous theorem.
\end{prop}
\begin{proof}
Consider the structural projections $p:X \rightarrow S$,
 $q:Y \rightarrow S$ and the diagonal embeddings
  $\delta_X:X \rightarrow X \times X$,
   $\delta_Y:Y \rightarrow Y \times Y$.
Let $n$ be the dimension of $X$
Put $\M(X)^\vee=\M(X)(-n)[-2n]$ and $\M(Y)^\vee=\M(Y)(-m)[-2m]$.

According to the first point which follows definition \ref{df:transpose},
we have to prove the following square is commutative~:
$$
\xymatrix@R=12pt@C=20pt{
\M(Y) \otimes \M(X)^\vee\ar^{f_* \otimes 1}[r]\ar_{1 \otimes f^*}[d]
 & \M(X) \otimes \M(X)^\vee\ar^{p_*\delta_X^*}[d] \\
\M(Y) \otimes \M(Y)^\vee\ar^/16pt/{q_*\delta_Y^*}[r] & {}\un.
}
$$
We introduce the following cartesian square~:
$$
\xymatrix@R=10pt@C=16pt{
Y\ar^/-6pt/\gamma[r]\ar_f[d] & Y \times X\ar^{f \times 1}[d] \\
X\ar^/-5pt/{\delta_X}[r] & X \times X
}
$$
Note that $f_* \otimes 1=(f \times 1)_*$ and $1 \otimes f^*=(1 \times f)^*$
(cf \ref{general_Gysin&product}). The result follows from the computation~:
$$
p_*\delta_X^*(f \times 1)_*=p_*f_*\gamma^*=q_*\delta_Y^*(1 \times f)^*
$$
which uses \ref{projection4general_Gysin}(i)
 and \ref{functoriality4general_Gysin}.
\end{proof}

\subsection{Two illustrations}

\num \label{cobordism_classes} \textit{Cobordism classes}.---
\begin{df}
Let $X$ be a smooth projective scheme of pure dimension $n$.
Let $p:X \rightarrow S$ be its structural projection.

We define the cobordism class of $X/S$ as the element of $A$,
of (cohomological) degree $(-2n,-n)$,
$$
[X]=\un \xrightarrow{p^*} \M(X)(-n)[-2n]
 \xrightarrow{p_*} \un(-n)[-2n].
$$
\end{df}
In other words, $[X]=p_*(1)$ as a cohomological class.
Note that according to definition \ref{gysin4projection}
 and what follows it, we obtain that $[\PP^n]=\eta'_n$.
Note also that $[X \sqcup Y]=[X]+[Y]$ (from axiom (Add))
 and $[X \times_S Y]=[X] \ncup [Y]$
  (from \ref{general_Gysin&product}).

\begin{ex}
Consider a factorization $X \xrightarrow i \PP^N \xrightarrow \pi S$
of the morphism $p$ into a closed immersion followed by the canonical
projection. Let $c=N-n$ be the codimension of $i$.
Let $\eta_{\PP^N}(X) \in H^{2c,c}(\PP^N)$ be the fundamental class
 associated to the embedding $i$. 
Then from corollary \ref{cor:assoc},
 $[X]=p_*(\eta_{\PP^N}(X))$ (as a cohomological class). \\
Thus, to compute this cobordism class, 
we can use the projective bundle theorem, 
which implies we can write
$\eta_{\PP^N}(X)=\sum_{i=0}^N x_i.c^i$
where $c$ is the Chern class of the dual canonical line bundle,
 and $x_i$ is an element of $A$. Then,
$$
[X]=\sum_{i=0}^N x_i \ncup [\PP^{N-i}]
$$
as $p_*(c^i)=[\PP^{N-i}]$ according to definition \ref{gysin4projection}.
\end{ex}

We want to compute now the cobordism class $[\PP^n]$.
Let $\delta:\PP^n \rightarrow \PP^n \times \PP^n$ be the diagonal
embedding. According to definition \ref{gysin4projection},
we have to compute 
\begin{equation} \label{fundamental_class_diag_pdl}
\delta_*(1)=\sum_{0 \leq i,j \leq n} \eta_{i+j-n}.c^i \ncup d^j
\end{equation}
with the notations preceding lemma \ref{pre_comput_cobordism_c}.

Let $p_1,p_2:\PP^n \times \PP^n \rightarrow \PP^n$ be the
 projections respectively on the first and second factor.
Let $\L$ (resp. $\Q$) be the canonical line bundle 
(resp. quotient bundle) 
 on $\PP^n$. Consider the vector bundles 
  $\L^{(i)}=p_i^{-1}(\lambda)$ and $\Q^{(i)}=p_i^{-1}(\Q)$
   for $i=1,2$.
In the preceding expression, $c=c_1(\L_1^\vee)$ and $d=c_1(\L_2^\vee)$.
Put $E=\uHom(\L_1,\Q_2)=\L_1^\vee \otimes \Q_2$.
We get a section $s$ of this vector bundle considering the canonical
morphism
$$
\L_1 \rightarrow \AA^{n+1} \times \PP^n=\PP^n \times \AA^{n+1}
 \rightarrow \Q_2.
$$
It is well known (see \cite{PPS}) that $\delta(\PP^n)$
is the subscheme defined by $s=0$.
Thus according to corollary \ref{fundamental_class_as_chern_class},
$\delta_*(1)=c_n(E)=c_n(\L_1^\vee \otimes \Q_2)$.
From this expression, we obtain easily~:
\begin{enumerate}
\item \emph{Additive case}~: When the formal group law is additive\footnote{
 This is the case for the category $DM(S)$.}
 (\emph{i.e.} $F(x,y)=x+y$), according to a well known formula
  (cf \cite[ex. 3.2.2]{Ful}),
$$c_n(\L_1 \otimes \Q_2)=\sum_{i=0}^n c_1(\L_1)^i \ncup c_{n-i}(\Q_2)
=\sum_{i=0}^n c^i \ncup d^{n-i}.
$$
Thus, $[\PP^n]=0$ if $n>0$.
\item \emph{Case $n=1$}~: As $c^2=d^2=0$, we simply obtain~: \\
\indent $c_1(\L_1 \otimes \Q_2)=F(c_1(\L_1),c_1(\Q_2))=c+d+a_{1,1}.c \ncup d$. \\
Thus $\eta_1=a_{1,1}$ which implies
$[\PP^1]=-a_{1,1}$.
\end{enumerate}

In the general case, we obtain the following computation~:
\begin{prop} \label{fdl_class_diagonal}
With the notations introduced above,
$$
\delta_*(1)=\sum_{0 \leq i,j \leq n} a_{1,i+j-n}.c^i \ncup d^j.
$$
\end{prop}
\begin{proof}
Consider the ind-scheme $\PP^\infty \times \PP^n$ and
the embedding $\tau:\PP^n \times \PP^n \rightarrow \PP^\infty \times \PP^n$.
Let $\tilde p_1$ (resp. $p_2$) be the projection on the first (resp. second)
factor of $\PP^\infty \times \PP^n$. Put 
$\tilde \L_1=\tilde p_1^{-1}(\L)$, $\L_2=p_2^{-1}(\L)$ and $\Q_2=p_2^{-1}(\Q)$.
Thus, with a little abuse of notation, 
 $c_n(\L_1 \otimes \Q_2)=\tau^* c_n\left(\tilde \L_1 \otimes \Q_2\right)$.
 
According to the definition of $\Q$,
 we consider the short exact sequence ~:
$$
0 \rightarrow \tilde \L_1^\vee \otimes \L_2
 \rightarrow \tilde \L_1^\vee \otimes \AA^{n+1}
  \rightarrow \tilde \L_1^\vee \otimes \Q_2 \rightarrow 0.
$$
From the Whitney sum formula \ref{additivity_Chern},
 we thus obtain the relation~:
$$
c_{n+1}(\tilde \L_1^\vee \otimes \AA^{n+1})
 =c_1(\tilde \L_1^\vee \otimes \L_2) \ncup c_n(\tilde \L_1^\vee \otimes \Q_2).
$$
Put $\tilde c=c_1(\tilde \L_1)$, $d=c_1(\L_2)$ as cohomology classes
in $B=H^{**}(\PP^\infty \times \PP^n)$. Note moreover the
$A$-algebra $B$ is equal to $(A[d]/{d^{n+1}})[[\tilde c]]$.
In terms of the fundamental group law $F$ and its inverse power series $m$,
the preceding relation reads\footnote{This expression for computing 
$\delta_*(1)$ was also obtained in \cite{Panin}.}
$\tilde c^{n+1}=F(\tilde c,m(d)) \ncup c_n(\tilde \L_1^\vee \otimes \Q_2)$.

We have to prove~:
$$
c_n(\tilde \L_1^\vee \otimes \Q_2)=
 \sum_{0 \leq i,j \leq n} a_{1,i+j-n}.\tilde c^i \ncup d^j \mod \tilde c^{n+1}.
$$

Let $m(x)=\sum_{k>0} m_k.x^k$ (thus $m_1=-1, m_2=a_{1,1}$, etc).
For any integers $l,s$, we put
$$
M_{l,s}=\sum_{\stackrel{k_1+...+k_l=s}{k_1,...,k_l>0}} m_{k_1}...m_{k_l}
$$
when $(l,s) \neq (0,0)$, and $M_{0,0}=1$.
Thus, $F(\tilde c,m(d))=\sum_{k,l,s} a_{k,l}M_{l,s}.\tilde c^k \ncup d^s$.
In particular, $F(\tilde c,m(d))=u.\tilde c+v$ where $u$ is invertible in $B$
 and $v$ is nilpotent. This implies $F(\tilde c,m(d))$ is a non zero divisor in $B$
  and we are reduced to prove~:
$$
F(\tilde c,m(d)) \ncup \sum_{0 \leq i,j \leq n} a_{1,i+j-n}.\tilde c^i \ncup d^j
 =0 \mod \tilde c^{n+1}.
$$
The left hand side can be expanded (modulo $\tilde c^{n+1}$) as the sum~:
$$
\sum_{0 \leq u,v \leq n}
 \left(\sum_{k,l,s} a_{k,l} M_{l,s} a_{1,u+v-n-k-s}\right).\tilde c^u \ncup d^v.
$$
Finally, for any integers $u,v \in [0,n]$, the coefficient of $\tilde c^u \ncup d^v$
in the preceding sum can be written
$$
\sum_w \left(\sum_{k,l} a_{k,l}M_{l,w-k}\right)a_{1,u+v-n-w}.
$$
This is zero according to the relation $F(x,m(x))=0$.
\end{proof}

From definition \ref{gysin4projection}, the previous proposition
reads $\eta_i=a_{1,i}$.
As a corollary (cf relation \eqref{relation_cobordism}),
 we recover the classical Myschenko theorem together with a nice expression
  of $[\PP^n]$ as a determinant~:
\begin{cor} \label{cor:Myschenko}
\begin{enumerate}
\item For any integer $n \geq 0$,
$$
[\PP^n]=(-1)^n.\det\left(
\raisebox{0.5\depth}{
\xymatrix@=0.1ex{
0\ar@{.}[rr]\ar@{.}[dd]
 & & 0\ar@{.}[lldd] & 1\ar@{-}[lllddd]
 & a_{1,1}\ar@{-}[lllldddd] \\
&&&& a_{1,2}\ar@{-}[lllddd]\ar@{.}[ddd] \\
0 & & & & \\
1 & & & & \\
a_{1,1} & a_{1,2}\ar@{.}[rrr] & & & a_{1,n}
}
}
\right).
$$
\item For any integer $n>0$, 
 $\sum_{0 \leq i \leq n} a_{1,i}.[\PP^{n-i}]=0$.
\end{enumerate}
\end{cor}
The usual formulation of the relations given in (2) uses
 the series $p(x)=\sum_i [\PP^i].x^i$
  and $\omega(x)=\frac{\partial F}{\partial y}(x,0)$.
It reads $p(x)=\omega(x)^{-1}$.

\begin{rem}
An interesting problem is to extend this computation 
 to the case of an arbitrary projective bundle $\PP(E)$.
We hope the fundamental class $\eta_{\PP(E) \times \PP(E)}(\PP(E))$
as an explicit description in terms of the coefficients $a_{1,i}$
and the Chern classes of $E$ which would give an expression
of $[\PP(E)]$ as a determinant analog of the above.
This will give a counter-part to a classical formula of Quillen.
\end{rem}

\num \label{blow-up_fml} \textit{Blow-up formulas}.---

\begin{prop} \label{lm1:proj_bdl_fml}
Let $(X,Z)$ be a smooth closed pair and $B$ be the blow-up
of $X$ with center $Z$. Let $f:B \rightarrow X$ be the canonical
projection.

Then, $f_*f^*=1.$
\end{prop}
\begin{proof}
Let $s_1$ (resp. $s_0$, $\pi$) be the unit section 
(resp. zero section, canonical projection) 
of $\AA^1_X/X$. 
Let $B'$ be the blow-up of $\AA^1_X$ with center $0 \times Z$.
We consider the following cartesian square~:
$$
\xymatrix{
X\ar^{\bar \sigma_1}[r]\ar@{=}[d] & B'\ar^{f'}[d] \\
X\ar^{s_1}[r] & {}\AA^1_X.
}
$$
From the projection formula \ref{projection4general_Gysin}(i),
we obtain $f'^*s_{1*}=\bar \sigma_{1*}$ which implies
$f'^*=\bar \sigma_{1*} \pi_*$ by the axiom (Htp).

Thus we deduce easily~: $f'_*f'^*=f'_*\bar \sigma_{1*} \pi_*=s_{1*} \pi_*=1$.

Finally we consider the cartesian diagram~:
$$
\xymatrix@=10pt{
B\ar^f[r]\ar_\nu[d] & X\ar^{s_0}[d] \\
B'\ar^{f'}[r] & {}\AA^1_X
}
$$
Using once again the projection formula \emph{loc. cit.} we get~:
$f'_*f'^*s_{0*}=s_{0*}f_*f^*$. This concludes using axiom (Htp).
\end{proof}

\begin{lm} \label{lm2:proj_bdl_fml}
Let $P/X$ be a projective bundle over a smooth scheme $X$
 of pure dimension $d$.
Let $\Q$ be the canonical quotient bundle of $P/X$
 and put $e=c_d(\Q)$ seen as a morphism $\M(P) \rightarrow \un(d)[2d]$.

Then, $(p_* \gcup e) \circ p^*$ is an isomorphism.
\end{lm}
\begin{proof}
Using the projection formula \ref{2ndproj_formula4general_Gysin},
 we have to prove that the cohomological class $p_*(e) \in H^{00}(X)$ is a unit.
By compatibility of this class with base change
 and invariance under isomorphisms of projective bundles,
  we reduce to the case of $P=\PP^d_S$. 
Let $s:S \rightarrow \PP^d_S$ be the canonical section.
Then, $e=s_*(1)$ (cf remark \ref{Thom&quotient} combined
with example \ref{ex:Gysin&Thom}). Thus, following lemma
\ref{lm:Gysin3}, $p_*(e)=1$.
\end{proof}

\begin{rem}
In the case of an additive formal group law, we can easily see that $p_*(e)=1$
 for any projective bundle $P/X$ which implies the composite 
 isomorphism of the lemma is just the identity.
\end{rem}

\num \label{notations_proj_bdl_fml}
Let $X$ be a smooth scheme, $Z$ be a smooth closed subscheme of $X$
of pure codimension $n$. 
Let $B$ be the blow-up of $X$ with center $Z$
and 
$P$ be the exceptional divisor. Consider the cartesian squares~:
$$
\xymatrix@R=14pt@C=24pt{
P\ar^k[r]\ar_p[d] & B\ar^f[d] & B-P\ar_/3pt/l[l]\ar^h[d]\\
Z\ar^i[r] & X & X-Z.\ar_/3pt/j[l]
}
$$
We let $\L$ (resp. $\Q=p^{-1}(N_ZX)/\L$) be the canonical line bundle
(resp. quotient bundle) on $P=\PP(N_ZX)$ and we put~: 
$e=c_{n-1}(\Q)$.
\begin{prop} \label{proj_bdl_fml1}
Using the notations above,
 the short sequence
$$
0 \rightarrow \M(P)
 \xrightarrow{
\text{\scriptsize $\begin{pmatrix} p_* \\ k_* \end{pmatrix}$}} \M(Z) \oplus \M(B)
  \xrightarrow{(-i_*,f_*)} \M(X) \rightarrow 0
$$
is split exact with splitting 
{\tiny $\begin{pmatrix} 0 \\ f^* \end{pmatrix}$}.

By abuse of notation,
 we denote by $\M(P/Z)$ the kernel\footnote{If we had a splitting
  $s:Z \rightarrow P$ of $p$, this will be the motive associated
   to the immersion $s$.}
of the split monomorphism $p_*$
 and let $\tilde k_*:\M(P/Z) \rightarrow \M(B)$
  be the morphism induced by $k_*$.
Then, we obtain an isomorphism
$$
\M(P/Z) \oplus \M(X) \xrightarrow{\left(\tilde k_*,f^*\right)} \M(B).
$$
\end{prop}
\begin{proof}
The previous short sequence is obviously a complex.
The fact {\tiny $\begin{pmatrix} 0 \\ f^* \end{pmatrix}$} is a splitting
 is proposition \ref{lm1:proj_bdl_fml}.
 
We directly prove the last assertion of the proposition
 which then concludes.
Consider the following diagram~:
$$
\xymatrix@C=39pt@R=35pt{
\M(X-Z)
 \ar^-{\text{\scriptsize $\begin{pmatrix} 0 \\ j_* \end{pmatrix}$}}[r]
 \ar_{h^*}[d]
 \ar@{}|{(1)}[rd]
 & *{{\!\begin{array}{c}\M(P/Z) \\ \oplus \\ \M(X)\end{array}}\!}
 \ar[r]^-{\text{\scriptsize $\begin{pmatrix} 1\!\!\!\! & 0 \\ 0\!\!\!\! & i^* \end{pmatrix}$}}
 \ar^-{(\tilde k_*,f^*)}[d]
 \ar@{}|{(2)}[rd]
 & {\!\!\begin{array}{c}\M(P/Z) \\ \oplus \\ \M(Z)(n)[2n]\end{array}\!\!}
 \ar[r]^-{(0,\partial_{X,Z})}\ar^-{(k^*\tilde k_*,p^*)}[d]
 \ar@{}|{(3)}[rd]
 & M(X-Z)[1]\ar^-{h^*}[d] \\
\M(B-P)\ar_-{l_*}[r]
 & \M(B)\ar_-{k^*}[r]
 & \M(P)(1)[2]\ar_-{\partial_{B,P}}[r]
 & \M(B-P)[1]
}
$$
The two horizontal lines are distinguished triangles.
It is commutative~: 
for square $(1)$, 
 use the projection formula \ref{projection4general_Gysin}(i),
for square $(2)$, the functoriality of the Gysin morphism
\ref{functoriality4general_Gysin}, for square $(3)$,
the compatibility of residues and Gysin morphism \ref{compatibility_Residues/Gysin}
 and the defining property of the residue $\partial_{B,P}$. \\
As $h$ is an isomorphism,
 we are reduced to prove $(k^*\tilde k_*,p^*)$ is an isomorphism.
 
The normal bundle of $k:P \rightarrow B$ is the canonical line bundle $\L$.
Thus, from the self-intersection formula \ref{self_intersection},
 $k^* k_*=1_{P*} \gcup c$ with $c=c_1(\L)$.
The remaining assertion is local in $X$ so that we can assume
that $N_ZX$ is trivializable.
Finally, we compute easily the matrix of 
$$
(k^*k_*,p^*):\oplus_{i=0}^{n-1} M(Z)(i)[2i] \oplus M(Z)(n)[2n]
 \rightarrow \oplus_{i=0}^{n-1} M(Z)(i+1)[2i+1]
$$
obtained through the projective bundle isomorphism \ref{th:projbdl}~:
$$
\left(
\raisebox{0.5\depth}{
\xymatrix@=0.2ex{
0\ar@{.}[ddddd]\ar@{-}[rrrrrrddddd] & 1\ar@{-}[rrrrrdddd] & 0\ar@{.}[rrrr]\ar@{-}[rrrrddd] &&&& 0\ar@{.}[ddd] & [\PP^n] \\
&&&&&&& [\PP^{n-1}]\ar@{.}[dddd] \\
&&&&&&& \\
&&&&&& 0 & \\
&&&&&& 1 & \\
0\ar@{.}[rrrrrr] &&&&&& 0 & 1
}
}
\right).$$
As the matrix of $(k^*\tilde k_*,p^*)$ is obtained from the above one removing
the first column, it is obviously invertible.
\end{proof}

\begin{prop} \label{proj_bdl_fml2}
Consider the notations \ref{notations_proj_bdl_fml}.
The short sequence
$$
0 \rightarrow \M(B)
 \xrightarrow{
\text{\scriptsize $\begin{pmatrix} k^* \\ f_* \end{pmatrix}$}}
 \M(P)(1)[2] \oplus \M(X)
  \xrightarrow{(p_* \gcup e,-i^*)} \M(Z)(n)[2n] \rightarrow 0
$$
is split exact with pseudo-splitting
 {\tiny $\begin{pmatrix} p^* \\ 0 \end{pmatrix}$}.

Let $C$ be the cokernel of the split mono
 $p^*:\M(Z)(n-1)[2n-2] \rightarrow \M(P)$ and
  $\tilde k^*:\M(B) \rightarrow C(1)[2]$ the morphism induced by $k^*$.
Then the following morphism is an isomorphism~:
$$
\M(B) \xrightarrow{
 \text{\scriptsize $\begin{pmatrix} \tilde k^* \\ f_* \end{pmatrix}$}}
  C(1)[2] \oplus \M(X).
$$
\end{prop}
\begin{rem}
This second blow-up formula is a generalization of \cite[6.7(a)]{Ful}.
In case $X$ and $Z$ are projective smooth,
 it is simply the dual statement of the previous proposition
  using \ref{Gysin=transpose}.
More precisely, from \ref{proj_bdl_fml1} (resp. \ref{proj_bdl_fml2})
the morphism 
$$
\text{\scriptsize $\begin{pmatrix} k_* & p_* \\ f^* & 0 \end{pmatrix}$}
\text{ (resp. \scriptsize $\begin{pmatrix} k^* & f_* \\ p^* & 0 \end{pmatrix}$)}
$$
is an isomorphism. These two matrices are dual.
\end{rem}
\begin{proof}
The above sequence is a complex from the excess intersection formula
\ref{excess} applied to the morphism $(f,p)$. The pseudo-splitting of this 
sequence is exactly lemma \ref{lm2:proj_bdl_fml}.
We thus are reduced to the last assertion. \\
Let $\pi:\M(P)(1)[2] \rightarrow C(1)[2]$ be the canonical projection.
Consider the following diagram~:
$$
\xymatrix@C=37pt@R=30pt{
\M(B-P)\ar^{l_*}[r] \ar@{}|{(1)}[rd]
 \ar_{h^*}[d]
 & \M(B)\ar^{k^*}[r] \ar@{}|{(2)}[rd]
 \ar_-{\text{\scriptsize $\begin{pmatrix} \tilde k^* \\ f_* \end{pmatrix}$}}[d]
 & \M(P)(1)[2]\ar^{\partial_{B,P}}[r] \ar@{}|{(3)}[rd]
 \ar^-{\text{\scriptsize $\begin{pmatrix} \pi \\ p_* \gcup e \end{pmatrix}$}}[d]
 & \M(B-P)[1]\ar^{h^*}[d] \\
\M(X-Z)
 \ar_-{\text{\scriptsize $\begin{pmatrix} 0 \\ j_* \end{pmatrix}$}}[r]
 & {\!\begin{array}{c}C(1)[2] \\ \oplus \\ \M(X)\end{array}\!}
 \ar_-{\text{\scriptsize $\begin{pmatrix} 1\!\!\!\! & 0 \\ 0\!\!\!\! & i^* \end{pmatrix}$}}[r]
 & {\!\!\begin{array}{c}C(1)[2] \\ \oplus \\ \M(Z)(n)[2n]\end{array}\!\!}
 \ar_-{(0,\partial_{X,Z})}[r]
 & M(X-Z)[1]
}
$$
The horizontal lines are distinguished triangles. The diagram is commutative~:
$(1)$ follows from definitions,
 $(2)$ is a consequence of the excess intersection formula \ref{excess} for $(f,p)$
  and $(3)$ is a consequence of the same formula, considered for residues.
Finally, we are reduced to prove that 
$\text{\scriptsize $\begin{pmatrix} \pi \\ p_* \gcup e \end{pmatrix}$}$
is an isomorphism. But $\mathrm{coKer}(p^*) \simeq \mathrm{Ker(p_* \gcup e)}$
by a canonical isomorphism so that the latter morphism is simply the 
decomposition isomorphism associated to the split epimorphism $p_* \gcup e$.
\end{proof}